\documentclass{amsart}

\usepackage{graphicx} 
\usepackage{amsfonts, amsmath,amssymb,amsthm}
\usepackage{thmtools}
\usepackage{xcolor}
\usepackage[shortlabels]{enumitem}
\usepackage[hypertexnames=false]{hyperref}
\usepackage{cleveref}
\AddToHook{cmd/appendix/before}{\crefalias{section}{appendix}}

\crefname{enumi}{}{}

\usepackage{tikz}
\usepackage{graphicx}
\usepackage[normalem]{ulem}
\usepackage{fontawesome5}
\usepackage{mathrsfs}
\usepackage{longtable}
\usepackage{array}

\crefformat{enumi}{(\mathrm{#2#1#3})}
\Crefformat{enumi}{(\mathrm{#2#1#3})}
\crefname{enumi}{}{ }

\usepackage[margin=1 in]{geometry}

\newcommand{\rk}{\mathrm{rk}}

\newcommand{\dual}[1]{#1^{\vee}}
\newcommand{\GHV}{the GHV lower bound}

\newcommand{\KK}{\mathbb{K}}

\newcommand{\RR}{\mathbb{R}}

\newcommand{\ZZ}{\mathbb{Z}}
\newcommand{\NN}{\mathbb{N}}

\newcommand{\cB}{\mathcal{B}}
\newcommand{\C}{{\mathcal{C}}}

\newcommand{\cI}{\mathcal{I}}

\newcommand{\cL}{\mathcal{L}}

\newcommand{\conv}{\mathrm{conv}}
\newcommand{\cone}{\mathrm{cone}}
\newcommand{\NP}{\mathrm{NP}}
\newcommand{\SP}{\mathrm{SP}}

\newcommand{\loops}{\mathrm{loops}}

\theoremstyle{plain}
\newtheorem{Theorem}{Theorem}[section]
\newtheorem{Lemma}[Theorem]{Lemma}
\newtheorem{Proposition}[Theorem]{Proposition}

\newtheorem{Corollary}[Theorem]{Corollary}
\newtheorem{Conjecture}[Theorem]{Conjecture}

\newtheorem*{Main1}{{\Cref{thm:asym-res-facet}}}
\newtheorem*{Main2}{{\Cref{thm:asym-order}}}
\newtheorem*{Main3}{{\Cref{thm:relaxation-to-uniform}}}
\newtheorem*{Main4}{\Cref{thm:asym-res-dual-paving}}
\newtheorem*{Main5}{\Cref{thm:asym-res-PMD}}
\newtheorem*{Main6}{\Cref{thm:asym-res-Steiner-systems}}
\theoremstyle{definition}
\newtheorem{Definition}[Theorem]{Definition}
\newtheorem{Example}[Theorem]{Example}
\newtheorem{Question}[Theorem]{Question}

\newtheorem{Remark}[Theorem]{Remark}

\providecommand{\customgenericname}{}
\newtheorem{innercustomgeneric}{\customgenericname}
\providecommand{\customgenericname}{}
\newcommand{\newcustomtheorem}[2]{%
  \newenvironment{#1}[1]
  {%
   \renewcommand\customgenericname{#2}%
   \renewcommand\theinnercustomgeneric{##1}%
   \innercustomgeneric
  }
  {\endinnercustomgeneric}
}
\def\urltilda{\kern -.15em\lower .7ex\hbox{\~{}}\kern .04em} 
\newcustomtheorem{customthm}{Theorem}
\newcustomtheorem{acknowledgement}{Acknowledgement}

\newcommand{\cir}{\mathcal{C}}

\begin{document}

\title{Asymptotic Resurgence of Facet and Stanley-Reisner ideals of Matroids }

\author{
Michael DiPasquale$^{\dagger}$ \and
Louiza Fouli $^{*,\ddagger}$ \and
Arvind Kumar$^{\S}$
}



\begin{abstract}
Matroid configurations -- introduced by Geramita, Harbourne, Migliore, and Nagel -- are projective varieties which generalize so-called \textit{star configurations} and whose defining ideals are obtained by appropriately specializing the Stanley-Reisner ideal of a matroid.  Motivated by this connection, we study the asymptotic resurgence of the Stanley-Reisner ideals of matroids.  A result of Villarreal shows that it is equivalent to study the asymptotic resurgence of facet ideals.  We prove a formula for the asymptotic resurgence of the facet ideal of a matroid in terms of the Waldschmidt constant of facet ideals of the contractions of the matroid.  As a consequence, we show that asymptotic resurgence respects the weak order on matroids of the same rank. Therefore, the asymptotic resurgence of the facet ideal of a given matroid is bounded above by the asymptotic resurgence of the facet ideal of a so-called \textit{almost-uniform} matroid of the same rank, which we compute explicitly.
  
Guardo, Harbourne, and Van Tuyl showed that the asymptotic resurgence of an ideal is bounded below by the ratio of the initial degree of the ideal by its Waldschmidt constant.  We prove that this lower bound is an equality for facet ideals of many classes of matroids, including matroids of rank $k$ on a ground set of size $n\ge 2k$ whose dual is paving, perfect matroid designs, and sparse paving matroids arising from Steiner systems.  For the latter two classes, we explicitly compute the asymptotic resurgence.
\end{abstract}

\maketitle
\begingroup
\renewcommand\thefootnote{\fnsymbol{footnote}}

\footnotetext[1]{Corresponding author}
\footnotetext[2]{Department of Mathematical Sciences, 
New Mexico State University,
E-mail: midipasq@nmsu.edu}

\footnotetext[3]{Department of Mathematical Sciences, 
New Mexico State University,
E-mail: lfouli@nmsu.edu}
\footnotetext[4]{Department of Mathematics, Indian Institute of Technology Palakkad, India,
E-mail: arvkumar@iitpkd.ac.in}

\endgroup
\section{Introduction}
Suppose $I$ is a homogeneous ideal in the polynomial ring $S=\KK[x_1,\ldots,x_n]$, where $\KK$ is a field.  If $r,s\in\NN$, two notions of taking powers of $I$ are the \textit{ordinary powers} $I^r$ and the \textit{symbolic powers} $I^{(s)}:=\bigcap_{P\in\mbox{Ass}(I)} (I^sS_P)\cap S$.  The ordinary powers are more natural from an algebraic point of view; the generators of $I^r$ are just the products of $r$ generators of $I$.  The symbolic powers are more natural from a geometric point of view; if $I$ is the ideal of homogeneous polynomials vanishing on a projective variety $X\subset\mathbb{P}^{n-1}_\KK$ then $I^{(s)}$ consists of all homogeneous polynomials that vanish to order $s$ along $X$ by the Zariski-Nagata theorem.  There is a great deal of literature on symbolic powers of ideals -- we direct the reader to the survey~\cite{DDsGHN}.

One of the primary ways to compare symbolic and ordinary powers of an ideal is captured by the celebrated \textit{Containment Problem}, which deals with determining the pairs $\{(s,r)\in\NN^2~:~ I^{(s)}\subseteq I^r\}$.  Seminal containment results of Ein-Lazarsfeld-Smith~\cite{ELS01}, Hochster-Huneke~\cite{HH02}, and Ma-Schwede~\cite{MS17} establish that if $I$ is a radical ideal of an excellent regular ring $R$, then $I^{(hr)}\subset I^r$ for all $r\ge 1$, where $h$ is the maximum height of an associated prime of $I$ (called the \textit{big height} of $I$).

If $I$ is an ideal with big height $h$, it is often the case that $I^{(qr)}\subset I^r$ for $q<h$.  The `best' slope one can put in for $q$ is quantified by the \textit{resurgence} of $I$, denoted by $\rho(I)$, and defined by Bocci and Harbourne~\cite{BH10} as
\[
\rho(I):=\sup\left\lbrace\frac{s}{r}~:~ s,r \in \NN \text{ and } I^{(s)}\not\subset I^r\right\rbrace.
\]
The containment results above~\cite{ELS01,HH02,MS17} imply that $\rho(I)\le h$.

The topic of this paper is a variant of resurgence called \textit{asymptotic resurgence} introduced by Guardo-Harbourne-Van Tuyl~\cite{GHV13}.  The asymptotic resurgence of an ideal $I$ is denoted $\widehat{\rho}(I)$ and defined as
\[
\widehat{\rho}(I):=\sup\left\lbrace\frac{s}{r}~:~ s,r \in \NN \text{ and } I^{(st)}\not\subset I^{rt} \mbox{ for all } t\gg 0\right\rbrace. 
\]
It follows from the definitions that $\widehat{\rho}(I)\le \rho(I)$.  It can happen that $\widehat{\rho}(I)<\rho(I)$, although these examples are somewhat rare~\cite{DHNSST2015,BDHHSS2019,DFMS19,JKM22}.  Both $\rho(I)$ and $\widehat{\rho}(I)$ are extremely difficult to compute in general.  If $I$ is a monomial ideal, one can extract from~\cite{DFMS19} an algorithm to compute $\widehat{\rho}(I)$ -- see \Cref{app:algo}.  Another algorithm appears in~\cite{Villarreal-2023}.  To our knowledge, no algorithm exists for resurgence, even for monomial ideals.

In this paper, we address the problem of computing the asymptotic resurgence for the facet and Stanley-Reisner ideals of the independence complex of a matroid (henceforth we will write simply `the facet ideal of a matroid' or `the Stanley-Reisner ideal of a matroid).  If $M$ is a matroid with independence complex $\Delta(M)$, the \textit{facet ideal} $I(M)$ is the equigenerated squarefree monomial ideal with generators corresponding to the \textit{bases} of $M$, while the \textit{Stanley-Reisner} ideal $I_{\Delta(M)}$ is the squarefree monomial ideal with generators corresponding to \textit{circuits} of $M$.  The significance of symbolic powers in the context of matroids stems largely from a result due independently to Minh-Trung~\cite{MT-2011} and Varbaro~\cite{Varbaro-2011}, which states that a simplicial complex is the independence complex of a matroid if and only if all the symbolic powers of its Stanley-Reisner ideal are Cohen-Macaulay.  The Cohen-Macaulay property of symbolic powers of Stanley-Reisner ideals of matroids is crucial to the notion of \textit{matroid configurations}, introduced by Geramita-Harbourne-Migliore-Nagel~\cite{GHMN17}.  Matroid configurations are subvarieties of projective space obtained from certain specializations of Stanley-Reisner ideals of matroids; they generalize the well-studied notion of \textit{star configurations}~\cite{GHM13,AGT-2017,Man20}.  It follows from~\cite{GHMN17} that the asymptotic resurgence (and resurgence) of a matroid configuration is bounded above by the asymptotic resurgence (respectively, resurgence) of the Stanley-Reisner ideal of the corresponding matroid.  Thus, our findings have immediate consequences for the asymptotic resurgence of matroid configurations.  In an upcoming paper, we study both the asymptotic resurgence and resurgence of certain matroid configurations.

We summarize our methodology and highlight our main results.  While Stanley-Reisner ideals of matroids are perhaps more natural to study than facet ideals of matroids, a result of Villarreal~\cite{Villarreal-2023} implies that the asymptotic resurgence of the Stanley-Reisner ideal of a matroid is equal to the asymptotic resurgence of the facet ideal of the dual matroid (see \Cref{prop:Alex duality facet}).  With this in mind, we primarily study the asymptotic resurgence of facet ideals of matroids and periodically emphasize the consequences for the asymptotic resurgence of Stanley-Reisner ideals.  As facet ideals of matroids are equigenerated, we can apply a technique from~\cite{DFMS19} to reduce the computation of their asymptotic resurgence to the computation of \textit{Waldschmidt constants}.  Before explaining this in more detail, we recall the notion of the Waldschmidt constant.

The Waldschmidt constant originated in complex analysis~\cite{Waldschmidt-1977} and was rediscovered in the context of commutative algebra by Bocci-Harbourne~\cite{BH10}.  For a homogeneous ideal $I=\bigoplus_{i\ge 0} I_i\subset S$, we let $\alpha(I):=\min\{i~:~ I_i\neq 0\}$ denote the smallest generator degree of $I$.  The Waldschmidt constant of $I$ is the limit
$
\widehat{\alpha}(I):=\lim\limits_{s\to\infty} \frac{\alpha(I^{(s)})}{s}.
$
The limit exists by subadditivity (see~\cite{BH10}).  The Waldschmidt constant is itself the subject of a great deal of research; see, for example, \cite{DumnickiFashamiSzpondTutajGasinska, DumnickiSzemberSzpond, BCGHJNSVV}.

Our first result, derived in \Cref{Sec:asymp-resurgence}, expresses the asymptotic resurgence of the facet ideal of a matroid in terms of the Waldschmidt constant of the facet ideals of its contractions.  This is one of the main tools we use to study asymptotic resurgence of facet ideals of matroids.

\begin{Main1}
Let $M$ be a matroid on a ground set $E$.  Then
\[
\widehat{\rho}(I(M))=\max_{U\subseteq E, \rk_M(U)<\rk(M)}
\left\lbrace\dfrac{\rk(M)-\rk_M(U)}{\widehat{\alpha}(I(M/U))}\right\rbrace,
\]
where $\rk(M)$ is the rank of $M$, $\rk_M(U)$ is the rank of $U$, and $M/U$ is the contraction of $U$ from $M$.
\end{Main1}

\Cref{thm:asym-res-facet} highlights the importance of determining the Waldschmidt constant of the facet ideal of a matroid.  
In general, determining the Waldschmidt constant of the facet ideal of a matroid appears to be more challenging than determining the Waldschmidt constant of the Stanley-Reisner ideal of a matroid (see \Cref{rem:appSRTutte} and \Cref{rem:appFTutte}).  The latter problem is studied in ~\cite{DFKT24} and~\cite{ManNgu}.

In~\cite[Theorem~1.2]{GHV13}, Guardo-Harbourne-Van Tuyl prove that
\begin{equation}\label{eq:rhohatineq}
\frac{\alpha(I)}{\widehat{\alpha}(I)}\le \widehat{\rho}(I).
\end{equation}
If $I$ is the facet ideal of a matroid, this inequality is apparent from \Cref{thm:asym-res-facet}.  We refer to \eqref{eq:rhohatineq} as \textit{the GHV lower bound}; if an ideal $I$ satisfies \eqref{eq:rhohatineq} with equality, we say $I$ achieves equality in \GHV{}.

In \Cref{sec:GHVbound}, we pose \Cref{ques:rhohat=alpha/alphahat}: \textit{Which matroids have a facet ideal achieving equality in \GHV{}?}  This is a useful question to have in the back of one's mind for the remainder of the paper.  We provide some initial answers to this question in \Cref{sec:GHVbound}; further answers for several classes of matroids are provided in \Cref{sec:weakorder,sec:DesignsAndPMDs,sec:asym-Steiner-system}. 

In \Cref{sec:weakorder}, we first derive another tool for studying asymptotic resurgence of facet ideals of matroids.  It states that asymptotic resurgence respects the weak order on matroids of the same rank.

\begin{Main2}
Let $M_1$ and $M_2$ be matroids of rank $k$ on the same ground set so that $M_1\preceq M_2$ in the weak order.  Then $\widehat{\rho}(I(M_1)) \le \widehat{\rho}(I(M_2))$ and $\widehat{\rho}(I_{\Delta(M_1)})\le \widehat{\rho}(I_{\Delta(M_2)})$.  
\end{Main2}

An immediate consequence of \Cref{thm:asym-order} is that the uniform matroid $\mathrm{U}_{k,n}$ of rank $k$ on a ground set of size $n$ achieves maximal asymptotic resurgence among matroids of rank $k$ on $n$ elements.  However, this upper bound is only achieved by uniform matroids.  If a matroid is not uniform, it necessarily precedes a so-called \textit{almost-uniform matroid} $\mathrm{U}^{-}_{k,n}$ (see \Cref{ss:almost-uniform}) in weak order.  We compute the Waldschmidt constant of almost-uniform matroids in \Cref{lem:relaxation-to-uniform} and deduce, via \Cref{thm:asym-res-facet} and \Cref{thm:asym-order}, the asymptotic resurgence of facet ideals of almost-uniform matroids in \Cref{thm:relaxation-to-uniform}.

\begin{Main3}
Let $\mathrm{U}_{k,n}^-$ be an  almost-uniform matroid of rank $k$ on the ground set $E$ of size $n$ with $k <n$.   Then, 
    \[
    \widehat{\rho}(I(\mathrm{U}_{k,n}^-))=\max\left\{\frac{(k-1)(n-k+1)}{n-1},\frac{k(n-k)}{n-1}\right\}
    =\begin{cases}
     \dfrac{k(n-k)}{n-1} & \text{ if } 2k\le {n+1}\\[10 pt]
    \dfrac{(k-1)(n-k+1)}{n-1} & \text{ if } 2k\ge n+1.
    \end{cases}
    \]
\end{Main3}

\Cref{thm:relaxation-to-uniform}, coupled with \Cref{thm:asym-order}, yields an upper bound on the asymptotic resurgence of facet ideals of non-uniform matroids of rank $k$ on $n$ elements (\Cref{cor:asym-upper-bound}).  This upper bound is achieved by quite a few matroids besides the almost-uniform matroids (see \Cref{rem:chainsinweakorder}).

By \Cref{thm:relaxation-to-uniform} and \Cref{lem:relaxation-to-uniform}, the facet ideal of an almost-uniform matroid of rank $k$ on $n$ elements achieves equality in \GHV{} if and only if $2k\le n+1$.

One of our main results is that the large class of matroids whose dual is paving roughly follows the trend of almost-uniform matroids with regard to \GHV{}.

\begin{Main4}
Let $M$ be a non-uniform matroid of rank $k$ on the ground set $E$ of size $n$.  If $M^*$ is paving and $2k\le n$, then
\[
\widehat{\rho}(I(M))=\frac{\alpha(I(M))}{\widehat{\alpha}(I(M))}=\frac{k}{\widehat{\alpha}(I(M))}.
\]
\end{Main4}

\Cref{thm:asym-res-dual-paving} is sharp.  If $n\le 2k-2$, the almost-uniform matroid $\mathrm{U}_{k,n}^-$ satisfies $\widehat{\rho}(I(\mathrm{U}_{k,n}^-))>\frac{k}{\widehat{\alpha}(I(\mathrm{U^-_{k,n}}))}$.  Moreover, for any integer $k\ge 5$, we construct in \Cref{app:sharp} a sparse paving matroid $M_{k,2k-1}$ of rank $k$ on $2k-1$ elements so that $\widehat{\rho}(I(M_{k,2k-1}))>\frac{k}{\widehat{\alpha}(I(M_{k,2k-1}))}$ (\Cref{ex:sparse-paving-family} and \Cref{prop:sparse-paving-family}).

In \cite{MNWW-2011}, it is conjectured that, as $n$ tends to infinity, the ratio of the number of sparse paving matroids of rank $k$ with $(n-1)/2\le k\le (n+1)/2$ to the total number of matroids on $n$ elements tends to one.  Based on this conjecture, \Cref{thm:asym-res-dual-paving}, and computational evidence (see \Cref{tab:master-full}), we propose in \Cref{conj:asymptoticresurgencemostmatroids} that, as $n$ tends to infinity, the ratio of the number of matroids whose facet ideal achieves equality in \GHV{} to the total number of matroids on $n$ elements tends to one.

While \Cref{thm:asym-res-dual-paving} covers a large class of matroids, many of which have no symmetries, in \Cref{sec:DesignsAndPMDs} and \Cref{sec:asym-Steiner-system} we address two specialized classes of matroids with a great deal of symmetry: perfect matroid designs and sparse paving matroids arising from Steiner systems (our study of the latter was inspired by \cite{BFGM21}).  Both of these classes of matroids are closed under contraction and have symmetries encoded by \textit{designs}, which allow the determination of the Waldschmidt constant (\Cref{lem:bd-growth-rate}).  Thus, we determine the asymptotic resurgence using \Cref{thm:asym-res-facet}.  Both classes have facet ideals which achieve equality in \GHV{}.

\begin{Main5}
Suppose $M$ is a perfect matroid design of rank $k$ on the ground set $E$ of size $n$.  Let $\ell$ be the number of loops of $M$ and $c$ be the size of a circuit of $M^*$ (all of these have the same size).  Then
\[
\widehat{\rho}(I(M))=\frac{kc}{n-\ell}=\widehat{\rho}(I_{\Delta(M^*)}).
\]
\end{Main5}

If $\mathcal{S}$ is a Steiner system of type $S(t,k,n)$, then one can define a sparse paving matroid $M(\mathcal{S})$ whose circuit-hyperplanes are the blocks of $\mathcal{S}$ (see \Cref{def:matroidfromSteiner}).

\begin{Main6}
Suppose $M=M(\mathcal{S})$ is the matroid of a Steiner system of type $S(t,k,n)$ with $t<k<n$.  Then
\[
\widehat{\rho}(I(M))=\dfrac{k(n-k)}{n}=\widehat{\rho}(I_{\Delta(M^*)}).
\]
\end{Main6}

In \Cref{sec:conc-rmk}, we close the main body of the paper with a few remarks making additional connections to the literature.

We include three appendices.  \Cref{app:algo} contains a description of an algorithm to compute asymptotic resurgence of squarefree monomial ideals, which we implemented in Macaulay2~\cite{M2}. \Cref{app:sharp} contains a family of examples that show that \Cref{thm:asym-res-dual-paving} is sharp, including a couple of explicit examples where we used the algorithm from \Cref{app:algo} to compute the asymptotic resurgence.  \Cref{app:small-rank} contains data on the asymptotic resurgence of simple matroids on ground sets of size eight or less, computed using the algorithm from \Cref{app:algo}.  This data informed \Cref{conj:asymptoticresurgencemostmatroids}.  We make a number of additional observations regarding the data in a sequence of remarks in \Cref{app:small-rank}.  Notably, computations show that neither the Waldschmidt constant nor the asymptotic resurgence of the facet ideal of a matroid can be computed from its Tutte polynomial (see~\Cref{rem:appFTutte}).  It follows that the asymptotic resurgence of the Stanley-Reisner ideal of a matroid cannot be detected from its Tutte polynomial.  This is in contrast to the \textit{Waldschmidt constant} of the Stanley-Reisner ideal of a matroid, which \textit{can} be determined from its Tutte polynomial (see~\Cref{rem:appSRTutte}).

\vspace{10 pt}

\section{Background}
In this section, we collect the notation and background that we need for the remainder of the paper.  We fix the following conventions and notation:
\begin{itemize}
\item $S=\KK[x_1,\ldots,x_n]$ is a polynomial ring over a field $\KK$.
\item If $f_1,\ldots,f_k\in S$, we write $\langle f_1,\ldots,f_k\rangle\subset S$ for the ideal generated by $f_1,\ldots,f_k$.
\item If ${\bf a}=(a_1,\ldots,a_n)\in\ZZ_{\ge 0}^n$, we write $x^{\bf a}$ for the monomial $x_1^{a_1}\cdots x_n^{a_n}$.
\item If $I$ is a homogeneous ideal of $S$, we write $\mathcal{G}(I)$ for a minimal homogeneous generating set; if $I$ is a monomial ideal, $\mathcal{G}(I)$ consists of a minimal set of monomial generators of  $I$.
\item If $n$ is a positive integer, we put $[n]:=\{1,\ldots,n\}$.
\item If $U\subset [n]$, we put $x^U:=\prod\limits_{i\in U}x_i\in S$ and $P_U:=\langle x_i~:~i\in U\rangle\subset S$.
\end{itemize}
In this paper, all the ideals we consider are homogeneous.  For an ideal $I\subset S$, we write $\mbox{Min}(I)$ for its set of minimal primes.  
The $s^{th}$ symbolic power of a squarefree monomial ideal $I$ satisfies (see~\cite{CEHH17})
\[
I^{(s)}=\bigcap_{P\in\mbox{Min}(I)}P^s.
\]

\subsection{The Newton and symbolic polyhedra}
In this section, we recall the definitions of the Newton and symbolic polyhedra.  The Waldschmidt constant and asymptotic resurgence can be computed from the symbolic polyhedron by \cite{CEHH17,DFMS19}.  Our references for this section are Ziegler's text~\cite{Ziegler95} for polyhedra and the article~\cite{CEHH17} for the Newton and symbolic polyhedra.

Let $\bf{e}_1,\ldots,\bf{e}_n\in \RR^n$ be the standard basis vectors.  If ${\bf v}=\sum\limits_{i=1}^n v_i {\bf e}_i\in\RR^n$, we will also write ${\bf v}$ as the tuple ${\bf v}=(v_1,\ldots,v_n)$.  We write $\widehat{\RR^n}$ for the dual vector space with dual basis $\widehat{\bf{e}}_1,\ldots,\widehat{\bf{e}}_n$.  For vectors $\bf{v}\in \RR^n,\bf{w}\in\widehat{\RR^n},$ we write $\langle\bf{v},\bf{w}\rangle$ for the pairing between $\RR^n$ and $\widehat{\RR^n}$, which is given by the standard dot product if $\bf{v},\bf{w}$ are written in the bases $\bf{e}_1,\ldots,\bf{e}_n$ and $\widehat{\bf{e}}_1,\ldots,\widehat{\bf{e}_n}$, respectively.

Suppose $A=\{{\bf a}_1,\ldots,{\bf a}_p\}\subseteq \RR^n$ is a finite collection of vectors.  The \textit{conical hull} of $A$ is the collection of all \textit{conical combinations} of vectors of $A$:
\[
\cone(A):=\left \{\sum_{i=1}^p c_i{\bf a}_i~:~ c_i \ge 0 \mbox{ for } i=1,\ldots, p\right\}.
\]
The \textit{convex hull} of $A$, written $\conv(A)$, is the collection of all \textit{convex combinations} of vectors in $A$:
\[
\conv(A):=\left\{\sum_{i=1}^p c_i{\bf a}_i~:~ c_i\ge 0 \mbox{ for } i=1,\ldots, p\mbox{ and }\sum_{i=1}^p c_i=1\right\}.
\]
A \textit{polyhedral cone} is the conical hull of a finite set of vectors in $\RR^n$, and a \textit{polytope} is the convex hull of a finite set of vectors in $\RR^n$.  In particular, we define $\RR^n_{\ge 0}:=\cone({\bf e}_1,\ldots,{\bf e}_n)$ and $\widehat{\RR^n}_{\ge 0}=\cone(\widehat{e_1},\ldots,\widehat{e_n})$.

Given sets $X,Y\subseteq  \RR^n$ (not necessarily finite), the \textit{Minkowski sum} of $X$ and $Y$ is
\[
X+Y:=\{{\bf a}+{\bf  b}~:~{\bf a}\in X,{\bf b}\in Y\}.
\]
A \textit{polyhedron} is the Minkowski sum of a polyhedral cone and a polytope.

There are dual descriptions for polytopes, cones, and polyhedra in terms of intersections of half-spaces.  Given a vector ${\bf w}\in \widehat{\mathbb{R}^n}$ and a real number $r\in\RR$, define
\[
H^+_{{\bf w},r}=\{{\bf v}\in\RR^n~:~ \langle {\bf w}, {\bf v}\rangle \ge r\} \quad\mbox{and}\quad H_{{\bf w},r}=\{{\bf v}\in\RR^n~:~ \langle {\bf w},{\bf v}\rangle =r\}.
\]
Sets of the form $H^+_{{\bf w},r}$ are called \textit{half-spaces},  $H_{{\bf w},r}$ is the boundary hyperplane of $H^{+}_{{\bf w},r}$,
and ${\bf w}$ is the inward pointing normal to $H^{+}_{{\bf w},r}$.  A polyhedron may alternatively be described as the intersection of finitely many half-spaces, while a polytope is a bounded polyhedron.  A cone is an intersection of finitely many half-spaces whose boundary hyperplane passes through the origin.

\begin{Definition}\label{def:NewtonPolyhedron}
Let $I\subset S$ be a monomial ideal with generating set $\mathcal{G}(I)$.  The Newton polyhedron of $I$ is the Minkowski sum
\[
\NP(I):=\conv({\bf a} ~:~ x^{\bf a} \in \mathcal{G}(I))+\RR^n_{\ge 0}.
\]
\end{Definition}

In~\cite{CEHH17}, the symbolic polyhedron is defined for an arbitrary monomial ideal.  We give the definition only in the squarefree case.  If $U\subset [n]$, we define $\chi_U:=\sum_{i\in U} \widehat{{\bf e}}_i\in \widehat{\RR^n}$. 

\begin{Definition}\label{def:SymbolicPolyhedron}
Let $I$ be a squarefree monomial ideal of $S$. The symbolic polyhedron of $I$ is
\[
\SP(I):=\left\{{\bf a}=(a_1,\ldots,a_n)\in\RR^n_{\ge 0}~:~ \langle \chi_U, {\bf a} \rangle=\sum_{i\in U} a_i \ge 1\mbox{ for all } P_U\in\mbox{Min}(I)\right\}.
\]
\end{Definition}

The following lemma encodes an important dictionary between the algebra of monomials and convex geometry.  Recall that the \textit{integral closure} of a monomial ideal $I\subset S$ is defined as
\[
\overline{I}=\{m\in S: m^r\in I^r \mbox{ for some positive integer } r \}.
\]

\begin{Lemma}\label{lem:regular_symbolic_SP}
Suppose $m=x^{\bf a}$ is a monomial in $S$ and $I$ is a squarefree monomial ideal of $S$.  Then
\begin{enumerate}
    \item $m\in \overline{I^r}$ if and only if $\frac{{\bf a}}{r}\in \NP(I)$ and
\item  $m\in I^{(s)}$ if and only if $\frac{{\bf a}}{s}\in \SP(I)$.
\end{enumerate}
\end{Lemma}
\begin{proof}
The proof is straightforward from the definitions and the description of integral closure of monomial ideals in~\cite{HS-2006}.  See also~\cite[Section~2]{DFMS19}.
\end{proof}

We will use the following result that allows the computation of the Waldschmidt constant of a squarefree monomial ideal from its symbolic polyhedron (see also \cite[Theorem~3.2]{BCGHJNSVV}).

\begin{Theorem}[{\cite[Corollary~6.3]{CEHH17}}]\label{thm:SPWaldschmidt}
Let $I$ be a squarefree monomial ideal of $S$.  Then
\[
\widehat{\alpha}(I)=\min\left\{\langle \chi_{[n]},{\bf a}\rangle=\sum_{i=1}^n a_i~:~ {\bf a}=(a_1,\ldots,a_n)\in \SP(I)\right\}.
\]
\end{Theorem}

\subsection{Stanley-Reisner ideal of a simplicial complex and Alexander duality}
A \textit{simplicial complex} $\Delta$ on a vertex set $[n]$ is a non-empty collection of subsets of $[n]$ (called \textit{faces} of $\Delta$) satisfying that if $\sigma\in\Delta$ and $\sigma'\subseteq\sigma$, then $\sigma'\in\Delta$.  A \textit{non-face} of a simplicial complex $\Delta$ is a subset of $[n]$ that is not an element of $\Delta$.  Observe that a simplicial complex is determined by its set of \textit{minimal} non-faces; the \textit{Stanley-Reisner ideal} of $\Delta$ is the squarefree monomial ideal $I_{\Delta}$ defined as
\[
I_{\Delta} :=\langle x^U~:~ U \mbox{ a minimal non-face of } \Delta\rangle.
\]
The Stanley-Reisner ideal admits a prime decomposition (see~\cite[Theorem~1.7]{Miller-Sturmfels-2005})
\begin{equation}\label{eq:SRPrimeDecomposition}
I_{\Delta}=\bigcap_{\sigma\in\Delta} P_{[n]\setminus \sigma}.
\end{equation}
Clearly, the intersection can be taken over the maximal faces $\sigma\in\Delta$, called \textit{facets}.

Recall that if $I$ is a squarefree monomial ideal, the \textit{Alexander dual} of $I$, which we denote by $\dual{I}$, is defined by
\[
\dual{I}:=\langle x^U~:~ P_U\in\mbox{Min}(I)\rangle.
\]
For Stanley-Reisner ideals, this has a nice interpretation in terms of $\Delta$ due to~\eqref{eq:SRPrimeDecomposition}:
\begin{equation}\label{eq:AlexDual}
\dual{I_{\Delta}}= \langle x^{[n]\setminus\sigma}: \sigma\in\Delta\rangle.
\end{equation}
We will frequently use the following important result of Villarreal.

\begin{Theorem}[{\cite[Theorem~3.7]{Villarreal-2023}}] \label{thm:Villa-duality}
If $I$ is a squarefree monomial ideal, then $\widehat{\rho}(I)=\widehat{\rho}(\dual{I})$.
\end{Theorem}

\subsection{Matroids and their facet and Stanley-Reisner ideals}

In this section, we collect some definitions and terminology related to matroids and their associated facet and Stanley-Reisner ideals.  Standard references for matroids, which we use throughout the paper, are~\cite{Oxley2011} and~\cite{Wel}.

A \emph{matroid} $M=(E,\cI)$ is a pair where $E$ is a finite set called the \textit{ground set} of $M$ and $\cI$ is a collection of subsets of $E$ called the \textit{independent sets} of $M$, which satisfy the following (independent set) axioms:
\begin{enumerate}
\item $\emptyset\in \cI$. 
\item If $\sigma\in\cI$ and $\sigma'\subset\sigma$, then $\sigma'\in\cI$. 
\item If $\sigma,\sigma'\in\cI$ and $|\sigma|>|\sigma'|$, then there is some $x\in\sigma\setminus\sigma'$ so that $\sigma'\cup\{x\}\in\cI$.
\end{enumerate}
Observe that the first two independent set axioms imply that the independent sets of a matroid $M$ form a simplicial complex.  This simplicial complex is called the \textit{independence complex} of the matroid; we denote it by $\Delta(M)$.

The \emph{bases} of a matroid $M$, which we denote $\cB(M)$, are the maximal independent sets.  A pair $M=(E, \cB)$ is a matroid with bases $\cB$ if and only if the following (basis) axioms are satisfied:
\begin{enumerate}
\item $\cB\neq \emptyset$,
\item (\textit{Basis Exchange Axiom}) If $B,B'\in\cB$ and $x\in B\setminus B'$, then there is some $y\in B'\setminus B$ so that $\left(B\setminus \{x\}\right)\cup\{y\}\in\cB$.
\end{enumerate}

A subset $D\subseteq E$ is called \emph{dependent} in $M$ if $D\not\in \cI$.  The minimal dependent sets are called the \emph{circuits} of $M$. We denote by $\cir(M)$ the collection of circuits of $M$. 

Let $M$ be a matroid on the ground set $E$ of size $n$.  Without loss of generality, we may assume $E=[n]$.  Let $S=\KK[x_e~:~e\in E]$.  The {\it facet ideal (or basis ideal) of $M$}, denoted as $I(M)$, is a squarefree monomial ideal defined as
\[
I(M):=\langle x^B~:~ B \in \cB(M) \rangle \subset S.
\]
Since the non-faces of $\Delta(M)$ are the dependent sets and the minimal dependent sets are circuits of $M$, the Stanley-Reisner ideal of $\Delta(M)$ is
\[
I_{\Delta(M)}=\langle x^C~:~ C\in\cir(M)\rangle \subset S.
\]
For the prime decomposition of $I_{\Delta(M)}$, we use the notion of matroid duality.  If $M=(E,\cB)$ is a matroid, the \textit{dual matroid} $M^*$ is the matroid on the same ground set with bases $\cB(M^*)=\{E\setminus B: B\in \cB\}$.  From~\eqref{eq:SRPrimeDecomposition}, we have
\[
I_{\Delta(M)}=\bigcap_{B\in \cB(M^*)} P_B.
\]

Alexander duality relates the facet ideal of a matroid to the Stanley-Reisner ideal of the dual matroid.

\begin{Proposition} \label{prop:Alex duality facet}
  Let $M$ be a matroid on the ground set $E$. Then $\dual{I_{\Delta(M^*)}}=\langle x^B: B\in \cB(M)\rangle=I(M)$. 
In particular, $\widehat{\rho}(I(M))=\widehat{\rho}(I_{\Delta(M^*)})$.
\end{Proposition}
\begin{proof}
Since $\Delta(M^*)$ is the independence complex of $M^*$, its simplices are independent sets of $\Delta(M^*)$.  By \eqref{eq:AlexDual}, $\dual{I_{\Delta(M^*)}}$ is generated by squarefree monomials corresponding to complements of independent sets of $M^*$.  The minimal generators of this ideal are the squarefree monomials corresponding to complements of bases of $M^*$, which by definition is the set $\{x^B~:~ B\in \cB(M)\}$.  This proves the first equality.  The second equality follows from \Cref{thm:Villa-duality}.
\end{proof}

Since Alexander duality exchanges minimal generators with minimal primes, the prime decomposition of the facet ideal of a matroid is given by
\[
I(M)=\dual{I_{\Delta(M^*)}}=\bigcap_{C\in\C(M^*)} P_C.
\]
This decomposition yields a concrete description of $\SP(I(M))$, which is immediate from \Cref{def:SymbolicPolyhedron}.
\begin{Corollary}\label{cor:SPIM}
Let $M$ be a matroid on the ground set $E=[n]$.  The defining inequalities of $\SP(I(M))\subseteq\RR^n$ are given by $\langle \chi_C,{\bf a}\rangle=\sum\limits_{i\in C}a_i\ge 1$ for every circuit $C\in \C(M^*)$.
\end{Corollary}

\section{Asymptotic resurgence of the facet ideal of a matroid}\label{Sec:asymp-resurgence}
In this section, we develop one of our main techniques for studying the asymptotic resurgence of facet ideals of matroids. Our approach is based on a reduction principle that expresses the asymptotic resurgence of a facet ideal of a matroid in terms of the Waldschmidt constants of facet ideals and the asymptotic resurgence of single-element contractions (see \Cref{cor:rhosinglecontraction} and \Cref{prop:asym-single-contraction-matroid}). 
As a consequence, we obtain the main result of this section, showing that the asymptotic resurgence of a facet ideal of a matroid can be computed entirely from the Waldschmidt constants of the facet ideals of its contractions (see \Cref{thm:asym-res-facet}).

\subsection{Asymptotic resurgence for squarefree monomial ideals via contraction}\label{ss:squarefreecontraction}

Let $S=\KK[x_1,\ldots,x_n]$ be a polynomial ring over a field $\KK$.
If $I$ is a monomial ideal of $S$, the \textit{contraction} of $I$ at a subset $U\subset [n]$ is the ideal $I/U$ generated by the monomials obtained from the monomial generators of $I$ by setting the variables $\{x_i:i\in U\}$ to one.  Similarly, the \textit{deletion} of $I$ at $U$ is the ideal $I\backslash U=\langle m\in \mathcal{G}(I) \mid x_i \nmid m \mbox{ for all } i\in U\rangle$ generated by the monomials obtained from the monomial generators of $I$ by setting the variables $\{x_i:i\in U\}$ to zero.  Clearly, we can view $I/U$ and $I\backslash U$ as living in the smaller polynomial ring $\KK[x_i:i\in [n]\setminus U]$.  For our purposes, it is more convenient to view $I/U$ and $I\backslash U$ as living in the same ring as $I$; then we have $I/U=I:(x^U)^\infty$  and $I\backslash U+P_U=I+P_U$.  If $I$ is squarefree, then $I:(x^U)^\infty=I:x^U$, so $I/U=I:x^U$. 

\begin{Definition}\label{def:rhocontraction}
Suppose $I\subset S$ is a  squarefree monomial ideal -- that is, all its generators have the same degree.  We define
\[
\widehat{\rho}_c(I)=\max\{\widehat{\rho}(I/U)~:~ U\subset [n] \text{ and } U \neq \emptyset\}.
\]
We also define
\[
\widehat{\rho}_{c,1}(I)=\max\{\widehat{\rho}(I/\{i\})~:~ i\in [n] \text{ and } I:x_i \neq I\}.
\]
The definition of $\widehat{\rho}_{c,1}(I)$ ensures that we take the maximum over those $i\in [n]$ so that $I/\{i\}\neq I$.
\end{Definition}

The following proposition is a slight modification of~\cite[Corollary~2.24]{DFMS19}.

\begin{Proposition}\label{prop:rhohatdimensionreduction}
Let $I\subset S$ be an equigenerated squarefree monomial ideal. 
 Then
\[
\widehat{\rho}(I)=\max\left\lbrace \frac{\alpha(I)}{\widehat{\alpha}(I)},\widehat{\rho}_c(I)\right\rbrace
.\]
\end{Proposition}
\begin{proof}
According to ~\cite[Corollary~2.24]{DFMS19}, the statement of the proposition holds if the dimension of the affine span of the exponent vectors of the minimal generators of $I$ is equal to $n-1$.  Our assumption that $I$ is equigenerated implies that the dimension of the affine span of the exponent vectors of the minimal generators of $I$ is \textit{at most} $n-1$, so we consider the case when the affine span of the exponent vectors of the generators of $I$ is strictly less than $n-1$.  In this case, \cite[Corollary~2.24]{DFMS19} tells us that $\widehat{\rho}(I)=\widehat{\rho}_c(I)$.  However, the inequality $\frac{\alpha(I)}{\widehat{\alpha}(I)}\le \widehat{\rho}(I)$ holds for any ideal, so the equality $\widehat{\rho}_c(I)=\widehat{\rho}(I)=\max\left\lbrace \frac{\alpha(I)}{\widehat{\alpha}(I)},\widehat{\rho}_c(I)\right\rbrace$ is still true in this case.
\end{proof}

We simplify \Cref{prop:rhohatdimensionreduction} by showing that it suffices to take single-element contractions.  We start with the following lemma.

\begin{Lemma}\label{lem:contractioncontainment}
Let $I\subset S$ be a squarefree monomial ideal and $s,r$ positive integers.  If $I^{(s)}\subseteq I^r$, then for any $U\subset [n]$,
\begin{enumerate}
\item $(I\backslash U)^{(s)}\subseteq (I\backslash U)^r$ and
\item $(I/U)^{(s)}\subseteq (I/U)^r$.
\end{enumerate}
\end{Lemma}
\begin{proof}
Suppose that $I^{(s)}\subseteq I^r$ and $U\subset [n]$.  We observe that, if $i\in U$, then $I/U=(I/\{i\})/(U\setminus\{i\})$ and $I\backslash U=(I\backslash\{i\})\backslash(U\setminus \{i\})$.  Furthermore, $I:x^U=(I:x_i): x^{U\setminus\{i\}}$.  By a simple induction, it thus suffices to prove (a) and (b) for $U=\{i\}$, for some $i\in [n]$. 

\par (a) 
Observe that $(I\backslash \{i\})^{(s)}\subseteq I^{(s)}\subseteq I^r$ and $(I^r)\backslash\{i\}=(I\backslash\{i\})^r$. For any minimal generator $m \in (I\backslash \{i\})^{(s)}$, we note that $x_i \nmid m$ and thus $m\in (I^r)\backslash\{i\}$ and the conclusion now follows.

\par  (b) Since $I^{(s)}\subseteq I^r$, then $I^{(s)}:x_i^\infty\subseteq I^r:x_i^\infty$. Hence it suffices to prove that $I^{(s)}:x_i^\infty=(I:x_i)^{(s)}$ and $I^r:x_i^\infty=(I:x_i)^r$.  Write $I=\bigcap\limits_{P_B\in\mbox{Min}(I)}P_B$.  Then, $I^{(s)}=\bigcap\limits_{P_B\in\mbox{Min}(I)}P_B^s$ and since taking colons commutes with intersections, then $I^{(s)}:x_i^\infty=(I:x_i)^{(s)}$.

Now, suppose that $m\in I^r:x_i^\infty$ is a minimal generator, so that $x_i$ does not divide $m$.  If $m\in I^r$, then $m\in (I:x_i)^r$ and we are done.  So, let us assume that $m\notin I^r$.  It follows that, for some $k\ge 1$, $mx_i^k\in I^r$.  Let $k$ be chosen minimally, so that $mx_i^{k-1}\notin I^r$.  Then, $mx_i^k=\prod\limits_{j=1}^r M_j$, where $M_1,\ldots,M_r\in I$.  If $x_s^2|M_t$ for some $1\le s\le n$, $s\neq i$, and some $t\in [r]$, then
\[
(m/x_s)x_i^k=(M_t/x_s)\prod_{j\neq t}M_j,
\]
and so $m/x_s\in I^r:x_i^\infty$ (since $M_t/x_s\in I$ as $I$ is squarefree), contradicting the minimality of $m$.  If $x_i^2|M_t$ for some $t\in [n]$ then
\[
mx_i^{k-1}=(M_t/x_i)\prod_{j\neq t}M_j.
\]
Since $M_t/x_i\in I$, $mx_i^{k-1}\in I^r$, contradicting the minimality of $k$.  It follows that $M_1,\ldots,M_r$ are all squarefree, and so $k\le r$ and we may assume without loss that $x_i|M_t$ for $t=1,\ldots,k$.  So
\[
m=\prod_{j=1}^k (M_j/x_i)\cdot \prod_{j=k+1}^r M_j,
\]
and in particular $m\in (I:x_i)^r$.  Thus $I^r:x_i^\infty\subseteq (I:x_i)^r$.

Now, suppose $m\in (I:x_i)^r$ and write $m=\prod\limits_{j=1}^r N_j$, where $N_1,\ldots,N_r\in I:x_i$.  Then,
\[
mx_i^r=\prod_{j=1}^r (N_jx_i)\in I^r.
\]
Hence, $(I:x_i)^r=I^r:x_i^\infty$.
\end{proof}

\begin{Corollary}\label{cor:resurgencecontraction}
Let $I\subset S$ be a squarefree monomial ideal.  Then, for any $U\subset [n]$,  \[\rho(I)\ge\max\{\rho(I/U),\rho(I\backslash U)\} \text{ and } \widehat{\rho}(I)\ge \max\{\widehat{\rho}(I/U),\widehat{\rho}(I\backslash U)\}.\] 
\end{Corollary}
\begin{proof}
Let $U\subset [n]$.  It suffices to show $\rho(I)\ge \rho(I/U)$, $\widehat{\rho}(I)\ge \widehat{\rho}(I/U)$, $\rho(I)\ge \rho(I\backslash U)$, and $\widehat{\rho}(I)\ge \widehat{\rho}(I\backslash U)$.  We show that $\rho(I)\ge \rho(I/U)$ -- the remaining inequalities are proved in the same fashion.  Recall that $\rho(I)=\sup\left\{\frac{s}{r}~:~ s,r \in\NN \mbox{ and } I^{(s)}\not\subset I^r\right\}$.  If $(I/U)^{(s)}\not\subset (I/U)^r$, then by  \Cref{lem:contractioncontainment}, $I^{(s)}\not\subset I^r$.  So, $\left\{\frac{s}{r}~:~s,r \in \NN \text{ and } (I/U)^{(s)}\not\subset (I/U)^r\right\}\subseteq \left\{\frac{s}{r}~:~s,r \in \NN \text{ and } I^{(s)}\not\subset I^r\right\}$ and thus $\rho(I/U)\le \rho(I)$.
\end{proof}

\begin{Corollary}\label{cor:rhosinglecontraction}
Let $I\subset S$ be an equigenerated squarefree monomial ideal.  Then,
\[
\widehat{\rho}(I)=\max\left\lbrace \frac{\alpha(I)}{\widehat{\alpha}(I)},\widehat{\rho}_{c,1}(I)\right\rbrace
,\]
where $\widehat{\rho}_{c,1}(I)$ is as in \Cref{def:rhocontraction}.
\end{Corollary}
\begin{proof}
We first consider the case that every variable appears in some minimal generator of $I$, so $I:x_i\neq I$ for any $1\le i\le n$.  In this case, $\widehat{\rho}_{c,1}(I)=\max\{\widehat{\rho}(I/\{i\})~:~ i\in [n]\}$.  Now, if $U\subset [n]$ and $i\in U$, then $I/U=(I/\{i\})/(U\setminus\{i\})$ and $\widehat{\rho}(I/U)\le \widehat{\rho}(I/\{i\})$ by \Cref{cor:resurgencecontraction}.  Hence, $\widehat{\rho}(I/U)\le\widehat{\rho}_{c,1}(I)$ and the result follows from \Cref{prop:rhohatdimensionreduction}.

For the general case, let $A=\{i\in [n] \mid I:x_i=I\}$.  None of the variables $\{x_j:j\in A\}$ appear in a minimal generator of $I$.  Consider the polynomial ring $\KK[x_i:i\in [n]\setminus A]$ and the ideal $J$ generated by the same minimal generators as $I$, but in the smaller polynomial ring.  From the first part of the proof,
\[
\widehat{\rho}(J)=\max\left\lbrace \frac{\alpha(J)}{\widehat{\alpha}(J)},\widehat{\rho}_{c,1}(J)\right\rbrace.
\]
It is immediate from the definitions that $\widehat{\rho}_{c,1}(I)=\widehat{\rho}_{c,1}(J)$.  One can check that $I^{(s)}\subset I^r$ if and only if $J^{(s)}\subset J^r$ for any  $s,r \in \NN$, which implies that $\widehat{\rho}(I)=\widehat{\rho}(J)$.  Moreover, $\alpha(I^{(s)})=\alpha(J^{(s)})$ for all $s\ge 1$, so
\[
\widehat{\rho}(I)=\widehat{\rho}(J)=\max\left\lbrace \frac{\alpha(J)}{\widehat{\alpha}(J)},\widehat{\rho}_{c,1}(J)\right\rbrace=\max\left\lbrace \frac{\alpha(I)}{\widehat{\alpha}(I)},\widehat{\rho}_{c,1}(I)\right\rbrace.\qedhere
\]
\end{proof}

\subsection{Asymptotic resurgence for facet ideals of matroids via contraction}\label{sec:asym-res-via-contr-matroid}
Let $M$ be a matroid on the ground set $E$. The \textit{rank} function of $M$ is a function $\rk_M:2^E\to \ZZ_{\ge 0}$ taking each subset $A\subseteq  E$ to a non-negative integer $\rk_M(A)$ defined by $\rk_M(A)=\max\left\{|A\cap B|~:~ B\in \cB(M)\right\}$. The rank of the entire ground set, $\rk_M(E)$, is called the \textit{rank of the matroid}, which we denote by $\rk(M)$.  Note that $\rk(M)$ is the common cardinality of every basis of $M$.  To avoid trivialities, we will always assume that $\rk(M)<|E|$.

Given a subset $A\subseteq  E$, the \textit{closure} of $A$ in  $M$, written $\mbox{cl}_M(A)$, is defined as the maximal set $F$ under inclusion, which contains $A$ and has the same rank as $A$. A  subset $A\subseteq  E$  is called a \textit{flat} of the matroid if $\mbox{cl}_M(A)=A$. We denote the flats of $M$ by $\cL(M)$ and the flats of rank $k$ by $\cL_k(M)$.  The set $\cL(M)$ has the structure of a geometric lattice, where the meet and join operations for two flats $F_1, F_2\in \cL(M)$ are, respectively, $F_1\cap F_2$ and $\mbox{cl}_M(F_1\cup F_2)$.

A \textit{loop} of a matroid $M$ on the ground set $E$ is an element $e\in E$ which is not contained in any basis of $M$.  We write $\loops(M)$ for the collection of all loops of a matroid.

We now recall the operations of contraction and restriction for matroids.  Let $U\subset E$ be a nonempty set. The {\it restriction of $M$ to $U$}, denoted as $M|U$, is a matroid of rank $\rk_M(U)$ on the ground set $U$. The independent sets in  $M|U$ are the independent sets of $M$ that are subsets of $U$. The {\it contraction of $U$ from $M$}, denoted as $M/U$, is a matroid of rank $k - \rk_M(U)$ on the ground set $E \setminus U$ (see \cite[Chapter 3]{Oxley2011})  whose bases are
\[
\cB(M/U) = \{ B \subseteq E \setminus U ~:~ \text{ there exists a basis } B' \text{ of } M|U \text{ such that } B \cup B' \in \cB(M) \}.
\]

\begin{Lemma}\label{lem:quotient-subset}
Let $M$ be a matroid on the ground set $E$, and let $U$ be a nonempty proper subset of $E$. Then, \( I(M/U)=I(M):x^U = I(M):x^{\mathrm{cl}_M(U)} \).
\end{Lemma}
\begin{proof}
Let \( U \) be a nonempty proper subset of \( E \).  Notice that $I(M):x^U$ is a squarefree monomial ideal as $I(M)$ is a squarefree monomial ideal. Now, by definition,   \begin{align*}
I(M): x^U &= \langle x^W ~:~ W \cup U \text{ contains a basis of } M \rangle \\ & = \langle x^W ~:~ W \subseteq E \setminus U, W \in \mathcal{I}(M) \text{ and } W \cup U \text{ contains a basis of } M \rangle \\ & = \langle x^W ~:~ W \subseteq E \setminus U, W \in \mathcal{I}(M) \text{ and } M|U \text{ has a basis } B \text{ such that } W \cup B \in \mathcal{B}(M) \rangle \\ & = \langle x^W ~:~ W \in \mathcal{B}(M/U) \rangle \\ &= I(M/U),
\end{align*}
proving the first equality of the lemma.  The second equality follows from the fact that $M/U$ and $M/\mathrm{cl}_M(U)$ have the same bases.  In fact,  $M/U$ is obtained from  $M/\mathrm{cl}_M(U)$ by adding a collection of loops that correspond to the elements of $\mathrm{cl}_M(U)$ that are not in $U$, see~\cite[Section~3.1, Exercise~8]{Oxley2011}. \end{proof}

The \textit{deletion of $U$ from $M$} is the matroid of rank $k$ on the ground set $E\setminus U$ \cite[Chapter 3]{Oxley2011} whose bases are
\[
\cB(M\backslash U)=\{B\in\cB(M)~:~ B\cap U=\emptyset\}.
\]

\begin{Remark}
We have chosen notation so that $I(M\backslash U)=I(M)\backslash U$ and $I(M/U)=I(M)/U$.  The first of these equalities follows immediately from the definitions, and the second follows from \Cref{lem:quotient-subset}. 
\end{Remark}

\begin{Proposition}\label{prop:asym-single-contraction-matroid}
For a matroid $M$ we have
\[
\widehat{\rho}_{c,1}(I(M))=\max\limits_{F\in \cL_1(M)}\widehat{\rho}(I(M/F))
\quad
\mbox{and}
\quad
\widehat{\rho}(I(M))=\max\left\lbrace \frac{{\rk(M)}}{\widehat{\alpha}(I(M))},\widehat{\rho}_{c,1}(I(M))\right\rbrace
.\]
\end{Proposition}
\begin{proof}
Let $M$ be a matroid on the ground set $E$. Observe that $\loops(M)=\{i\in E \mid I(M):x_i=I(M)\}$. 
By \Cref{def:rhocontraction}, $\widehat{\rho}_{c,1}(I(M))=\max\{\widehat{\rho}(I(M/e))~:~e\in E \setminus \loops(M)\}$.  For all $e \in E \setminus \loops(M),$ there is a flat $F \in \cL_1(M)$ so that $\text{cl}_M\{e\} = F$. Also, for every $F \in \cL_1(M)$, there is an $e \in F\setminus \loops(M)$ so that $\text{cl}_M\{e\} = F$. By \Cref{lem:quotient-subset}, $I(M/e)=I(M/\text{cl}_M\{e\})$, which implies that $\widehat{\rho}(I(M/e))=\widehat{\rho}(I(M/\text{cl}_M\{e\}))$. Thus,  \begin{align*}\widehat{\rho}_{c,1}(I(M))&=\max\{\widehat{\rho}(I(M/e))~:~e\in E \setminus \loops(M)\}\\&=\max\{\widehat{\rho}(I(M/\text{cl}_M\{e\}))~:~e\in E \setminus \loops(M) \}\\&=\max\{\widehat{\rho}(I(M/F))~:~F\in \cL_1(M) \}.\end{align*} Observe that $\alpha(I(M))=\rk(M)=k$, so the second part follows immediately from \Cref{cor:rhosinglecontraction}.\end{proof}

In our next result, we show that, in order to compute the asymptotic resurgence of $I(M)$, it suffices to compute the Waldschmidt constant of the facet ideals of all contractions of $M$. We will identify $M/(\text{cl}_M\{\emptyset\})$ with $M$ throughout this article.

\begin{Theorem}\label{thm:asym-res-facet}
Let $M$ be a matroid of rank $k$ on the ground set $E$ of size $n$.  Then, 
\[
\widehat{\rho}(I(M))= \max\limits_{U \subseteq E,~\rk_M(U)<k}\left\{ \frac{k-\rk_M(U)}{\widehat{\alpha}(I(M/U))}\right \} 
= \max\limits_{U \in \mathcal{I}(M),~|U|<k}\left\{ \frac{k-|U|}{\widehat{\alpha}(I(M/U))} \right \}
= \max\limits_{F \in  \cL(M) \setminus \{E\}}\left\{ \frac{k-\rk_M(F)}{\widehat{\alpha}(I(M/F))}  \right \},
\] 
where $\mathcal{I}(M)$ is the set of independent sets of $M$ and $\cL(M)$ is the lattice of flats of $M$.
\end{Theorem}
\begin{proof}
Let $U \subseteq E$ so that $\rk_M(U)<k$.  Let $V\in\cI(M)$ and $F\in \cL(M)$ so that $V \subseteq U \subseteq F$  and $\rk_M(U)=\rk_M(V)=\rk_M(F)$.  By \Cref{lem:quotient-subset}, $I(M/U)=I(M/V)=I(M/F)$.   Since $V$ is an independent set then $\rk_M(V)=|V|$.  Thus, we have 
\[
\max\limits_{U \subset E,~\rk_M(U)<k} \left\{ \frac{k-\rk_M(U)}{\widehat{\alpha}(I(M/U))}\right \} = \max\limits_{U \in \mathcal{I}(M),|U|<k}\left\{ \frac{k-|U|}{\widehat{\alpha}(I(M/U))} \right \}= \max\limits_{F \in  \mathcal{L}(M) \setminus \{E\}}\left\{ \frac{k-\rk_M(F)}{\widehat{\alpha}(I(M/F))}  \right \}.
\]
We now prove by induction on the rank of $M$ that
\[ 
\widehat{\rho}(I(M))= \max\limits_{F \in  \mathcal{L}(M) \setminus \{E\}}\left\{ \frac{k-\rk_M(F)}{\widehat{\alpha}(I(M/F))}  \right \}.
\]
Suppose \(\rk(M) = 1\).  Then $I(M)$ is generated by a subset of the variables so $\widehat{\rho}(I(M))=1=\frac{\rk(M)}{\widehat{\alpha}(I(M))}$, as claimed.  Now suppose $\rk(M)=k>1$, and that the statement holds for every matroid of rank $k-1$. Let $F\in \cL_1(M)$, so $\rk(M/F)=k-1$, and for $0 \le i \le k-1$, $\cL_{i}(M/F)=\{G \setminus F~:~ G\in \cL_{i+1}(M) \text{ and } F \subseteq G\}$.  By induction,
\[
\widehat{\rho}(I(M/F))=\max_{G\in \cL(M/F)\setminus \{E\setminus F\}}\left\lbrace \frac{k-1-\rk_{M/F}(G)}{\widehat{\alpha}(I((M/F)/G))}\right\rbrace=\max_{G\in \mathcal{L}(M)\setminus \{E\},F\subseteq G}\left\lbrace \frac{k-\rk_M(G)}{\widehat{\alpha}(I(M/G))}\right\rbrace.
\]   
Using \Cref{def:rhocontraction} and \Cref{prop:asym-single-contraction-matroid}, it follows that
\[
\widehat{\rho}_{c,1}(I(M))=\max\limits_{F\in \cL_1(M)}\widehat{\rho}(I(M/F))=\max_{F\in \cL(M)\setminus \{E\}, \rk(F)\ge 1}\left\lbrace \frac{k-\rk_M(F)}{\widehat{\alpha}(I(M/F))}\right\rbrace.
\]
Now suppose $\rk_M(F)=0$.  Then $F=\loops(M)$ so $k-\rk_M(F)=k$ and $M/F=M\setminus \loops(M)$. Observe that $I(M\setminus\loops(M))$ is generated by the same minimal generators as $I(M)$.  Hence, 
$
\frac{k-\rk_M(F)}{\widehat{\alpha}(I(M/F))}=\frac{k}{\widehat{\alpha}(I(M))}
$
and now the theorem follows by \Cref{prop:asym-single-contraction-matroid}. \end{proof}

\begin{Remark}\label{rmk:MatroidPolytope}
In light of \Cref{thm:asymptoticresurgencebyvaluations}, \Cref{thm:asym-res-facet} suggests that the Newton polyhedron $\NP(I(M))$ has a particularly nice structure.  This is indeed the case.  One can show that $\NP(I(M))$ is defined by inequalities that arise from the \textit{matroid polytope} of $M$ introduced in~\cite{GGM87}.  The matroid polytope of $M$ is the convex hull of the indicator vectors of the bases of $M$ (i.e. the convex hull of the exponent vectors of the minimal generators of $I(M)$), and its defining inequalities are well-known.  From these, one can deduce that the defining inequalities for the Newton polyhedron $\NP(I(M))$ of a rank $k$ matroid $M$ are of the form $\langle \chi_{E\setminus F}, \mathbf{a}\rangle\ge k-\rk_M(F)$ for all flats $F\in\cL(M)$.  Thus, the inward pointing normals to facets of $\NP(I(M))$ are contained among the indicator vectors of complements of flats of $M$.  Following this line of reasoning, \Cref{thm:asym-res-facet} can be proved using \Cref{thm:asymptoticresurgencebyvaluations}.  A similar phenomenon happens for the asymptotic resurgence of edge ideals of graphs -- see \cite[Remark~3.12]{DFMS19}.
\end{Remark}

\section{Equality and inequality in \GHV{}}\label{sec:GHVbound}

In this section, we set up a useful perspective for the remainder of the paper by posing the following question and answering it in a few cases.

\begin{Question}\label{ques:rhohat=alpha/alphahat}
Which matroids have a facet ideal which achieves equality in \GHV{}?  That is, for which matroids $M$ does the equality $\widehat{\rho}(I(M))=\frac{\alpha(I(M))}{\widehat{\alpha}(I(M))}$ hold?
\end{Question}

\noindent We will answer \Cref{ques:rhohat=alpha/alphahat} for a number of families of matroids in \Cref{sec:weakorder}, \Cref{sec:DesignsAndPMDs}, and \Cref{sec:asym-Steiner-system}.  In this section we collect a few initial examples of matroids for which we can answer \Cref{ques:rhohat=alpha/alphahat} using known results in the literature.

An important class of matroids which achieves equality in \GHV{} is the class of \textit{uniform matroids}.  The uniform matroid $\mathrm{U}_{k,n}$ is the rank $k$ matroid on the ground set $E$ of size $n$ whose bases are all subsets of $E$ with cardinality $k$.  Hence, the circuits of $\mathrm{U}_{k,n}$ are all subsets of $E$ with cardinality $k+1$ (provided that $k<n$).  Thus $I_{\Delta(\mathrm{U}_{k,n})}=I(\mathrm{U}_{k+1,n})$ if $k<n$.  The Waldschmidt constant, asymptotic resurgence (and even resurgence) are all implicitly known for facet ideals of uniform matroids.  We record these explicitly in the following proposition.

\begin{Proposition}\label{prop:uniform-matroid}
Let $\mathrm{U}_{k,n}$ be the uniform matroid of rank $k$ on a ground set of size $n$ with $k<n$.  Then
\begin{enumerate}
\item $\widehat{\alpha}(I_{\Delta(\mathrm{U}_{k,n})})=\dfrac{n}{n-k}$  \mbox{ and }  $\widehat{\alpha}(I(\mathrm{U}_{k,n}))=\dfrac{n}{n-k+1}$.
\item $\widehat{\rho}(I_{\Delta(\mathrm{U}_{k,n})})=\dfrac{(k+1)(n-k)}{n}=\dfrac{\alpha(I_{\Delta(\mathrm{U}_{k,n})})}{\widehat{\alpha}(I_{\Delta(\mathrm{U}_{k,n})})}$ \mbox{ and } $\widehat{\rho}(I(\mathrm{U}_{k,n}))=\dfrac{k(n-k+1)}{n}=\dfrac{\alpha(I(\mathrm{U}_{k,n}))}{\widehat{\alpha}(I(\mathrm{U}_{k,n}))}$.
\end{enumerate}
\end{Proposition}
\begin{proof}
Since $I_{\Delta(\mathrm{U}_{k,n})}=I(\mathrm{U}_{k+1,n})$ if $k<n$, then it suffices to prove the formulas for the Waldschmidt constant and the asymptotic resurgence of the Stanley-Reisner ideal of a uniform matroid. Part (a)  is given in \cite[Theorem~7.5]{BCGHJNSVV}, but also follows  from \cite[Lemma 2.4.1, Lemma 2.4.2, and the proof of Theorem 2.4.3]{BH10}.  In~\cite{MG18} it is shown that $\rho(I_{\Delta(\mathrm{U}_{k,n})})=\frac{(k+1)(n-k)}{n}=\frac{\alpha(I_{\Delta(\mathrm{U}_{k,n})})}{\widehat{\alpha}(I_{\Delta(\mathrm{U}_{k,n})})}$.  Since we always have $\frac{\alpha(I)}{\widehat{\alpha}(I)}\le \widehat{\rho}(I)\le \rho(I)$ for any ideal $I$, then (b) follows. 
\end{proof}

Another class of matroids whose facet ideals attain equality in \GHV{} is the class of rank-one and rank-two matroids.  If $M$ has rank one, then $I(M)$ is generated by a subset of the variables and so is a complete intersection.  In this case, $\alpha(I(M))=\widehat{\alpha}(I(M))=\widehat{\rho}(I(M))=1$.  For rank two matroids, we have the following.

\begin{Corollary}\label{lem:rank-two}
For any matroid $M$ of rank two we have 
\[
\widehat{\rho}(I(M))=\max\left\lbrace \frac{2}{\widehat{\alpha}(I(M))},\widehat{\rho}_{c,1}(I(M))\right\rbrace=\frac{2}{\widehat{\alpha}(I(M))}
.\]
\end{Corollary}
\begin{proof}
The proper flats of $M$ have rank zero or one.  If $F$ is a rank one flat of $M$, then $M/F$ has rank one, so $I(M/F)$ is generated by a subset of the variables.  Hence $k-\rk_M(F)=2-1=1$ and $\widehat{\alpha}(I(M/F))=1$.  If $\rk_M(F)=0$, then $F=\loops(M)$ and $M/F=M\setminus \loops(M)$.  Since $\widehat{\alpha}(I(M))\le \alpha(I(M))= 2$, it follows from \Cref{thm:asym-res-facet} that
$
\widehat{\rho}(I(M))=\frac{2}{\widehat{\alpha}(I(M))},
$
as claimed.
\end{proof}

\begin{Remark}\label{rmk:edge-ideal}
The facet ideal of a rank two matroid is also the edge ideal of a graph, so \Cref{lem:rank-two} is also a consequence of \cite[Theorem~3.12]{DFMS19}.
\end{Remark}

The Waldschmidt constant (and, by \Cref{lem:rank-two}, the asymptotic resurgence) of the facet ideal of a rank two matroid admits a straightforward formula.

\begin{Proposition}\label{prop:asym-res-rank-two}
Let $M$ be a rank two matroid.  Then
\[
\widehat{\alpha}(I(M))=\dfrac{|\cL_1(M)|}{|\cL_1(M)|-1} \quad \mbox{and} \quad \widehat{\rho}(I(M))=\dfrac{2(|\cL_1(M)|-1)}{|\cL_1(M)|}.
\]
\end{Proposition}
\begin{proof}
The formula for $\widehat{\rho}(I(M))$ is immediate from the formula for $\widehat{\alpha}(I(M))$ and \Cref{lem:rank-two}, so we prove the formula for $\widehat{\alpha}(I(M))$.  If $a,b\in E$, then $\rk_M(\{a,b\})=2$ (and so, $\{a,b\}$ is a basis of $M$) if and only if $a$ and $b$ do not belong to a common rank one flat.  It follows that the bases of $M$ are in bijection with the edges of the complete $s=|\cL_1(M)|$-partite graph $G$ whose maximal independent sets coincide with the rank one flats of $M$. So, the facet ideal $I(M)$ is the edge ideal of $G$.  The result now follows from~\cite[Theorem~6.7(ii)]{BocciCooperGuardoHarbourneJanssenNagelSeceleanuVanTuylVu}, which states that $\widehat{\alpha}(I(G))=\dfrac{s}{s-1}$ for a complete $s$-partite graph.
\end{proof}

Due to \Cref{lem:rank-two}, if the facet ideal of a matroid does not achieve equality in the \GHV{} lower bound, then it must have rank at least three.  The following is the smallest such matroid.

\begin{Example}\label{ex:rhohatnotalphaoveralphahat}
Let $\cB=\{\{1,2,4\},\{1,3,4\},\{2,3,4\}\}$.  It is straightforward to check that the basis exchange axiom holds for $\cB$, hence $\cB$ forms a set of bases of a matroid $M$ of rank $3$, called an \textit{almost-uniform matroid} (see~\Cref{ss:almost-uniform}).  Note that $I(M)=(x_4)(x_1x_2,x_1x_3,x_2x_3)$.  It follows from \cite[Example 1.1]{DFMS19} that     
\[
\widehat{\alpha}(I(M))=\frac{5}{2} \text{ and } \widehat{\rho}(I(M))= \widehat{\rho}(I(\mathrm{U}_{2,3}))=\frac{4}{3} >\frac{6}{5} =\frac{\alpha(I(M))}{\widehat{\alpha}(I(M))}.
\]
In \Cref{ss:almost-uniform} we compute the Waldschmidt constant and asymptotic resurgence for the facet ideal of any almost-uniform matroid.
\end{Example}

Direct sums of matroids can be leveraged to construct additional examples of matroids $M$ for which $\widehat{\rho}(I(M))>\frac{\alpha(I(M))}{\widehat{\alpha}(I(M))}$.

\begin{Example}[Direct Sums] \label{ex:direct-sum}
Suppose $M=M_1\oplus M_2$ is a direct sum of matroids $M_1$ and $M_2$.  Write $I_1=I(M_1)$, $I_2=I(M_2)$, and $I=I(M)$.  It is straightforward to check that $I=I_1I_2$.  Observe that $I_1$, $I_2$ have minimal generators in disjoint sets of variables. Therefore, it follows from \cite[Proposition 3.5]{JKM22} that $I^{(s)}=I_1^{(s)}I_2^{(s)}$ for all $s \ge 1$ and  $\widehat{\rho}(I)=\max\{\widehat{\rho}(I_1),\widehat{\rho}(I_2)\}$.
Moreover, it is straightforward to check that
\[
\frac{\alpha(I)}{\widehat{\alpha}(I)}=\frac{\alpha(I_1)+\alpha(I_2)}{\widehat{\alpha}(I_1)+\widehat{\alpha}(I_2)}.
\]
Without loss of generality, suppose $\frac{\alpha(I_1)}{\widehat{\alpha}(I_1)}\le \frac{\alpha(I_2)}{\widehat{\alpha}(I_2)}$.  Then
\[
\frac{\alpha(I_1)}{\widehat{\alpha}(I_1)}\le\frac{\alpha(I_1)+\alpha(I_2)}{\widehat{\alpha}(I_1)+\widehat{\alpha}(I_2)}\le \frac{\alpha(I_2)}{\widehat{\alpha}(I_2)},
\]
and both inequalities are strict unless $\frac{\alpha(I_1)}{\widehat{\alpha}(I_1)}=\frac{\alpha(I_2)}{\widehat{\alpha}(I_2)}$. It follows that if $\frac{\alpha(I_1)}{\widehat{\alpha}(I_1)}<\frac{\alpha(I_2)}{\widehat{\alpha}(I_2)}$,
then
\[
\widehat{\rho}(I)=\max\{\widehat{\rho}(I_1), \widehat{\rho}(I_2)\}\ge \frac{\alpha(I_2)}{\widehat{\alpha}(I_2)}>\frac{\alpha(I_1)+\alpha(I_2)}{\widehat{\alpha}(I_1)+\widehat{\alpha}(I_2)}=\frac{\alpha(I)}{\widehat{\alpha}(I)}.
\]
For example if $M=\mathrm{U}_{2,3}\oplus \mathrm{U}_{2,4}$, then \Cref{prop:uniform-matroid} implies that
\[
\widehat{\rho}(I(M))=\max\left\lbrace \frac{2\cdot 2}{3},\frac{2\cdot 3}{4} \right\rbrace=\frac{3}{2}
\quad
\mbox{while}
\quad
\frac{\alpha(I(M))}{\widehat{\alpha}(I(M))}=\frac{4}{\frac{3}{2}+\frac{4}{3}}=\frac{24}{17}<\frac{3}{2}.
\]
\end{Example}

\section{Asymptotic resurgence via the weak order} \label{sec:weakorder}
In this section, we first prove that asymptotic resurgence of facet and Stanley-Reisner ideals of matroids respects the \textit{weak order} on matroids of the same rank (\Cref{thm:asym-order}).  It follows from \Cref{thm:asym-order} that the asymptotic resurgence of matroids that are greatest in the weak order bounds the asymptotic resurgence of all matroids of a given rank.  In particular, a non-uniform matroid is always less in the weak order than an \textit{almost-uniform} matroid.  We use \Cref{thm:asym-order} and \Cref{thm:asym-res-facet} to compute the asymptotic resurgence of (facet ideals of) almost-uniform matroids (\Cref{thm:relaxation-to-uniform}), thus resulting in an upper bound on the asymptotic resurgence of the facet ideals of non-uniform matroids (\Cref{cor:asym-upper-bound}). We finish the section by proving that the facet ideal of a matroid whose dual is paving and whose rank is at most half the size of the ground set achieves equality in \GHV{} (\Cref{thm:asym-res-dual-paving}).  Based on \Cref{thm:asym-res-dual-paving}, conjectures of Mayhew-Newman-Welsh-Whittle \cite{MNWW-2011}, and computational evidence up to rank $8$ (see \Cref{tab:master-full}), we conjecture that asymptotically almost all matroids have a facet ideal which achieves equality in \GHV{}.

\subsection{Asymptotic resurgence respects weak order.}

Our first result is that the asymptotic resurgence of both the facet ideal and the Stanley-Reisner ideal of a matroid respects the weak order on matroids of the same rank. 

\begin{Definition}\label{def: weak order}
Let $M$ and $N$ be matroids on the ground set $E$.  We say that the matroid $M$ is less than the matroid $N$ in the \textit{weak order}, written $M\preceq N$, if any set $U\subset E$ which is independent in $M$ is also independent in $N$. In particular, if $M$ and $N$ have the same rank, then $M\preceq N$ if and only if $\cB(M) \subseteq \cB(N)$. 
\end{Definition}

\begin{Theorem}\label{thm:asym-order}
Let $M_1$ and $M_2$ be matroids both of rank $k$ on the same ground set $E$ so that $M_1\preceq M_2$. Then $\widehat{\rho}(I(M_1)) \le \widehat{\rho}(I(M_2))$ and $\widehat{\rho}(I_{\Delta(M_1)})\le \widehat{\rho}(I_{\Delta(M_2)})$.
\end{Theorem}
\begin{proof}
Suppose $M_1$ and $M_2$ are matroids of the same rank $k$ on the same ground set $E=[n]$ so that $M_1\preceq M_2$ in the weak order.  We first prove that $\widehat{\rho}(I(M_1)) \le \widehat{\rho}(I(M_2))$.  Let $U$ be an independent set of $M_1$ so that $\rk_{M_1}(U) <k$.  Since  $\cB(M_1) \subseteq \cB(M_2)$, $U$ is also an independent set of $M_2$ with $\rk_{M_2}(U)=\rk_{M_1}(U)<k$. Note that $I(M_1):x^U \subseteq I(M_2):x^U$ as  $I(M_1) \subseteq I(M_2)$.   By \Cref{lem:quotient-subset}, we get $I(M_1/U) \subseteq I(M_2/U)$ for all $U \in \cI(M_1)$ with $\rk_{M_1}(U)<k$.   This implies that $I(M_1/U)^{(s)} \subseteq I(M_2/U)^{(s)}$ for all $s \ge 1$, and hence, $\widehat{\alpha}(I(M_1/U)) \ge \widehat{\alpha}(I(M_2/U))$.  Thus, $\frac{k-|U|}{\widehat{\alpha}(I(M_1/U))} \le \frac{k-|U|}{\widehat{\alpha}(I(M_2/U))}$ for every $U \in \cI(M_1)$ with $\rk_{M_1}(U)<k$. 
By \Cref{thm:asym-res-facet}, we get 
\begin{align*}
\widehat{\rho}(I(M_1))=\max\limits_{U \in \mathcal{I}(M_1),~|U|<k}\left\{ \frac{k-|U|}{\widehat{\alpha}(I(M_1/U))} \right \}
 &\le \max\limits_{U \in \mathcal{I}(M_1),~|U|<k}\left\{ \frac{k-|U|}{\widehat{\alpha}(I(M_2/U))} \right \}\\[7 pt]
& \le \max\limits_{U \in \mathcal{I}(M_2),~|U|<k}\left\{ \frac{k-|U|}{\widehat{\alpha}(I(M_2/U))} \right \}
 =\widehat{\rho}(I(M_2)).\end{align*}

Next we prove that $\widehat{\rho}(I_{\Delta(M_1)})\le \widehat{\rho}(I_{\Delta(M_2)})$.  Observe that if $M_1\preceq M_2$, then $M^*_1\preceq M^*_2$, as $M_1, M_2$  have the same rank.  From the previous part, $\widehat{\rho}(I(M_1^*))\le \widehat{\rho}(I(M_2^*))$.  Since $I(M_i^*)$ is the Alexander dual of $I_{\Delta(M_i)}$ for $i=1,2$, it follows from \Cref{prop:Alex duality facet} that $\widehat{\rho}(I_{\Delta(M_1)})\le \widehat{\rho}(I_{\Delta(M_2)})$.
\end{proof}

The same-rank assumption in \Cref{thm:asym-order} is essential. If this condition is relaxed, the result fails to hold, as demonstrated in the following example.

\begin{Example}
Let $\mathrm{U}_{k,n}$ be the uniform matroid of rank $k$ on a ground set  of size $n$. By \Cref{prop:uniform-matroid}, $\widehat{\rho}(I(\mathrm{U}_{k,n}))=\frac{k(n-k+1)}{n}$.

From the definition of weak order, if $1\le s\le t\le n$, then $\mathrm{U}_{s,n}\preceq \mathrm{U}_{t,n}$.  However, as a function of $k$, $\frac{k(n-k+1)}{n}$ is increasing for $0\le k\le (n+1)/2$ and decreasing for $(n+1)/2\le k\le n$.  So, if $(n+1)/2\le s<t\le n-1$, then $\mathrm{U}_{s,n}\preceq \mathrm{U}_{t,n}$ but $\widehat{\rho}(I(\mathrm{U}_{s,n}))=\frac{s(n-s+1)}{n}>\frac{t(n-t+1)}{n}=\widehat{\rho}(I(\mathrm{U}_{t,n}))$.
\end{Example}

\subsection{The case of almost-uniform matroids}\label{ss:almost-uniform}

If $\mathcal{M}$ is a (usually minor-closed) class of matroids, an \textit{almost}-$\mathcal{M}$ matroid $M$ on a ground set $E$ is one that satisfies that, for every $e\in E$, either $M/e\in \mathcal{M}$ or $M\backslash e\in \mathcal{M}$~\cite{KL-2002}.  In particular, a matroid $M$ of rank $k$ is \textit{almost-uniform} if, for every $e\in E$, either $M/e\cong \mathrm{U_{k-1,n-1}}$ or $M\backslash e=\mathrm{U_{k,n-1}}$.  A characterization of almost-uniform matroids appears in \cite{VN20}.

\begin{Proposition}\cite[Theorem~3.4]{VN20}\label{prop:almost-uniform-equivalence}
Let $M$ be a rank $k$ matroid on a ground set $E$.  The following are equivalent:
\begin{enumerate}
\item $M$ is almost-uniform.
\item The set of bases for $M$ consists of all $k$-subsets of $E$ except one.
\end{enumerate}
\end{Proposition}

We denote an almost uniform matroid of rank $k$ on a ground set of size $n$ by $\mathrm{U}_{k,n}^{-}$.  If we wish to specify the unique basis $C$ of $\mathrm{U}_{k,n}$ that is not in the bases of $\mathrm{U}^-_{k,n}$ (guaranteed by \Cref{prop:almost-uniform-equivalence}), we write $\mathrm{U}^{-}_{k,n}(C)$.  For our purpose, we do not consider $\mathrm{U}_{n,n}^-$ as $\mathrm{U}_{n,n}^-=\mathrm{U}_{0,n}$ and $I(\mathrm{U}_{0,n}) =S.$ So, whenever we write $\mathrm{U}_{k,n}^-$, we assume that $k<n.$

\begin{Proposition}\label{lem:relaxation-to-uniform}
Let $\mathrm{U}_{k,n}^-$ be an almost-uniform matroid of rank $k$ on the ground set $E$ of size $n$.  Then
\[
\widehat{\alpha}(I(\mathrm{U}_{k,n}^-))= \min\left\lbrace\frac{n-1}{n-k},\frac{n+1}{n-k+1}\right\rbrace=
\begin{cases}
\dfrac{n-1}{n-k} & \text{ if } 2k\le  {n+1},\\[10 pt]
\dfrac{n+1}{n-k+1} & \text{ if } 2k \ge {n+1}.
\end{cases}
\]
\end{Proposition}
\begin{proof}
Let $M= \mathrm{U}_{k,n}^-(C)$ and $I=I(M)$. When $k=1$, then $I(M)$ is generated by a subset of the variables and so is a complete intersection. Therefore, $\widehat{\alpha}(I(M)) =1=\frac{n-1}{n-1}=\min\{\frac{n-1}{n-1},\frac{n+1}{n}\}$. So, without loss of generality, we may assume that $2 \leq k <n.$   We use the symbolic polyhedron $\SP(I(M))$ to compute the Waldschmidt constant via \Cref{thm:SPWaldschmidt}.  Since the bases of $M^*$ are all subsets of size $n-k$ except $E\setminus C$,  $\C(M^*)=\{E  \setminus C\} \cup \mathcal{U}_{n,k,C}$, where $\mathcal{U}_{n,k,C}:=\{U\subset [n]~:~ |U|=n-k+1, E\setminus C\not\subset U\}$.  Thus, by \Cref{cor:SPIM}, the defining inequalities of $\SP(I)$ (inside $\RR^n_{\ge 0}$) are $\sum\limits_{i\in E\setminus C} a_i\ge 1$ and $\sum\limits_{i\in U} a_i\ge 1$ for all $U\in \mathcal{U}_{n,k,C}$.

Let $c\in C$ and define
\[
{\bf v}_c=\frac{1}{n-k}\sum_{i\neq c}{\bf e}_i.
\]
Observe that $\langle \chi_{E\setminus C},{\bf v}_c\rangle=1$ and for any  $U\in \mathcal{U}_{n,k,C}$,
\[
\langle \chi_U,{\bf v}_c\rangle =
\begin{cases}
1 & c\in U\\
1+\frac{1}{n-k} & c\notin U.
\end{cases}
\]
It follows that ${\bf v}_c\in \SP(I)$ for any $c\in C$.  Therefore, by \Cref{thm:SPWaldschmidt}, $\widehat{\alpha}(I)\le \frac{1}{n-k}\sum\limits_{i\neq c}1=\frac{n-1}{n-k}$.

Next, for $d \in E\setminus C$, we define
\[
{\bf u}_d=\frac{2}{n-k+1}{\bf e}_d+\sum_{i\neq d} \frac{1}{n-k+1} {\bf e}_i.
\]
Notice that $\langle \chi_{E\setminus C},{\bf u}_d\rangle=1$ and for any $U\in \mathcal{U}_{n,k,C}$,
\[
\langle \chi_U,{\bf u}_d\rangle =
\begin{cases}
1 & d\notin U\\
1+\frac{1}{n-k+1} & d\in U,
\end{cases}
\] which implies that ${\bf u}_d\in \SP(I)$ for any $d\in E\setminus C$. Thus, by \Cref{thm:SPWaldschmidt}, $\widehat{\alpha}(I)\le \frac{2}{n-k+1} + \sum\limits_{i\neq d}\frac{1}{n-k+1}=\frac{n+1}{n-k+1}$, and hence,  $\widehat{\alpha}(I)\le \min\left\lbrace\frac{n-1}{n-k},\frac{n+1}{n-k+1}\right\rbrace$.

For the inequality $\min\left\lbrace\frac{n-1}{n-k},\frac{n+1}{n-k+1}\right\rbrace\le \widehat{\alpha}(I)$, we define some auxiliary sets.  Since $U\in\mathcal{U}_{n,k,C}$ means that $|U|=n-k+1$ and $E\setminus C\not\subset U$, we must have $2\le |U\cap C|\le \min\{k,n-k+1\}$ for all $U\in\mathcal{U}_{n,k,C}$.  For any integer $2\le i\le \min\{k,n-k+1\}$, we set $\mathcal{U}_i:=\{U\in \mathcal{U}_{n,k,C}~:~ |U\cap C|=i\}$ and for any $j\in E$, we define $\mathcal{U}_i^{(j)}=\{U\in \mathcal{U}_i: j\in U\}$. 

Now, let ${\bf a}=(a_1,\ldots,a_n)\in \SP(I)$.  Therefore, by \Cref{def:SymbolicPolyhedron}, $\sum\limits_{j\in U}a_j\ge 1$ for all $U\in \mathcal{U}_i$. Consequently, $\sum\limits_{U \in \mathcal{U}_i}\sum\limits_{j\in U}a_j\ge \sum\limits_{U\in \mathcal{U}_i}1=|\mathcal{U}_i|$ which implies that $\sum\limits_{j\in E} |\mathcal{U}_i^{(j)}| a_j\ge |\mathcal{U}_i|$. 
Observe that $|\mathcal{U}_i|=\binom{k}{i}\binom{n-k}{n-k+1-i}$ for any $2\le i\le \min\{k,n-k+1\}$.  Observe also that
\[
|\mathcal{U}_i^{(j)}|=
\begin{cases}
\binom{k-1}{i-1}\binom{n-k}{n-k+1-i} & j\in C\\
\binom{k}{i}\binom{n-k-1}{n-k-i} & j\notin C
\end{cases},
\]
where we interpret $\binom{n-k-1}{n-k-i}=0$ if $i=n-k+1$.

Thus, for each $2\le i\le \min\{k,n-k+1\}$, we get the inequality
\begin{equation}\label{ineq:binom}
\binom{k-1}{i-1}\binom{n-k}{n-k+1-i}\left(\sum_{j\in C}a_j\right)+\binom{k}{i}\binom{n-k-1}{n-k-i}\left(\sum_{j\in E\setminus C} a_j\right)\ge \binom{k}{i}\binom{n-k}{n-k+1-i}.
\end{equation}
Observe that 
\begin{equation}\label{ineq:feasible}
\binom{k}{i}\binom{n-k-1}{n-k-i}\le \binom{k-1}{i-1}\binom{n-k}{n-k+1-i}
\end{equation}
is equivalent to $k(n-k+1)\le ni$, which holds when $i=\min\{k, n-k+1\}$. Choosing $i=\min\{k, n-k+1\}$, we may then add the inequality $\sum\limits_{i\in E\setminus C} a_i\ge 1$ to both sides of \eqref{ineq:binom} $\binom{k-1}{i-1}\binom{n-k}{n-k+1-i}-\binom{k}{i}\binom{n-k-1}{n-k-i}$ times to get the inequality
\begin{equation*}
\binom{k-1}{i-1}\binom{n-k}{n-k+1-i} \left(\sum_{i=1}^n a_i\right)\ge \binom{k}{i}\binom{n-k}{n-k+1-i}+\binom{k-1}{i-1}\binom{n-k}{n-k+1-i}-\binom{k}{i}\binom{n-k-1}{n-k-i}
\end{equation*}
which (since $0<i$ and $0<n-k$) reduces to
\begin{equation}\label{ineq:wald}
\sum_{i=1}^n a_i\ge \frac{k}{i}+1-\frac{k}{i}\cdot\frac{n-k+1-i}{n-k}.
\end{equation}
If $2k\le n+1$, then $i=\min\{k,n-k+1\}=k$, and thus, by \eqref{ineq:wald}, 
\[
\sum\limits_{i=1}^n a_i\ge 2-\frac{n-2k+1}{n-k}=\frac{n-1}{n-k},
\]
and hence, by \Cref{thm:SPWaldschmidt}, $\widehat{\alpha}(I)=\min\limits_{{\bf a}\in \SP(I)} \sum\limits_{i=1}^na_i \ge\frac{n-1}{n-k}$ when $k\le \frac{n+1}{2}$.  Since we already established that $\widehat{\alpha}(I)\le \frac{n-1}{n-k}$, we have $\widehat{\alpha}(I)=\frac{n-1}{n-k}$ when $2k\le n+1$.

Finally, if $2k\ge n+1$, then $i=\min\{k,n-k+1\}=n-k+1$, and thus, by  
\eqref{ineq:wald}, we get
\[
\sum_{i=1}^n a_i\ge \frac{k}{n-k+1}+1=\frac{n+1}{n-k+1}.
\] 
Hence, by \Cref{thm:SPWaldschmidt}, $\widehat{\alpha}(I)\ge \frac{n+1}{n-k+1}$ when $2k\ge  n+1$.  Since we already established that $\widehat{\alpha}(I)\le \frac{n+1}{n-k+1}$, we have $\widehat{\alpha}(I)=\frac{n+1}{n-k+1}$ when $2k\ge n+1$.
\end{proof}

\begin{Theorem}
    \label{thm:relaxation-to-uniform} Let $\mathrm{U}_{k,n}^-$ be an  almost-uniform matroid of rank $k$ on the ground set $E$ of size $n$.   Then 
    \[
    \widehat{\rho}(I(\mathrm{U}_{k,n}^-))=\max\left\{\frac{(k-1)(n-k+1)}{n-1},\frac{k(n-k)}{n-1}\right\}
    =\begin{cases}
     \dfrac{k(n-k)}{n-1} & \text{ if } 2k\le {n+1}\\[10 pt]
    \dfrac{(k-1)(n-k+1)}{n-1} & \text{ if } 2k\ge n+1.
    \end{cases}
    \]
\end{Theorem}
\begin{proof}
Let $M=\mathrm{U}_{k,n}^-$ and $I=I(M)$. When $k=1$, then $\widehat{\rho}(I(M))=1$, since  $I(M)$ is generated by a subset of the variables. So, without loss of generality, we may assume that $2 \leq k <n.$
Let $C$ be the $k$ size subset of $E$ so that $M=\mathrm{U}_{k,n}^-(C)$. Let $e\in E\setminus C$.  
Let $A \subset E \setminus \{e\}$ be a ($k-1$)-element subset. Then,  $A \cup \{e\}$ is a $k$-element subset of $E$ and $A \cup \{e\} \neq C$ as $e \not\in C$. Therefore, $A \cup\{e\} \in \cB(M)$ which implies $A \in \cB(M/e)$. So, every subset of size $k-1$ is a basis of $M/e$, and hence, $M/e\cong \mathrm{U}_{k-1,n-1}$.  

Observe that for $f \in E$, $M/f$ is a rank $k-1$ matroid on ground set $E\setminus \{f\}$.  Therefore, by~\Cref{thm:asym-order}, $\widehat{\rho}(I(M/f)) \le \widehat{\rho}(I(\mathrm{U}_{k-1,n-1}))$ for every $f \in E$.  Since $M/e \cong \mathrm{U}_{k-1,n-1}$, then $$\widehat{\rho}_{c,1}(I(M))=\max\{\widehat{\rho}(I(M/f))~:~f\in E\}= \widehat{\rho}(I(\mathrm{U}_{k-1,n-1}))=\frac{(k-1)(n-k+1)}{n-1},$$  where the last equality follows from \Cref{prop:uniform-matroid}. Now, by~\Cref{prop:asym-single-contraction-matroid},  \begin{align*} \widehat{\rho}(I(M)) & =\max\left\lbrace \frac{k}{\widehat{\alpha}(I(M))},\widehat{\rho}_{c,1}(I(M))\right\rbrace \\ & =\max\left\{\frac{k}{\widehat{\alpha}(I(M))}, \frac{(k-1)(n-k+1)}{n-1}\right\}=\begin{cases}
     \dfrac{k(n-k)}{n-1} & \text{ if } 2k\le {n+1}\\
    \dfrac{(k-1)(n-k+1)}{n-1} & \text{ if } 2k\ge n+1,
    \end{cases}\end{align*}
where the last equality follows by \Cref{lem:relaxation-to-uniform}.
\end{proof}

\subsection{An upper bound for the asymptotic resurgence of non-uniform matroids}
Recall that formulas for the asymptotic resurgence of the Stanley-Reisner and facet ideals of uniform matroids are given in \Cref{prop:uniform-matroid}.    In this section, we consider non-uniform matroids and obtain upper bounds for the asymptotic resurgence using  \Cref{thm:asym-order} and \Cref{thm:relaxation-to-uniform}.

\begin{Corollary}\label{cor:asym-upper-bound}
Let $M$ be a non-uniform matroid of rank $k$ on the ground set $E$ of size $n$. Then, 
\[\widehat{\rho}(I(M)) \le 
   \begin{cases}
     \dfrac{k(n-k)}{n-1} &  \text{if } 2k\le n+1\\[10 pt]
    \dfrac{(k-1)(n-k+1)}{n-1} & \text{if } 2k\ge n+1
    \end{cases}
    \text{ and }~  \widehat{\rho}(I_{\Delta(M)}) \le 
   \begin{cases}
     \dfrac{k(n-k)}{n-1} &  \text{if } n\le 2k+1\\[10 pt]
    \dfrac{(n-k-1)(k+1)}{n-1} & \text{if } n\ge 2k+1.
    \end{cases}
  \]
\end{Corollary}
\begin{proof}
Since $M$ is a non-uniform matroid of rank $k$ on the ground set $E$ of size $n$, $k<n$ and there is some $k$-element subset $C$ of $E$ which is not a basis. So, $M\preceq \mathrm{U}_{k,n}^-(C)$. 
 Then,  by \Cref{thm:asym-order} and \Cref{thm:relaxation-to-uniform},  \[\widehat{\rho}(I(M)) \le 
   \begin{cases}
     \frac{k(n-k)}{n-1} &  \text{if } 2k\le n+1\\
    \frac{(k-1)(n-k+1)}{n-1} & \text{if } 2k\ge n+1.
    \end{cases}
  \] Since $M\preceq \mathrm{U}_{k,n}^-(C)$, 
  we also have $M^*\preceq \mathrm{U}_{k,n}^-(C)^*$. Observe that $ \mathrm{U}_{k,n}^-(C)^*=\mathrm{U}_{n-k,n}^-(E\setminus C)$ and $n-k<n$. 
  So, by \Cref{thm:asym-order} and \Cref{thm:relaxation-to-uniform},  \[\widehat{\rho}(I(M^*)) \le 
   \begin{cases}
     \frac{k(n-k)}{n-1} &  \text{if } n\le 2k+1\\
    \frac{(n-k-1)(k+1)}{n-1} & \text{if } n\ge 2k+1.
    \end{cases}
  \] Now, by \Cref{prop:Alex duality facet},  $\widehat{\rho}(I_{\Delta(M)})=\widehat{\rho}(I(M^*))$. This completes the proof. 
\end{proof}

The upper bound in \Cref{cor:asym-upper-bound} is sharp for a broad class of matroids, including those admitting a single-element contraction that is uniform; see also Tables~\ref{tbl:Rank3groundset5} -- \ref{tbl:Rank3groundset7}.  
When no such contraction exists, the bound can be further improved, as shown in the following theorem.

\begin{Theorem}\label{thm:asym-res-uniform-contraction}
Let $M$ be a non-uniform matroid of rank $k$ on the ground set $E$ of size $n$ with $k \geq 2$.
\begin{enumerate}
    \item  If $M/e \cong \mathrm{U}_{k-1,n-1}$ for some $e \in E$, then $$\widehat{\rho}_{c,1}(I(M))=\frac{(k-1)(n-k+1)}{n-1} \quad \text{ and } \quad \widehat{\rho}(I(M))= \max\left\{\frac{(k-1)(n-k+1)}{n-1},\frac{k}{\widehat{\alpha}(I(M))}\right\}.$$ In particular, if $2k \ge n+1$, then $\widehat{\rho}(I(M))= \frac{(k-1)(n-k+1)}{n-1}$.
    \item If $M/e \not\cong \mathrm{U}_{k-1,n-1}$ for all $e \in E$, then $$\widehat{\rho}_{c,1}(I(M))\le \begin{cases}
     \dfrac{(k-1)(n-k)}{n-2} &  \text{if } 2k\le n+2\\[10 pt]
    \dfrac{(k-2)(n-k+1)}{n-2} & \text{if } 2k\ge n+2.
    \end{cases}$$
\end{enumerate}
\end{Theorem}
\begin{proof}
Since $M$ is a non-uniform matroid, we have $k<n$. Observe that for $e \in E \setminus \loops(M)$, $M/e$ is a rank $k-1$ matroid on the ground set $E\setminus \{e\}$.  Therefore, by~\Cref{thm:asym-order}, $\widehat{\rho}(I(M/e)) \le \widehat{\rho}(I(\mathrm{U}_{k-1,n-1}))$ for every $e \in E \setminus \loops(M)$.  Since $M/e \cong \mathrm{U}_{k-1,n-1}$ for some $e \in E$, $$\widehat{\rho}_{c,1}(I(M))=\max\{\widehat{\rho}(I(M/e))~:~e\in E \setminus \loops(M)\}= \widehat{\rho}(I(\mathrm{U}_{k-1,n-1}))=\frac{(k-1)(n-k+1)}{n-1},$$ where the last equality follows from \Cref{prop:uniform-matroid}. Now, by~\Cref{prop:asym-single-contraction-matroid},  $$\widehat{\rho}(I(M))=\max\left\{\frac{(k-1)(n-k+1)}{n-1},\frac{k}{\widehat{\alpha}(I(M))}\right\}.$$ 

 Since $M$ is non-uniform, there is a $k$-subset $C$ of $E$ so that  $M \preceq \mathrm{U_{k,n}^-}(C)$. Therefore, $I(M) \subseteq I(\mathrm{U}_{k,n}^-(C))$, and hence,  $\widehat{\alpha}(I(M)) \ge \widehat{\alpha}(I(\mathrm{U}_{k,n}^-(C)))$.  When $2k \ge n+1$,  \Cref{lem:relaxation-to-uniform} implies that $$\frac{k}{\widehat{\alpha}(I(M))}\le \frac{k}{\widehat{\alpha}(I(\mathrm{U}_{k,n}^-(C)))}=\frac{k(n-k+1)}{n+1} \le  \frac{(k-1)(n-k+1)}{n-1},$$ and the assertion in part (a) follows. 
Finally, part (b) follows from \Cref{cor:asym-upper-bound}. 
\end{proof}

\subsection{Asymptotic resurgence and paving matroids}
A matroid $M$ of rank $k$ is \emph{paving} if every circuit of $M$ has size at least $k$. 
Equivalently, every circuit of $M$ has size $k$ or $k+1$. A matroid $M$ is \emph{sparse paving} if both $M$ and $M^*$ are paving.  Almost-uniform matroids are paving, and the dual of an almost-uniform matroid is another almost-uniform matroid.  Hence, almost-uniform matroids are sparse paving matroids.  We show that \Cref{thm:relaxation-to-uniform} extends verbatim to certain sparse paving matroids of rank $k$ on $n$ elements.

There is a well-known characterization of sparse paving matroids in terms of their \textit{circuit-hyperplanes} which is very useful.  Suppose $M$ is a matroid of rank $k$ on a ground set $E$.  A \textit{hyperplane} of $M$ is a flat of rank $k-1$.  A \textit{circuit-hyperplane} of $M$ is a subset of $E$ that is both a circuit of $M$ and a hyperplane of $M$.  We denote the set of circuit-hyperplanes of $M$ by $\mathcal{CH}(M)$.  Equivalent versions of the following characterization have been well-known for quite some time (e.g. \cite{Knuth-1974}); we use one from \cite{MPTSZ}.

\begin{Proposition}[{\cite[Lemma~2 and Proposition~3]{MPTSZ}}]\label{prop:SparsePavingCharacterization}
For a set $E$ of cardinality $n$ and  $2\le k\le n$, a collection $\mathcal{CH}$ of $k$-element subsets of $E$ is the set of circuit-hyperplanes of a sparse paving matroid of rank $k$ if and only if for every $C,C'\in\mathcal{CH}$, $|C\cap C'|\le k-2$.
\end{Proposition}

For the remainder of the section, for $i\in\ZZ$, we put $\C_i(M):=\{C\in \C(M): |C|=i\}$.  We record some basic facts about circuit-hyperplanes and sparse paving matroids in the following lemma.

\begin{Lemma}\label{lem:sparsePavingCircuitHyperplanes}
If $M$ is a matroid on a ground set $E$, then $\mathcal{CH}(M^*)=\{E\setminus C: C\in \mathcal{CH}(M)\}$.  If $M$ is sparse paving of rank $k$, then $\C_k(M)=\mathcal{CH}(M)$; in particular $\C_{n-k}(M^*)=\{E\setminus C: C\in \C_{k}(M)\}$.
\end{Lemma}
\begin{proof}
First suppose $M$ is a matroid and $U$ is a $k$-subset of $E$.  Then $U$ is a circuit of $M$ if and only if $E\setminus U$ is a hyperplane of $M^*$ and $U$ is a hyperplane of $M$ if and only if $E\setminus U$ is a circuit of $M^*$.  So, the first equality follows.  Now suppose $M$ is sparse paving and $C\in \C_k(M)$.  It suffices to show that $C$ is a hyperplane of $M$, or equivalently that $E\setminus C\in \C_{n-k}(M^*)$.  The latter follows because $E\setminus C$ is not a basis for $M^*$ (as $C$ is not a basis of $M$), and every $(n-k)$-subset of $E$ is either a circuit or a basis for $M^*$ (since $M^*$ is paving of rank $n-k$).  The final equality $\C_{n-k}(M^*)=\{E\setminus C: C\in \C_{k}(M)\}$ follows from the first two parts.
\end{proof}

\begin{Proposition}\label{cor:dual-paving-not-covering}
Let $M$ be a non-uniform matroid of rank $k$ on the ground set $E$ of size $n$ and let $A(M)=\bigcup\limits_{C\in \C_{\rk(M)}(M)}C$.
\begin{enumerate}
    \item If $M$ is paving, $2k> n+1$, and $A(M)\neq E$,
    then $\widehat{\rho}(I(M))=\frac{(k-1)(n-k+1)}{n-1}$.  
\item If $M^*$ is paving, $2k> n+1$, and $\bigcap\limits_{C \in \C_{n-k}(M^*)}C \neq \emptyset$,
then $\widehat{\alpha}(I(M))=\widehat{\alpha}(I(\mathrm{U}^{-}_{k,n}))=\frac{n+1}{n-k+1}$. 
\item If $M^*$ is paving, $2k\le n+1$, and $A(M^*)\neq E$, then $\widehat{\alpha}(I(M))=\widehat{\alpha}(I(\mathrm{U}^{-}_{k,n}))=\frac{n-1}{n-k}$\\ and $\widehat{\rho}(I(M))=\frac{k}{\widehat{\alpha}(I(M))}=\frac{k(n-k)}{n-1}$. 
\end{enumerate}

\noindent In particular, if $M$ is sparse paving and $A(M)\neq E$ and $A(M^*)\neq E$, then
$\widehat{\alpha}(I(M))=\widehat{\alpha}(I(\mathrm{U}_{k,n}^-))$ and $\widehat{\rho}(I(M))=\widehat{\rho}(I(\mathrm{U}_{k,n}^-))$.
\end{Proposition}

\begin{proof}
Let $M$ be a paving matroid of rank $k$, where $2k>n+1$.  Suppose that $A(M)\neq E$.  Then, there is an element $e\in E$ which is not in any of the circuits of $M$ of size $k$.  Let $A \subset E\setminus \{e\}$ so that $|A|=k-1$. Since $M$ is paving, $A \cup \{e\}$ is a basis of $M$, and hence, $A \in \cB(M/e)$. It follows that $M/e\cong \mathrm{U}_{k-1,n-1}$.  Hence, $\widehat{\rho}(I(M))=\frac{(k-1)(n-k+1)}{n-1}$ by \Cref{thm:asym-res-uniform-contraction} (a).  This completes the proof of part (a).

Now, suppose that $M^*$ is paving.  Since $M\preceq \mathrm{U}_{k,n}^-$, $I(M)\subseteq I(U^{-}_{k,n})$ and so $\widehat{\alpha}(I(M))\ge \widehat{\alpha}(I(\mathrm{U}_{k,n}^-))$.  For (b) and the first part of (c), we need to show that $\widehat{\alpha}(I(M))\le \widehat{\alpha}(I(\mathrm{U}_{k,n}^-))$ under the given hypotheses.

First, suppose that $\bigcap\limits_{C \in \C_{n-k}(M^*)}C \neq \emptyset$ and $2k>n+1$.  Let $d\in \bigcap\limits_{C\in \mathcal{C}_{n-k}(M^*)} C$.  We now follow the same argument as in the proof of \Cref{lem:relaxation-to-uniform}.  Define 
\[
{\bf u}_d=\frac{2}{n-k+1}{\bf e}_d+\sum_{i\neq d} \frac{1}{n-k+1} {\bf e}_i.
\]
We claim ${\bf u}_d\in \SP(I(M))$.  If $C\in \C_{n-k}(M^*)$, then $d\in C$ and so
$
\langle \chi_C,{\bf u}_d\rangle=1.
$
If $C\in\C_{n-k+1}(M^*)$, then 
$
\langle \chi_C,{\bf u}_d\rangle\ge 1.
$
As $M^*$ is paving of rank $n-k$ (so every circuit of $M^*$ has cardinality $n-k$ or $n-k+1$), we deduce that ${\bf u}_d\in \SP(I(M))$ by \Cref{cor:SPIM}.  By \Cref{thm:SPWaldschmidt}, $\widehat{\alpha}(I(M))\le \frac{2}{n-k+1} + \sum\limits_{i\neq d}\frac{1}{n-k+1}=\frac{n+1}{n-k+1}$.  As $\widehat{\alpha}(I(\mathrm{U}^-_{k,n}))=\frac{n+1}{n-k+1}$ when $2k>n+1$ by \Cref{lem:relaxation-to-uniform}, this completes the proof of (b).

Now, suppose that $A(M^*)\neq E$ and $2k\le n+1$.  Then, there is some $c\in E$ so that $c\notin C$ for all $C\in \C_{n-k}(M^*)$.  We claim that ${\bf v}_c=\frac{1}{n-k}\sum\limits_{i\neq c}{\bf  e}_i \in \SP(I(M))$.   If $C\in \C_{n-k}(M^*)$ then $\langle \chi_C,{\bf v}_c\rangle=1$.  If $C\in \C_{n-k+1}(M^*)$ then $\langle \chi_C,{\bf v}_c\rangle\ge 1$.  As $M^*$ is paving, we conclude that ${\bf v}_c\in \SP(I(M))$ and therefore, by \Cref{thm:SPWaldschmidt}, $\widehat{\alpha}(I(M))\le \frac{n-1}{n-k}$.  As $\widehat{\alpha}(I(\mathrm{U}^-_{k,n}))=\frac{n-1}{n-k}$ when $2k\le n+1$ by \Cref{lem:relaxation-to-uniform}, we obtain that $\widehat{\alpha}(I(M))=\widehat{\alpha}(I(\mathrm{U}^-_{k,n}))$, as desired.  

To complete the proof of (c), observe that since $M$ is a non-uniform matroid, by \Cref{cor:asym-upper-bound}, $\widehat{\rho}(I(M)) \le \frac{k(n-k)}{n-1}$.   By \GHV{}, $\widehat{\rho}(I(M)) \ge \frac{k}{\widehat{\alpha}(I(M))}=\frac{k(n-k)}{n-1}$, so $\widehat{\rho}(I(M))=\frac{k(n-k)}{n-1}$.

For the final statement, suppose that $M$ is sparse paving and assume further that $A(M)\neq E$ and $A(M^*)\neq E$.  If $2k>n+1$, then $\widehat{\rho}(I(M))=\widehat{\rho}(I(\mathrm{U}_{k,n}^-))$ follows from part (a).  Furthermore, if $2k\le n+1$, both $\widehat{\alpha}(I(M))=\widehat{\alpha}(I(\mathrm{U}_{k,n}^-))$ and $\widehat{\rho}(I(M))=\widehat{\rho}(I(\mathrm{U}_{k,n}^-))$ follow from part (c).  It remains to prove that, if $2k>n+1$, then $\widehat{\alpha}(I(M))=\widehat{\alpha}(I(\mathrm{U}_{k,n}^-))$.  Since $A(M)\neq E$, it follows that $E\setminus A(M)\neq \emptyset$.  But $E\setminus A(M)=E\setminus \bigcup_{C\in \C_{k}(M)} C=\bigcap_{C\in \C_k(M)}(E\setminus C)=\bigcap_{C\in \C_{n-k}(M^*)}C$, where the last equality follows from \Cref{lem:sparsePavingCircuitHyperplanes}.  Hence, by part (b), $\widehat{\alpha}(I(M))=\widehat{\alpha}(I(\mathrm{U}_{k,n}^-))$ and the proof is complete.
\end{proof}

\begin{Remark}
It is straightforward to construct sparse paving matroids satisfying the conditions of \Cref{cor:dual-paving-not-covering}.  For instance, given a set $E$ of size $n$, we can pick two elements $a,b\in E$ and a collection $\mathcal{CH}=\{C_1,\ldots,C_\ell\}$ of subsets of cardinality $k<n$ (where $k \geq 2)$ satisfying that $|C_i\cap C_j|\le k-2$ for all $1\le i<j\le n$, $a\in \cap_{i=1}^\ell C_i$ and $b\notin\cup_{i=1}^\ell C_i$.  Then $\mathcal{CH}$ is the collection of circuit-hyperplanes of a rank $k$ sparse paving matroid $M$ (by \Cref{prop:SparsePavingCharacterization}) and $E$ is not covered by the circuit-hyperplanes of $M$ or $M^*$.
\end{Remark}

We now prove that if $M$ is a matroid of rank $k$ on a ground set of size $n \ge 2k$ whose dual is paving, then the asymptotic resurgence of $I(M)$ achieves \GHV{}. That is, we can drop the hypothesis $A(M^*)\neq E$ in part (c) of \Cref{cor:dual-paving-not-covering} and retain the conclusion that the Waldschmidt constant of $I(M)$ determines its asymptotic resurgence.  We begin by establishing an upper bound on the Waldschmidt constant of the facet ideal of matroids whose dual is paving.

\begin{Lemma}\label{lem:walds-bound-dual-paving}
Let $M$ be a non-uniform matroid of rank $k$ on the ground set $E$ of size $n$ so that $M^*$ is paving. Then, \[\widehat{\alpha}(I(M)) \le \frac{n}{n-k}.\] Moreover, if $M/e\cong  \mathrm{U}_{k-1,n-1}$ for some $e\in E$, and $n +1 \ge 2k$, then $\widehat{\alpha}(I(M)) \le \frac{(n-1)k}{(k-1)(n-k+1)}$.    
\end{Lemma}
\begin{proof}
As in the proofs of \Cref{lem:relaxation-to-uniform} and \Cref{cor:dual-paving-not-covering}, we use $\SP(I(M))$.  In particular, we claim that ${\bf v}=\frac{1}{n-k}\sum\limits_{i=1}^n{\bf  e}_i \in \SP(I(M))$.  Since $M^*$ is a paving matroid of rank $n-k$, $|C| \ge n-k$ for all $C \in \mathcal{C}(M^*)$.  So,  $\langle \chi_{C},{\bf v}\rangle =\frac{|C|}{n-k} \ge 1$ for all $C \in \mathcal{C}(M^*)$. Therefore, ${\bf v} \in \SP(I(M))$ by \Cref{cor:SPIM}, and hence, by \Cref{thm:SPWaldschmidt}, $\widehat{\alpha}(I(M)) \le \frac{n}{n-k}$. 

Suppose $M/f\cong  \mathrm{U}_{k-1,n-1}$ for some $f \in E$ and $n+1 \ge 2k$. Note that  $(M/f)^*=M^* \setminus f \cong \mathrm{U}_{n-k,n-1}$, i.e.,  the circuits of $M^*$ that do not contain $f$ are precisely the $n-k+1$ size subsets of $E\setminus \{f\}$. We claim that ${\bf u}_f=\frac{1}{n-k+1}\sum\limits_{i\neq f} {\bf e}_i+\frac{n-1}{(n-k+1)(k-1)} {\bf e}_f \in \SP(I(M))$.   Let $C \in \mathcal{C}(M^*)$. If $f \not\in C$, then $|C|=n-k+1$ and $\langle\chi_C,{\bf u}_f\rangle =\frac{|C|}{n-k+1}=1$. If $f \in C$, then $$\langle\chi_C,{\bf u}_f\rangle =\frac{|C|-1}{n-k+1}+\frac{n-1}{(n-k+1)(k-1)} \ge \frac{n-k-1}{n-k+1}+\frac{n-1}{(n-k+1)(k-1)} = \frac{k(n-k)}{(n-k+1)(k-1)}\ge  1,$$  where the last inequality follows as $n+1 \ge 2k$. This implies ${\bf u}_f \in \SP(I(M)),$ and hence, by \Cref{thm:SPWaldschmidt}, $\widehat{\alpha}(I(M)) \le \frac{(n-1)k}{(k-1)(n-k+1)}$. 
\end{proof}

We now prove that the facet ideal of a matroid whose dual is paving achieves the GHV lower bound whenever $n \ge 2k$. 
 
\begin{Theorem}\label{thm:asym-res-dual-paving}
Let $M$ be a non-uniform matroid of rank $k$ on the ground set $E$ of size $n$ so that $M^*$ is paving.  If $n \ge 2k$, then $I(M)$ achieves equality in \GHV{}.  That is,  
\[
\widehat{\rho}(I(M))=\frac{k}{\widehat{\alpha}(I(M))}.
\]
\end{Theorem}
 \begin{proof} 
 The result for $k=1,2$ follows from \Cref{lem:rank-two} and the discussion above it. So, we assume $k \geq 3.$ Suppose that $M/e\cong  \mathrm{U}_{k-1,n-1}$ for some $e \in E$.  Then, since $n\geq 2k-1,$ by \Cref{thm:asym-res-uniform-contraction} and \Cref{lem:walds-bound-dual-paving}, 
\[
\widehat{\rho}(I(M))= \max\left\{\frac{(k-1)(n-k+1)}{n-1},\frac{k}{\widehat{\alpha}(I(M))}\right\}=\frac{k}{\widehat{\alpha}(I(M))}.
\] 
Next, suppose that $M/e \not\cong  \mathrm{U}_{k-1,n-1}$ for all $e \in E$. By \Cref{thm:asym-res-uniform-contraction}, we have $\widehat{\rho}_{c,1}(I(M))\le \frac{(k-1)(n-k)}{n-2}$ since $n\ge 2k-2$. On the other hand, $\frac{(k-1)(n-k)}{n-2} \le \frac{k(n-k)}{n} \le \frac{k}{\widehat{\alpha}(I(M))}$, where the first inequality holds if and only if $n \ge 2k$ and the second inequality follows from \Cref{lem:walds-bound-dual-paving}. Hence, by \Cref{prop:asym-single-contraction-matroid} $\widehat{\rho}(I(M))=\frac{k}{\widehat{\alpha}(I(M))}$. \qedhere
\end{proof}

\begin{Remark}
\Cref{thm:asym-res-dual-paving} is sharp in the sense that if $M$ is a matroid of rank $k$ on $n$ elements with a paving dual and $n\le 2k-1$, $I(M)$ does not necessarily achieve equality in \GHV{}.  Observe that, by \Cref{cor:dual-paving-not-covering}, sparse paving matroids of rank $k$ on $n$ elements with $n<2k-1$ whose circuits of cardinality $k$ do not cover the ground set (such as $\mathrm{U}_{k,n}^{-}$) do not achieve equality in \GHV{}.  In \Cref{app:sharp}, we further prove that for any ground set on $2k-1$ elements, with $k\ge 5$, there is a sparse paving matroid $M_{k,2k-1}$ of rank $k$ so that $I(M_{k,2k-1})$ does not achieve equality in \GHV{}.
\end{Remark}

We conclude the section with a conjecture about the asymptotic resurgence of the facet ideals of `most' matroids.  We follow the terminology of \cite{MNWW-2011} to make this precise.  Let $u(n)$ denote the number of \textit{unlabeled} matroids on $n$ elements.  A \textit{matroid property} $\mathcal{P}$ is a class of matroids closed under isomorphism.  Let $\mathcal{P}_u(n)$ denote the set of all unlabeled matroids on $n$ elements with property $\mathcal{P}$.  We say that `asymptotically almost every matroid is $\mathcal{P}$' provided that $\lim\limits_{n\to\infty}\frac{|\mathcal{P}_u(n)|}{u(n)}=1$.

Clearly, the initial degree, Waldschmidt constant, and asymptotic resurgence of facet ideals of matroids are all invariant under matroid isomorphism, so equations involving these constants yield a matroid property.

\begin{Conjecture}\label{conj:asymptoticresurgencemostmatroids} 
Asymptotically almost every matroid has a facet ideal which achieves equality in \GHV{}. More precisely, given a matroid $M$, we say that $M$ has the matroid property $\mathcal{ASM}$ if $\widehat{\rho}(I(M))=\frac{\rk(M)}{\widehat{\alpha}(I(M))}$.  Then
\[
\lim_{n\to\infty}\frac{|\mathcal{ASM}_u(n)|}{u(n)}=1.
\]
\end{Conjecture}

We close this section by showing that if $\lim\limits_{n\to\infty}\frac{|\mathcal{ASM}_u(n)|}{u(n)}$ exists, then \Cref{conj:asymptoticresurgencemostmatroids} is implied by two conjectures in the literature~\cite[Conjectures 1.6, 1.10]{MNWW-2011}.

\begin{Proposition}
Suppose that $\lim\limits_{n\to\infty}\frac{|\mathcal{ASM}_u(n)|}{u(n)}$ exists.  Then \Cref{conj:asymptoticresurgencemostmatroids} is true provided that the following two conjectures are true:
\begin{enumerate}
\item[(1)] Asymptotically almost every matroid is sparse paving. (see discussion following \cite[Conjecture~1.6]{MNWW-2011}.
\item[(2)] Asymptotically almost every matroid on $n$ elements has rank between $(n-1)/2$ and $(n+1)/2$ \cite[Conjecture~1.10]{MNWW-2011}.
\end{enumerate}
\end{Proposition}
\begin{proof}
A straightforward computation with limits yields that if $\mathcal{P}$ and $\mathcal{Q}$ are two matroid properties so that asymptotically almost every matroid is $\mathcal{P}$ and asymptotically almost every matroid is $\mathcal{Q}$, then asymptotically almost every matroid is both $\mathcal{P}$ and $\mathcal{Q}$.  That is to say, if the conjectures (1) and (2) are true, then asymptotically almost every matroid on $n$ elements is sparse paving and has rank between $(n-1)/2$ and $(n+1)/2$.  In particular, if conjectures (1) and (2) are true then asymptotically almost every matroid on $2k$ elements is sparse paving of rank $k$. By~\Cref{thm:asym-res-dual-paving}, asymptotically almost every matroid on $2k$ elements has a facet ideal which achieves equality in \GHV{}.  Thus, assuming conjectures (1) and (2),
\[
\lim_{k\to\infty}\frac{|\mathcal{ASM}_u(2k)|}{u(2k)}=1.
\]
Hence, if $\lim_{n\to\infty}\frac{|\mathcal{ASM}_u(n)|}{u(n)}$ exists, then \Cref{conj:asymptoticresurgencemostmatroids} follows.
\end{proof}

\section{Designs and perfect matroid designs}\label{sec:DesignsAndPMDs}
In this section and \Cref{sec:asym-Steiner-system}, we show that two classes of matroids that have quite a bit of symmetry do satisfy $\widehat{\rho}(I(M))=\frac{\alpha(I(M))}{\widehat{\alpha}(I(M))}$.  The symmetry is encoded by \textit{designs}.  We first show that if the generators of minimal degree, say $k$, in a squarefree monomial ideal with $n$ variables form a design, then the Waldschmidt constant of the Alexander dual is $n/k$.  We use this to deduce the asymptotic resurgence of facet ideals of \textit{perfect matroid designs} - this class of matroids includes finite projective and affine geometries.

Our reference for designs is~\cite[Chapter~12]{Wel}.  Let $t,n,k,\lambda$ be integers with $1\le t\le k \le n$.  A $t$-$(n,k,\lambda)$ \textit{design} is a pair $\mathfrak{D}=(\mathfrak{B},E),$ where $E$ is a set of $n$ elements and $\mathfrak{B}$ is a collection of subsets of $E$ (called \textit{blocks}) of cardinality $k$ so that any subset of $E$ with cardinality $t$ is contained in exactly $\lambda$ blocks.

We let $\beta(\mathfrak{D})=|\mathfrak{B}|$ be the number of blocks.  For any subset $U\subset E$ with cardinality $0\le |U|=s\le t$, one can compute that there are $\lambda\dfrac{\binom{n-s}{t-s}}{\binom{k-s}{t-s}}$ many blocks containing $U$.  We write $\lambda_s(\mathfrak{D})$ (or simply $\lambda_s$ if $\mathfrak{D}$ is understood) for $\lambda\dfrac{\binom{n-s}{t-s}}{\binom{k-s}{t-s}}$.  Notice that $\lambda=\lambda_t$ and if $U=\emptyset$ we see that $\beta(\mathfrak{D})=\lambda_0(\mathfrak{D})=\lambda \frac{{n \choose t}}{{k \choose{t}}}$.  It follows that a $t$-$(n,k,\lambda)$ design $\mathfrak{D}$ is also an $s$-$(n,k,\lambda_s)$ design for $1\le s\le t$. A {\it balanced incomplete block design} (in short BIBD) $\mathfrak{D}$ is a $2$-$(n,k,\lambda)$ design with $b(\mathfrak{D})<\binom{n}{k}$.   

Let $\mathfrak{D}$ be a $t$-$(n,k,\lambda)$ design on the ground set $E=[n]$.  We associate to $\mathfrak{D}$ the squarefree monomial ideal $I(\mathfrak{D}):=\langle x_B~:~B \in \mathfrak{D}\rangle \subseteq \KK[x_1,\ldots,x_n]$, which we call the {\it block design ideal} of $\mathfrak{D}$.

We show that if the Alexander dual of a squarefree monomial ideal contains the block design ideal of a $t$-$(n,k,\lambda)$ design among its minimal generators, then the Waldschmidt constant is determined by the parameters $n$ and $k$ of the design.

\begin{Lemma}\label{lem:bd-growth-rate}
Suppose $I\subset \KK[x_1,\ldots,x_n]$ is a squarefree monomial ideal and $\mathfrak{D}$ is a $t$-$(n,k,\lambda)$ design on the ground set $E$ of size $n$ with $t\ge 1$.  If $I(\mathfrak{D})\subseteq \dual{I}$ and $\alpha(\dual{I})=k$, then $\widehat{\alpha}(I)=\frac{n}{k}$.
\end{Lemma}
\begin{proof}
It follows from \Cref{def:SymbolicPolyhedron} and the definition of the Alexander dual that the inequality $\langle \chi_B,{\bf a}\rangle=\sum\limits_{i\in B} a_i\ge 1$ is satisfied for all ${\bf a}=(a_1,\ldots,a_n)\in \SP(I)$ and $B\in\mathfrak{B}$.  Adding these inequalities together for all blocks $B\in\mathfrak{B}$ yields
\[
\lambda_1(\mathfrak{D})\sum_{i=1}^n a_i \ge \beta(\mathfrak{D})=\lambda_0(\mathfrak{D})
\]
for all ${\bf a}\in\SP(I)$.  It follows  from \Cref{thm:SPWaldschmidt} that $\widehat{\alpha}(I)\ge  \frac{\lambda_0(\mathfrak{D})}{\lambda_1(\mathfrak{D})}=\frac{n}{k}$.  For the opposite inequality, we claim that
\[
{\bf c}=\sum_{i=1}^n \frac{1}{k}{\bf e}_i\in\SP(I).
\]
Let $P_U\in \mbox{Min}(I)$.  Then, $x^U\in \dual{I}$.   Since $\alpha(\dual{I})=\alpha(I(\mathfrak{D}))=k$, $|U|\ge k$.  So $\langle \chi_U,{\bf c}\rangle=\sum\limits_{i\in U}\frac{1}{k}\ge 1$ for all $P_U \in  \mbox{Min}(I)$.  Therefore,  by \Cref{def:SymbolicPolyhedron},  ${\bf c}\in \SP(I)$.  It follows from \Cref{thm:SPWaldschmidt} that $\widehat{\alpha}(I)\le \sum\limits_{i=1}^n \frac{1}{k}=\frac{n}{k}$, completing the proof.
\end{proof}

We use the machinery of $t$-designs to compute the Waldschmidt constant of the facet ideal of a perfect matroid design. A \textit{perfect matroid design} is a matroid $M$ in which each flat of rank $r$ has the same cardinality for all $1 \le r \le \rk(M)$; we denote the cardinality of a rank $r$ flat of a perfect matroid design by $f_r(M)$.  Standard examples of perfect matroid designs are uniform matroids, \textit{projective geometries}, \textit{affine geometries}, and matroid designs of \textit{Steiner systems}. Also, truncations and contractions of a perfect matroid design are again perfect matroid designs. Our reference for perfect matroid designs is 
\cite[Chapter 12.5]{Wel}.

\begin{Proposition}\label{prop:wald-PMD}
Let $M$ be a perfect matroid design of rank $k$ on the ground set $E$ of size $n$. Then, $$\widehat{\alpha}(I(M))= \frac{n-\ell(M)}{\alpha(I_{\Delta(M^*)})},$$ where $\ell(M)$ denotes the number of loops in $M$.
\end{Proposition}
\begin{proof}
As $M$ is a perfect matroid design of rank $k$, for each $0 \le i \le k$, the flats of rank $i$ have the same size, namely $f_i(M)$.  Suppose $k=2$.  Then $|\cL_1(M)|=(n-\ell(M))/(f_1(M)-\ell(M))$ and, by \Cref{prop:asym-res-rank-two},
\[
\widehat{\alpha}(I(M))=\frac{(n-\ell(M))/(f_1(M)-\ell(M))}{(n-\ell(M))/(f_1(M)-\ell(M))-1}=\frac{n-\ell(M)}{n-f_1(M)}=\frac{n-\ell(M)}{\alpha(I_{\Delta(M^*)})}.
\]
The final equality holds because $n-f_1(M)$ is the cardinality of a complement of a rank one flat, which is the cardinality of a circuit of $M^*$.

Now suppose $k\ge 3$.  For any $i\le j\le \ell \in \{0,1,\ldots,k\}$ and given a rank $i$-flat $F$ and a rank $\ell$-flat $G$ with $F \subset G$, we define $t_M(i,j,\ell)$ as the number of rank $j$-flats which contain $F$ and are contained within $G$. It follows from \cite[Theorem 12.5.1]{Wel} that $t_M(i,j,\ell)$ is independent of the choice of $F$ and $G$, and  hence, well defined  and satisfies  
\[
t_M(i,i+1,j) =\dfrac{t_M(0,1,j)-t_M(0,1,i)}{t_M(0,1,i+1)-t_M(0,1,i)}, \text{ for } 0 \le i \le j \le k.
\]
We claim that $\C(M^*)$ is the collection of blocks of a $1$-$(n-\ell(M),n-f_{k-1}(M),t_M(0,k-1,k)-t_M(0,1,k-1))$ design on $E\setminus\loops(M)$.  To see this, first observe that a circuit of $M^*$ is the complement of a hyperplane of $M$.  Since every hyperplane of $M$ has cardinality $f_{k-1}(M)$, every circuit of $M^*$ has cardinality $n-f_{k-1}(M)$.  Now suppose that $e\in E$ is not a loop of $M$.  Then $e$ is contained in exactly one rank one flat of $M$, namely $\mbox{cl}(\{e\})$.  Thus $e$ is contained in exactly $t_M(0,1,k-1)$ hyperplanes of $M$.  It follows that $e$ is contained in $|\cL_{k-1}(M)|-t_M(0,1,k-1)=t_M(0,k-1,k)-t_M(0,1,k-1)$ many complements of hyperplanes of $M$ (i.e. circuits of $M^*$).  Note also that $\alpha(I_{\Delta(M^*)}) =n-f_{k-1}(M)$.  Observe that if a variable corresponds to a loop of $M$, then it will not appear in either $I(M)$ or $I_{\Delta(M^*)}$.  Thus we can regard both as ideals of a polynomial ring in $n-\ell(M)$ variables (corresponding to non-loops of $M$).  We apply \Cref{lem:bd-growth-rate} in this new ring to obtain $\widehat{\alpha}(I(M)) = \frac{n-\ell(M)}{\alpha(I_{\Delta(M^*)})}$.
\end{proof}

We now provide the asymptotic resurgence of facet ideals of perfect matroid designs and Stanley-Reisner ideals of their dual. 
\begin{Theorem}\label{thm:asym-res-PMD}
Let $M$ be a perfect matroid design of rank $k$ on the ground set $E$ of size $n$. Then, \[\widehat{\rho}(I(M))=\frac{\alpha(I(M))}{\widehat{\alpha}(I(M))}=\frac{k\alpha(I_{\Delta(M^*)})}{n-\ell(M)}=\frac{\alpha(I_{\Delta(M^*)})}{\widehat{\alpha}(I_{\Delta(M^*)})}=\widehat{\rho}(I_{\Delta(M^*)}).\] 
\end{Theorem}
\begin{proof}
We first prove  $\widehat{\rho}(I(M))=\frac{\alpha(I(M))}{\widehat{\alpha}(I(M))}=\frac{k\alpha(I_{\Delta(M^*)})}{n-\ell(M)}$ by induction on $k$. If $k=2$, then by \Cref{lem:rank-two}, $\widehat{\rho}(I(M))=\frac{2}{\widehat{\alpha}(I(M))}=\frac{2\alpha(I_{\Delta(M^*)})}{n-\ell(M)}$, where the last equality follows from \Cref{prop:wald-PMD}.   Now, we assume that $\rk(M)>2$ and the result is true for any perfect matroid design $N$ with $\rk(N)< \rk(M)$.  It follows from \Cref{prop:asym-single-contraction-matroid} and \Cref{prop:wald-PMD} that  $\widehat{\rho}(I(M))=\max\left\lbrace \frac{k}{\widehat{\alpha}(I(M))},\widehat{\rho}_{c,1}(I(M))\right\rbrace= \max\left\lbrace \frac{k\alpha(I_{\Delta(M^*)})}{n-\ell(M)},\widehat{\rho}_{c,1}(I(M))\right\rbrace.
$ We claim that $\widehat{\rho}_{c,1}(I(M)) \le \frac{k\alpha(I_{\Delta(M^*)})}{n-\ell(M)}$. By \Cref{prop:asym-single-contraction-matroid}, $\widehat{\rho}_{c,1}(I(M))=\max\{\widehat{\rho}(I(M/F))~:~F\in \cL_1(M)\}$.  Let $F\in \cL_1(M)$. Then, $M/F$ is a perfect matroid design of rank $k-1$ on the ground set $E \setminus F$, which has size $n-|F|$. Therefore, by induction, $\widehat{\rho}(I(M/F))=\frac{\alpha(I(M/F))}{\widehat{\alpha}(I(M/F))}=\frac{(k-1)\alpha(I_{\Delta((M/F)^*)})}{n-|F|-\ell(M/F)}$, where $\ell(M/F)$ is the number of loops in $M/F$. Since $F$ is a flat, $\ell(M/F)=0$.   Observe that a $j$-flat of $M/F$ is obtained from the $(j+1)$-flat of $M$ that contains $F$ by removing $F$ from it.  Therefore,  $\alpha(I_{\Delta((M/F)^*)})=n-|F|-f_{k-2}(M/F)=n-|F|-(f_{k-1}(M) -|F|)=n-f_{k-1}(M)=\alpha(I_{\Delta(M^*)})$.   So, in order to prove that $\widehat{\rho}_{c,1}(I(M)) \le \frac{k\alpha(I_{\Delta(M^*)})}{n-\ell(M)}$, it is sufficient to prove that  $\frac{k-1}{n-|F|}\le \frac{k}{n-\ell(M)}$ for every $F \in \cL_1(M)$.   Let $F\in \cL_1(M)$.  It suffices to prove that $k|F| +\ell(M) \le n+k\ell(M)$.  Since $\rk(M)=k$,  for any basis $B=\{e_1,\ldots,e_k\}$, $E=\text{cl}(B)$. Notice that  $\text{cl}(\{e_1\}),\ldots, \text{cl}(\{e_k\})$ are distinct rank one flats of $M$, all of which have the same cardinality, namely $|F|$, since $M$ is a perfect matroid design.  Also, $\bigcup\limits_{i=1}^k \text{cl}(\{e_i\}) \subseteq E$.  Therefore, $|E\setminus \loops(M) | \ge |\text{cl}(\{e_1\}) \setminus \loops(M)| +\cdots + |\text{cl}(\{e_k\}) \setminus \loops(M)|$.   This implies that $n-\ell(M) \ge k(|F| -\ell(M))$ which completes the proof of the claim. Thus,  $\widehat{\rho}(I(M))=\frac{\alpha(I(M))}{\widehat{\alpha}(I(M))}=\frac{k\alpha(I_{\Delta(M^*)})}{n-\ell(M)}$.  The penultimate equality follows from \cite[Remark 2.6, Theorem 6.8]{DFKT24}, and the final equality  follows from \Cref{prop:Alex duality facet}. 
\end{proof}

We conclude this section by computing the asymptotic resurgence of facet ideals of the four aforementioned families of perfect matroid designs: uniform matroids, projective geometries, affine geometries, and perfect matroid designs arising from Steiner systems.  
\begin{Example}\label{ex:uniformMatroids}
Let ${\mathrm U}_{k,n}$ be the uniform matroid of rank $k$ on the ground set $E$ of size $n$. For all $i \in [k-1]$, the flats of rank $i$ are all $i$-subsets of $E$ and the flat of rank $k$ is $E$. Therefore, ${\mathrm U}_{k,n}$ is a perfect matroid design. Observe that $\mathrm{U}_{k,n}^*=\mathrm{U}_{n-k,n}$. So, $\alpha(I_{\Delta({\mathrm{U}_{n-k,n}})})=n-k+1$. By \Cref{thm:asym-res-PMD},  \[\widehat{\rho}(I(\mathrm{U}_{k,n}))=\frac{\alpha(I(\mathrm{U}_{k,n}))}{\widehat{\alpha}(I(\mathrm{U}_{k,n}))}=\frac{k(n-k+1)}{n}=\frac{\alpha(I_{\Delta({\mathrm{U}_{n-k,n}})})}{\widehat{\alpha}(I_{\Delta({\mathrm{U}_{n-k,n}})})}=\widehat{\rho}(I_{\Delta({\mathrm{U}_{n-k,n}})}).\] 
\end{Example}

\begin{Example}[Projective geometries]\label{ex:projectiveGeometries}
Let $\KK=\mathbb{F}_q$ be a finite field of order $q$, where $q$ is a power of a prime.  Given an integer $m\ge 0$, there is a matroid of rank $m+1$, derived from the projective space $\mathbb{P}^{m}_{\KK}$, called the \textit{projective geometry} $PG(m,\KK)$.  See~\cite[Chapter~6.1]{Oxley2011} for a detailed discussion.  The ground set of the projective geometry $PG(m,\KK)$ is the set of all $(q^{m+1}-1)/(q-1)$ points of $\mathbb{P}^m_{\KK}$.  For all $i\in[m+1]$ a flat of rank $i$ in $PG(m,\KK)$ corresponds to a  linear subspace of $\KK^{m+1}$ of dimension $i$, and thus is isomorphic to the projective geometry $PG(i-1,\KK)$.  It follows that a flat of rank $i$ in $PG(m,\KK)$ has size $(q^{i}-1)/(q-1)$, for all $i\in[m+1]$.  Thus, $PG(m,\KK)$ is a perfect matroid design.  Note that $\alpha(I_{\Delta(PG(m,\KK)^*)})=\frac{q^{m+1}-1}{q-1}-\frac{q^m-1}{q-1}=q^m$.  Hence, by \Cref{thm:asym-res-PMD}, \[\widehat{\rho}(I(PG(m,\KK)))=\frac{\alpha(I(PG(m,\KK)))}{\widehat{\alpha}(I(PG(m,\KK)))}=\frac{(m+1)(q-1)q^m}{q^{m+1}-1}=\frac{\alpha(I_{\Delta({PG(m,\KK)}^*)})}{\widehat{\alpha}(I_{\Delta({PG(m,\KK)}^*)})}=\widehat{\rho}(I_{\Delta({PG(m,\KK)}^*)}).\] \end{Example}

\begin{Example}[Affine geometries]\label{ex:affineGeometries}
Let $\KK=\mathbb{F}_q$ be a finite field of order $q$, where $q$ is a power of a prime.  Given a natural number $m\ge 1$, the \textit{affine geometry} $AG(m,\KK)$ is a rank $m+1$ matroid obtained from the projective geometry $PG(m,\KK)$ by removing a hyperplane of $PG(m,\KK)$.  See~\cite[Chapter~6.2]{Oxley2011} for a detailed discussion.  Explicitly, the ground set of $AG(m,\KK)$ consists of all $q^m$ points in $\KK^m$.  The flats of rank $i$ ($1\le i\le m+1$) of $AG(m,\KK)$ correspond exactly to affine linear subspaces of $\KK^m$ of dimension $i-1$.  Thus, a flat of rank $i$ in $AG(m,\KK)$ is isomorphic to $AG(i-1,\KK)$ and hence has $q^{i-1}$ elements for $1\le i\le m+1$.  Evidently, $AG(m,\KK)$ is a perfect matroid design. Since flats of rank $m$ have size $q^{m-1}$, $\alpha(I_{\Delta(AG(m,\KK)^*)}) = q^m-q^{m-1}$. Therefore, $\frac{(m+1)\alpha(I_{\Delta(AG(m,\KK)^*)})}{q^m}= \frac{(m+1)(q^m-q^{m-1})}{q^{m}}=\frac{(m+1)(q-1)}{q}$.  Hence, by \Cref{thm:asym-res-PMD}, \[\widehat{\rho}(I(AG(m,\KK)))=\frac{\alpha(I(AG(m,\KK)))}{\widehat{\alpha}(I(AG(m,\KK)))}=\frac{(m+1)(q-1)}{q}=\frac{\alpha(I_{\Delta({AG(m,\KK)}^*)})}{\widehat{\alpha}(I_{\Delta({AG(m,\KK)}^*)})}=\widehat{\rho}(I_{\Delta({AG(m,\KK)}^*)}).\] \end{Example}

\begin{Example}[Matroid designs of Steiner systems]\label{ex:SteinerSystemPMD}
Let $\mathcal{S}=(E,\mathfrak{B})$ be a Steiner system of type $S(t,k,n)$ with $1 \le t<k<n$, see \Cref{sec:asym-Steiner-system} for definition and more details. Let  $\cB$ be those $t+1$ subsets of $E$ that are not contained in any block of $\mathcal{S}$. Then $\cB$ satisfies the basis exchange property and, hence, gives rise to a matroid $M$ of rank $t+1$ on the ground set $E$ of size $n$. The matroid arising this way is known as the {\it matroid design of the Steiner system} $\mathcal{S}$, and it is a perfect matroid design (also a paving matroid) as its hyperplanes are the blocks of the Steiner system $\mathcal{S}$, and for each $0 \le i \le t-1$, the rank $i$ flats are all size $i$ subsets of $E$.  See \cite[Chapter 12]{Wel} for a detailed discussion. Let $M$ be a matroid design of a Steiner system $\mathcal{S}$ of type $S(t,k,n)$.  By \Cref{thm:asym-res-PMD}, \[\widehat{\rho}(I(M))=\frac{\alpha(I(M))}{\widehat{\alpha}(I(M))}=\frac{(t+1)(n-k)}{n}=\frac{\alpha(I_{\Delta(M^*)})}{\widehat{\alpha}(I_{\Delta(M^*)})}=\widehat{\rho}(I_{\Delta(M^*)}).\]
\end{Example}

\section{Steiner Systems: Waldschmidt Constant and Asymptotic Resurgence Numbers}\label{sec:asym-Steiner-system}
In this section, we compute the Waldschmidt constant and the resurgence number of facet ideals of matroids arising from Steiner systems.  These are different from the perfect matroid designs of \Cref{ex:SteinerSystemPMD} - see \Cref{rmk:TwoSteinerMatroids}.  As a consequence, it follows that facet ideals of matroids arising from Steiner systems achieve equality in \GHV{}\eqref{eq:rhohatineq}.  Our interest in matroids arising from Steiner systems was sparked by a paper of Ballico-Favacchio-Guardo-Milazzo ~\cite{BFGM21} and its sequel by Ballico-Favacchio-Guardo-Milazzo-Thomas~\cite{BFGMT21}.  \Cref{thm:asym-res-Steiner-systems} generalizes the asymptotic resurgence statement appearing in \cite[Corollary~4.8]{BFGM21}.

A {\it Steiner system of type $S(t,k,n)$} is a $t$-$(n,k,1)$ design.  That is, a Steiner system of type $S(t,k,n)$ is a pair $\mathcal{S}=(E,\mathfrak{B})$ where $E$ is a set of cardinality $n$ and $\mathfrak{B}$ is a collection of subsets of $E$ (called the \textit{blocks} of $\mathcal{S}$).  The blocks of $\mathcal{S}$ all have cardinality $k$ and any $t$-element subset of $E$ is contained in a unique block.

The blocks of a Steiner system of type $S(k,k,n)$ are simply all $k$-element subsets of $E$.  There is only one block of a Steiner system of type $S(t,n,n)$ -- namely the whole ground set $E$.  As we are interested in Steiner systems where the blocks do not consist of all possible $k$-element subsets of $E$, we will assume throughout this section that $1\le t<k<n$.
 
There is a matroid one can naturally associate with a Steiner system, as follows.

\begin{Definition}\label{def:matroidfromSteiner}
If $\mathcal{S}=(E,\mathfrak{B})$ is a Steiner system of type $S(t,k,n)$, then $M(\mathcal{S})=(E,\mathcal{B})$ is a matroid of rank $k$, where $\mathcal{B}$ consists of all subsets of $E$ of size $k$ which are not blocks of $\mathcal{S}$.
\end{Definition}

\begin{Remark}\label{rmk:TwoSteinerMatroids}
If $\mathcal{S}$ is a Steiner system of type $S(t,k,n)$, the matroid $M(\mathcal{S})$ from \Cref{def:matroidfromSteiner} is usually different from the \textit{matroid design} of $\mathcal{S}$ described in \Cref{ex:SteinerSystemPMD}.  In fact, these two matroids match if and only if $k=t+1$. 
\end{Remark}

\begin{Remark}\label{rem:sparsepaving}
If $\mathcal{S}$ is a Steiner system, then $M(\mathcal{S})$ is a sparse paving matroid -- see \cite[Proposition~6.6]{DFKT24} or \cite[Proposition~7.15]{ManNgu}.
\end{Remark}

We first compute the Waldschmidt constant of the facet ideals of $M(\mathcal{S})$ and $M(\mathcal{S})^*$ using the techniques of \Cref{sec:DesignsAndPMDs}.

\begin{Lemma}\label{lem:walds-steiner-systems}
Let $\mathcal{S}=(E,\mathfrak{B})$ be a Steiner system of type $S(t,k,n)$ and $M(\mathcal{S})$ be the associated matroid in the sense of Definition~\ref{def:matroidfromSteiner}. Then, $$\widehat{\alpha}(I(M(\mathcal{S})))=\frac{n}{n-k} \mbox{  and  } \widehat{\alpha}(I(M(\mathcal{S})^*))=\frac{n}{k}.$$
\end{Lemma}
\begin{proof}
First, note that $\dual{I(M(\mathcal{S}))}=I_{\Delta(M(\mathcal{S})^*)}$ and  $\dual{I(M(\mathcal{S})^*)}=I_{\Delta(M(\mathcal{S}))}$ by \Cref{prop:Alex duality facet}.  Since $M(\mathcal{S})$ is paving (see \Cref{rem:sparsepaving}), each block of $\mathcal{S}$ is a circuit of $M(\mathcal{S})$ of minimal size, and therefore,  $I(\mathcal{S}) \subseteq I_{\Delta(M(\mathcal{S}))}$ and $\alpha(I_{\Delta(M(\mathcal{S}))})=k$.  By \Cref{lem:bd-growth-rate}, $\widehat{\alpha}(I(M(\mathcal{S})^*))=\frac{n}{k}$.

Since $M(\mathcal{S})^*$ is also paving (see \Cref{rem:sparsepaving}), the circuits of $M(\mathcal{S})^*$ of minimal size are subsets of $[n]$ of size $n-k$ which are complements of blocks of $\mathcal{S}$.  The complements of blocks of $\mathcal{S}$ are a $1$-$(n,n-k,b(\mathcal{S})-\lambda_1)$ design, called the \textit{complementary design} of $\mathcal{S}$ (which we denote $\mathcal{S}^c$).  Therefore, $I(\mathcal{S}^c) \subseteq I_{\Delta(M(\mathcal{S})^*)}$ and $\alpha(I_{\Delta(M(\mathcal{S})^*)})=n-k$.  It follows from \Cref{lem:bd-growth-rate} that $\widehat{\alpha}(I(M(\mathcal{S})))=\frac{n}{n-k}$.
\end{proof}

We now prove the main result of this section: we obtain the asymptotic resurgence of facet ideals of matroids that arise from Steiner Systems, along with the Stanley-Reisner ideal of their duals. 

\begin{Theorem}\label{thm:asym-res-Steiner-systems}
 Let $\mathcal{S}=(E,\mathfrak{B})$ be a Steiner system of type $S(t,k,n)$ and $M(\mathcal{S})$ be the associated matroid in the sense of Definition~\ref{def:matroidfromSteiner}. Then, 
 \[
 \widehat{\rho}(I(M(\mathcal{S})))=
 \frac{\alpha(I(M(\mathcal{S})))}{\widehat{\alpha}(I(M(\mathcal{S})))}
 =\frac{k(n-k)}{n}=\frac{\alpha(I_{\Delta(M(\mathcal{S})^*)})}{\widehat{\alpha}(I_{\Delta(M(\mathcal{S})^*)})}=\widehat{\rho}(I_{\Delta(M(\mathcal{S})^*)}).
 \]
\end{Theorem}
\begin{proof}
The second equality follows from \Cref{lem:walds-steiner-systems} and the third follows from \cite[Corollary 6.7]{DFKT24}.  \Cref{prop:Alex duality facet} implies $\widehat{\rho}(I(M(\mathcal{S})))=\widehat{\rho}(I_{\Delta(M(\mathcal{S})^*)})$.  A classical inequality of Tits~\cite{Tits1964} (see also~\cite{HuberSteiner2008}) states that $n \geq (t+1)(k-t+1)$.  It follows that $n-2k\ge (t+1)k-(t+1)(t-1)-2k=(t-1)(k-t-1)\ge 0$, since $1\le t<k$.  Hence $n\ge 2k$.  We observed in \Cref{rem:sparsepaving} that $M(\mathcal{S})$ is sparse paving.  It follows from \Cref{thm:asym-res-dual-paving} that
\[
\widehat{\rho}(I(M(\mathcal{S})))=\dfrac{k}{\widehat{\alpha}(I(M(\mathcal{S})))}.
\]
Now the result follows from \Cref{lem:walds-steiner-systems}.
\end{proof}

\section{Concluding Remarks}\label{sec:conc-rmk}

We end the main body of the paper with a few remarks and additional connections to the literature.

\begin{Remark}
In~\cite[Corollary~8.8]{DFKT24}, the authors -- together with Toh\v{a}neanu -- provide bounds on the asymptotic resurgence and resurgence of ideals defining \textit{matroid configurations of points}.  Matroid configurations of points are projective varieties defined by certain specializations of Stanley-Reisner ideals of matroids.  In \cite[Question~9.5]{DFKT24}, the authors and Toh\v{a}neanu ask whether the bounds of~\cite[Corollary~8.8]{DFKT24} lift to the Stanley-Reisner ideal.  Many of our results in this paper (e.g. \Cref{thm:asym-res-dual-paving}, \Cref{thm:asym-res-PMD}, and \Cref{thm:asym-res-Steiner-systems}) are partial answers to  \cite[Question~9.5]{DFKT24} for the asymptotic resurgence.
\end{Remark}

In~\cite{GHV13}, Guardo-Harbourne-Van Tuyl prove that $\widehat{\rho}(I)\le \frac{\text{reg}(I)}{\widehat{\alpha}(I)}$ for an ideal $I$ defining a smooth subscheme of $\mathbb{P}^{n-1}$.  The bound of \Cref{cor:asym-upper-bound} implies the corresponding bound for $\widehat{\rho}(I_{\Delta(M)})$.  Note that the projective variety defined by $I_{\Delta(M)}$ is a union of coordinate hyperplanes, which typically has a non-empty singular locus.

\begin{Remark}\label{rmk:regUpperSR}

Let $M$ be a matroid of rank $k$ on the ground set $[n]$.  It follows from \cite{JV13} that $\text{reg}(I_{\Delta(M)}) =k+1-\ell(M^*)$. By \cite[Lemma 2.4, Corollary 3.2, Theorem 3.4]{DFKT24},  $\widehat{\alpha}(I_{\Delta(M)}) \le \frac{n-\ell(M^*)}{n-k}$. Observe that $I_{\Delta(M/ \loops(M^*))}$ is generated by the same minimal generators as $I_{\Delta(M)}$.   So, $\widehat{\rho}(I_{\Delta(M)})=\widehat{\rho}(I_{\Delta(M/\loops(M^*))})$.  Since $M/\loops(M^*)$ is a matroid of rank $k-\ell(M^*)$ on the ground set of size $n-\ell(M^*)$, by  \Cref{thm:asym-order} and \Cref{prop:uniform-matroid},
\[ \widehat{\rho}(I_{\Delta(M/\loops(M^*))}) \le\widehat{\rho}(I_{\Delta(\mathrm{U}_{k-\ell(M^*),n-\ell(M^*)})})=\frac{(n-k)(k-\ell(M^*)+1)}{n-\ell(M^*)}\le \frac{\mathrm{reg}(I_{\Delta(M)})}{\widehat{\alpha}(I_{\Delta(M)})}.
\]
\end{Remark}

The upper bound $\widehat{\rho}(I_{\Delta(M)})$ in \Cref{cor:asym-upper-bound}  can be significantly better than $\frac{\mathrm{reg}(I_{\Delta(M)})}{\widehat{\alpha}(I_{\Delta(M)})}$.  We illustrate this with the class of so-called \textit{Theta matroids} -- see \cite[page 664]{Oxley2011}.
\begin{Example}[Theta matroids]\label{ex:ThetaMatroids}
Let $k \ge 2$ be a positive integer, and let $U=\{u_1,\ldots,u_k\}$ and $V=\{v_1,\ldots,v_k\}$ be disjoint ordered sets. Set $E=U\cup V$.  The {\it $k$-theta matroid}, denoted by $\Theta_k$, is a rank $k$ self-dual matroid on the ground set $E$ with bases 
$$\cB(\Theta_k)=\left\{B \subset E~:~ |B|=k, \, |B \cap U| \le 2 \text{ and } B \neq (V\setminus \{v_i\} ) \cup \{u_i\} \text{ for each } i \in [k]\right\}.$$

Put $I=I_{\Delta(\Theta_k)}$.  In~\cite[Example~5.11]{DFKT24}, it is shown that $\widehat{\alpha}(I)=\frac{k}{k-2}$ if $k\ge 5$.  Thus
$
\frac{\mathrm{reg}(I)}{\widehat{\alpha}(I)}=\frac{(k+1)(k-2)}{k}
$
and
$
\frac{k(n-k)}{n-1}=\frac{k^2}{2k-1}.
$
Since $k\ge 5$, $\frac{k^2}{2k-1}<\frac{(k+1)(k-2)}{k}$.  In fact, as $k\to\infty$, the ratio between these approaches $2$. 

A computation in Macaulay2~\cite{M2} shows that $\widehat{\rho}(I_{\Delta(\Theta_3)})=\frac{3}{2}<\frac{9}{5}$ (see Row 6 in \Cref{tbl:Rank3groundset6}).  Thus $\widehat{\rho}(I)<\frac{k^2}{2k-1}$ in general.
\end{Example}

In~\cite{GHV13}, Guardo-Harbourne-Van Tuyl also prove that if $I$ is the ideal of a smooth subscheme of $\mathbb{P}^{n-1}$, then $\widehat{\rho}(I)\le \frac{\omega(I)}{\widehat{\alpha}(I)}$, where $\omega(I)$ is the largest degree of a minimal generator of $I$.  If $M$ is a matroid, then $\omega(I_{\Delta(M)})$ is the largest size of a circuit of $M$.  We conjecture this bound holds for Stanley-Reisner ideals of matroids, and have checked it on all simple matroids on ground sets up to size $8$.

\begin{Conjecture}\label{q:omegaupper}
If $M$ is a matroid then $\widehat{\rho}(I_{\Delta(M)})\le \frac{\omega(I_{\Delta(M)})}{\widehat{\alpha}(I_{\Delta(M)})}$.
\end{Conjecture}

Even if \Cref{q:omegaupper} is true, the bound of \Cref{cor:asym-upper-bound} can be better; in \Cref{ex:ThetaMatroids}, the matroid $\Theta_k$ has maximal circuit size of $k+1$, which is equal to $\mbox{reg}(I_{\Delta(\Theta_k)})$. On the other hand, if \Cref{q:omegaupper} is true, it would sometimes improve on the upper bound of \Cref{cor:asym-upper-bound}.  In particular, if $\omega(I_{\Delta(M)})=\alpha(I_{\Delta(M)})$, then \Cref{q:omegaupper} would imply that $\widehat{\rho}(I_{\Delta(M)})=\frac{\alpha(I_{\Delta(M)})}{\widehat{\alpha}(I_{\Delta(M)})}$.  Matroids that satisfy $\omega(I_{\Delta(M)})=\alpha(I_{\Delta(M)})$ are dual to so-called \textit{matroid designs}.  In particular, if \Cref{q:omegaupper} is true, then it implies \Cref{thm:asym-res-PMD}.

\begin{Remark}\label{rmk:regUpperFacet}
If $M$ is a matroid, it is generally not true that $\widehat{\rho}(I(M))\le \frac{\text{reg}(I(M))}{\widehat{\alpha}(I(M))}$.  The ideal $I(M)$ is a \textit{polymatroidal} ideal, which is well-known to have linear quotients (see~\cite[Lemma~1.3]{HT2002}) and hence a linear minimal free resolution.  Thus $\alpha(I(M))=\mathrm{reg}(I(M))$.  It follows that $\widehat{\rho}(I(M))\le\frac{\mathrm{reg}(I(M))}{\widehat{\alpha}(I(M))}$ if and only if $\widehat{\rho}(I(M))=\frac{\alpha(I(M))}{\widehat{\alpha}(I(M))}$.  That is, $\widehat{\rho}(I(M))\le \frac{\text{reg}(I(M))}{\widehat{\alpha}(I(M))}$ if and only if $I(M)$ achieves equality in \GHV{}.  Clearly, then, the question of which matroids satisfy $\widehat{\rho}(I(M))\le\frac{\mathrm{reg}(I(M))}{\widehat{\alpha}(I(M))}$ is equivalent to \Cref{ques:rhohat=alpha/alphahat}.  In particular, we have seen a number of matroids $M$ for which $I(M)$ does not achieve equality in \GHV{} (e.g. \Cref{ex:direct-sum} and, in \Cref{thm:relaxation-to-uniform}, the almost-uniform matroid $\mathrm{U}_{k,n}^{-}$ with $2k>n+1$).
\end{Remark}

We end by remarking on another key difference between the facet and Stanley-Reisner ideal.
\begin{Remark} 
If $M$ is a matroid, it follows from \cite[Corollary 3.8]{Vill-poly-normal} that the facet ideal $I(M)$ is a \textit{normal} ideal - that is, all its powers are integrally closed.  Thus $\widehat{\rho}(I(M))=\rho(I(M))$ by \cite[Corollary~4.14]{DFMS19}.  In an upcoming paper, we study the resurgence of Stanley-Reisner ideals of matroids (and of specializations which define \textit{matroid configurations}).  As part of this work, we show that the asymptotic resurgence and resurgence of $I_{\Delta(M)}$ are not necessarily equal.
\end{Remark}

\noindent \textbf{Acknowledgements:} 
We thank Elena Guardo, T\`ai H\`a, Brian Harbourne, Juan Migliore, Uwe Nagel, Adam Van Tuyl, and Rafael Villarreal for providing feedback on this paper.  DiPasquale was partially supported by NSF grant DMS--2344588. Kumar was partially supported by an AMS-Simons Travel grant.

\vspace{10 pt}

\bibliography{bibl}
\bibliographystyle{plain}

\appendix

\section{An algorithm for asymptotic resurgence of squarefree monomial ideals}\label{app:algo}

The asymptotic resurgence of a squarefree monomial ideal can be computed from its symbolic polyhedron, as shown in \cite{DFMS19}. For ${\bf w}\in \widehat{\RR^n}$,    set \[{\bf w}(I)=\min\{\langle {\bf w},{\bf v}\rangle~:~ {\bf v}\in \NP(I)\} \text{ and } \widehat{\bf w}(I)=\min\{\langle {\bf w},{\bf v}\rangle ~:~ {\bf v}\in \SP(I)\}.\]  We say ${\bf w}\in \widehat{\RR^n}$ is \textit{supported} on $I$ if ${\bf w}(I)>0$.  In this case $\widehat{\bf w}(I)>0$ also.

\begin{Theorem}[{\cite[Corollary~2.16]{DFMS19}}]\label{thm:asymptoticresurgencebyvaluations}
Let $I$ be a squarefree monomial ideal of $S$.  Then
\[
\widehat{\rho}(I)=\sup\left\lbrace \frac{{\bf w}(I)}{\widehat{\bf w}(I)}~:~ {\bf w}\mbox{ is supported on }I\right\rbrace=\max\left\lbrace\frac{{\bf w}_1(I)}{\widehat{\bf w}_1(I)},\ldots,\frac{{\bf w}_k(I)}{\widehat{\bf w}_k(I)} \right\rbrace,
\]
where ${\bf w}_1,\ldots,{\bf w}_k$ are supported on $I$ and are inward pointing normals to facets of $\NP(I)$.
\end{Theorem}

From this, we can extract a naive algorithm to compute asymptotic resurgence $\widehat{\rho}(I)$ for a squarefree monomial ideal $I$.  Namely, first compute the facet-defining inequalities of $\NP(I)$ and take the inward pointing normals.  Then dot these normals with the vertices of $\SP(I)$ and take the minimum value.  This procedure is straightforward to implement in Macaulay2~\cite{M2}, and it is how we obtained the asymptotic resurgence values in \Cref{tbl:Rank3groundset5},\Cref{tbl:Rank3groundset6} and \Cref{tbl:Rank3groundset7}. To implement this algorithm, we used functions from the Matroids~\cite{MatroidsSource,MatroidsArticle} and Polyhedra~\cite{PolyhedraSource,PolyhedraArticle} packages in Macaulay2.

A different algorithm which uses $\SP(\dual{I})$ instead of $\NP(I)$ is described by Villarreal in~\cite{Villarreal-2023}.

\section{Sharpness of \Cref{thm:asym-res-dual-paving}: the case $n=2k-1$}\label{app:sharp}

In this appendix, we show that it is not possible in general to extend \Cref{thm:asym-res-dual-paving} to matroids of rank $k$ on a ground set of size $2k-1$ with a paving dual.  We illustrate this with a series of examples.

\begin{Example}[Rank 3 on 5 elements]\label{ex:notSparsePavingZero}
Let $M$ be the rank three matroid on the five elements $\{0,1,2,3,4\}$ with bases $\{\{0,1,3\}, \{0,2,3\}, \{1,2,3\} ,\{ 0, 1, 4\},\{0,2,4\},\{1,2,4\}\}$.  Observe that $\{3,4\}$ is a circuit of $M$, so $M$ is not paving.  One can compute $3/\widehat{\alpha}(I(M))=6/5$.  Contracting at $3$ or $4$ yields a uniform matroid of rank $2$ on $3$ elements, hence $2/\widehat{\alpha}(I(M/\{3\}))=4/3>6/5$.  In fact, $\widehat{\rho}(I(M))=4/3$.  The dual of $M$ has bases $\{\{2,4\},\{1,4\},\{0,4\},\{2,3\},\{1,3\},\{0,3\}\}$ and circuits the four remaining subsets of $\{0,1,2,3,4\}$ of size two.  Hence, the dual of $M$ is paving and $\widehat{\rho}(I(M))>\frac{k}{\widehat{\alpha}(I(M))}$.  Up to isomorphism, this is the only rank three matroid on five elements with these properties.
\end{Example}

\begin{Example}[Rank 4 on 7 elements]\label{ex:notSparsePaving}
Let $M$ be the rank $4$ matroid on the ground set $E=\{0,1,2,3,4,5,6\}$ which is the dual of the simple rank $3$ matroid $N$ whose geometric representation is shown on row $16$ of \Cref{tbl:Rank3groundset7}. $N$ is a paving matroid since every simple rank $3$ matroid is paving.  There are two disjoint rank $2$ flats of $N$: $F=\{0,1,2,3\}$ and $G=\{4,5,6\}$.  There are $30$ bases of $N$; namely, all size three subsets of $E$ except for $\{0,1,2\},\{0,1,3\},\{0,2,3\},\{1,2,3\},$ and $\{4,5,6\}$.  We claim $G$ is a circuit of $M$.  First, it is a dependent set of $M$ because every basis of $N$ intersects $G$ non-trivially.  It is a minimal dependent set of $M$ because every size two subset of $E$ is independent in $M$.  Since $M$ has rank $4$, $M$ is not paving.

We can compute in Macaulay2~\cite{M2} that $\widehat{\alpha}(I(M))=\frac{7}{3}$ and $\widehat{\rho}(I(M))=\frac{9}{5}>\frac{12}{7}=\frac{4}{\widehat{\alpha}(I(M))}$.

Moreover, contracting any element of $M$ in the circuit $G$ produces a matroid $M'$ of rank $3$ on $6$ elements (which is not a simple matroid) so that $\widehat{\alpha}(I(M'))=\frac{5}{3}$.  Thus, $\frac{3}{\widehat{\alpha}(I(M'))}=\frac{9}{5}=\widehat{\rho}(I(M))$.

Computations in Macaulay2~\cite{M2} indicate that the matroid $M$ described in this example is the \textit{only} matroid of rank $4$ on $7$ elements satisfying that $M$ is simple, $M^*$ is paving, and $\widehat{\rho}(I(M))>\frac{4}{\widehat{\alpha}(I(M))}$.
\end{Example}

The matroids in \Cref{ex:notSparsePavingZero} and \Cref{ex:notSparsePaving} are not sparse paving matroids.  We have verified in Macaulay2~\cite{M2} that if $M$ is a sparse paving matroid of rank $3$ on $5$ elements (see \Cref{tab:master-full} and \Cref{tbl:Rank3groundset5}) or rank $4$ on $7$ elements (see \Cref{tab:master-full} and \Cref{tbl:Rank3groundset7}), then $\widehat{\rho}(I(M))=\frac{\rk(M)}{\widehat{\alpha}(I(M))}$.  One might start to suspect that if $M$ is a rank $k$ sparse paving matroid on $2k-1$ elements, then $\widehat{\rho}(I(M))=\frac{k}{\widehat{\alpha}(I(M))}$.   However, for every $k\ge 5$, there exists a sparse paving matroid of rank $k$ on $2k-1$ elements, which we denote $M_{k,2k-1}$, such that $
\widehat{\rho}(I(M))>\frac{k}{\widehat{\alpha}(I(M))}$.  We first give the construction of $M_{k,2k-1}$ and then explicitly compute the Waldschmidt constant and asymptotic resurgence of its facet ideal, proving that $\widehat{\rho}(I(M_{k,2k-1}))>\frac{k}{\widehat{\alpha}(I(M_{k,2k-1}))}$ (see \Cref{prop:sparse-paving-family}).

\begin{Example}\label{ex:sparse-paving-family}
Let $k \ge 5$ be a positive integer, and $E_k=[2k-4] \cup \{a,b,c\}$.  Note that $|E_k|=2k-1$.   Let $\sigma = (1\,2\,\cdots\,2k-4)$ in the symmetric group $S_{2k-4}$, and set $\gamma=\sigma^2$.  Since $2k-4$ is even, $\gamma$ is a permutation of order $k-2$. Now, for $A \subseteq [2k-4]$, put $\gamma^i(A):=\{\gamma^i(a)~:~ a\in A\}$.  Define $X_i:=\gamma^{i}([k-3])$ for $0\le i\le k-3$.  Observe that $\gamma(X_i)=X_{i+1 \mbox{\scriptsize{ mod }} (k-2)}$ and $\gamma^{-1}(X_i)=X_{i-1 \mbox{\scriptsize{ mod }} (k-2)}$.  Now put $C_i:=\{a,b\}\cup X_i$ for $0\le i\le k-3$, $C_{k-2}:=\{c\} \cup \{2i-1~:~i\in [k-2]\}$ and $C_{k-1}:=\{c\} \cup \{2i~:~i\in [k-2]\}$.

We prove that $\{C_0,\ldots,C_{k-1}\}$ is the collection of circuit-hyperplanes of a sparse paving matroid of rank $k-1$.  We use \Cref{prop:SparsePavingCharacterization}.  We first prove that $|X_i\cap X_j|\le k-5$ for $0\le i<j\le k-3$.  Via the symmetry induced by $\gamma$, it suffices to prove that $|X_0\cap X_j|\le k-5$ for $1\le j\le k-3$.  If $k=2s$ is even, then $|X_0\cap X_i|=2s-3-2i$ for $1\le i\le s-2$, $|X_0\cap X_{k-2-i}|=2s-3-2i$ for $1\le i \le s-2$, and $|X_0\cap X_i|=\emptyset$ only for $i=s-1$.  If $k=2s-1$ is odd, then $|X_0\cap X_i|=2s-4-2i$ for $1\le i\le s-3$, $|X_0\cap X_{k-2-i}|=2s-4-2i$ for $1\le i\le s-3$, and $|X_0\cap X_i|=\emptyset$ if and only if $i=s-2$ or $i=s-1$.  It follows that $|X_i\cap X_j|\le k-5$ for $0\le i<j\le k-3$ and so $|C_i\cap C_j|\le k-3$ for $0\le i<j\le k-3$.

Next, observe that $|C_i\cap C_j|\le k-3$ for $0\le i\le k-3$ and $k-2\le j\le k-1$ since $a,b\notin C_{k-2},C_{k-1}$.  Finally, $|C_{k-2}\cap C_{k-1}|=1\le k-3$.  By \Cref{prop:SparsePavingCharacterization}, $\{C_0,\ldots,C_{k-1}\}$ is the collection of circuit-hyperplanes of a sparse paving matroid of rank $k-1$.  We denote this matroid by $N_{k-1,2k-1}$ and its dual by $M_{k,2k-1}$.
\end{Example}

We next compute the Waldschmidt constant of $I(M_{k,2k-1})$.

\begin{Lemma}\label{lem:walds-sparse-paving-family}
Let $k \ge 5$ be an integer, and let $M_{k,2k-1}$ be the sparse paving matroid of rank $k$ on the ground set $E_k$ constructed in \Cref{ex:sparse-paving-family}. Then, $\widehat{\alpha}(I(M_{k,2k-1}))=\frac{2k+1}{k}$.     
\end{Lemma}
\begin{proof}
First, we prove that $\widehat{\alpha}(I(M)) \le \frac{2k+1}{k}$.  In order to achieve that we show that $${\bf v}=\frac{1}{k}\sum\limits_{i=1}^{2k-4}{\bf  e}_i +\frac{1}{k}{\bf e}_a+\frac{2}{k}{\bf e}_b+\frac{2}{k}{\bf e}_c\in \SP(I(M_{k,2k-1})).$$ By \Cref{cor:SPIM}, the defining inequalities of $\SP(I(M_{k,2k-1}))$ correspond to the circuits of the dual matroid $N_{k-1,2k-1}$. Since $N_{k-1,2k-1}$ is a paving matroid of rank $k-1$, every circuit $C \in \mathcal{C}(N_{k-1,2k-1})$ satisfies $|C| \ge k-1$. 
Moreover, the circuits of size $k-1$ are precisely the circuit-hyperplanes of $N_{k-1,2k-1}$, as described in \Cref{ex:sparse-paving-family}. So,  $\langle \chi_{C},{\bf v}\rangle \ge \frac{|C|}{k} = 1$ for all $C \in \mathcal{C}(N_{k-1,2k-1})$ with $|C|=k$. For $ C \in \mathcal{CH}(N_{k-1,2k-1})$, either $b \in C$ or $c \in C$, which implies  $\langle \chi_{C},{\bf v}\rangle \ge \frac{|C|-1}{k}+\frac{2}{k}= 1$. Therefore, ${\bf v} \in \SP(I(M))$, and hence, by \Cref{thm:SPWaldschmidt}, $\widehat{\alpha}(I(M)) \le \frac{2k+1}{k}$. 

Next, we show that $\widehat{\alpha}(I(M)) \ge  \frac{2k+1}{k}$.  By \Cref{thm:SPWaldschmidt}, it is sufficient to show that  $\sum\limits_{i\in E_k} u_i \ge \frac{2k+1}{k}$ for every ${\bf u} \in \SP(I(M))$. Let ${\bf u} \in \SP(I(M))$. Then, $\langle \chi_{C},{\bf u} \rangle \ge 1$  for all $C \in \mathcal{C}(N_{k-1,2k-1})$. We split this proof into the following cases: 

\textbf{Case 1:} Assume $k$ is odd. Since $k-3$ is even, each $e \in [2k-4]$ belongs to exactly $\frac{k-3}{2}$ of the circuit-hyperplanes $C_0,\ldots,C_{k-3}$. Summing the inequalities $\langle \chi_{C_i}, \mathbf{u} \rangle \ge 1$ for $0 \le i \le k-3$ yields
\begin{equation}\label{eq:ab}
(k-2)({u}_a+{u}_b)
+\frac{k-3}{2}\sum_{i=1}^{2k-4}{u}_i
\ge k-2.
\end{equation}
Similarly, adding the inequalities 
$\langle \chi_{C_{k-2}}, \mathbf{u} \rangle \ge 1$ and 
$\langle \chi_{C_{k-1}}, \mathbf{u} \rangle \ge 1$ gives
\begin{equation}\label{eq:c}
2{u}_c
+\sum_{i=1}^{2k-4}{u}_i
\ge 2.
\end{equation}
Multiplying \eqref{eq:ab} by $2$ and \eqref{eq:c} by $k-2$, and then summing, we obtain
\begin{equation}\label{eq:abc}
 (2k-4)({u}_a+{u}_b+{u}_c)
+(2k-5)\sum_{i=1}^{2k-4}{u}_i
\ge 4k-8.   
\end{equation}
Note that any subset $C \subset [2k-4]$ of size $k$ contains none of the sets $C_i$, 
and hence belongs to $\mathcal{C}(N_{k-1,2k-1})$. 
Moreover, each element $e \in [2k-4]$ is contained in exactly 
${2k-5 \choose k-1}$ such subsets. 
Summing the ${2k-4 \choose k}$ corresponding inequalities yields
\[
{2k-5 \choose k-1} 
\sum_{i=1}^{2k-4} {u}_i 
\ge 
{2k-4 \choose k}.
\]
Consequently,
\[
\sum_{i=1}^{2k-4} {u}_i 
\ge 
\frac{2k-4}{k}.
\]
Summing this with \eqref{eq:abc} gives
\[
(2k-4)({u}_a+{u}_b+{u}_c)
+(2k-4)\sum_{i=1}^{2k-4}{u}_i
\ge 
4k-8+\frac{2k-4}{k}.
\] Thus, \[
{u}_a+{u}_b+{u}_c
+\sum_{i=1}^{2k-4}{u}_i
\ge 
\frac{2k+1}{k}.
\]

\textbf{Case 2:} Assume $k$ is even. Since $k-3$ is odd, each even $e \in [2k-4]$ belongs to exactly $\lfloor{\frac{k-3}{2}}\rfloor$ of the circuit-hyperplanes $C_0,\ldots,C_{k-3}$, and each odd $e \in [2k-4]$ belongs to exactly $\lceil{\frac{k-3}{2}}\rceil$ of the circuit-hyperplanes $C_0,\ldots,C_{k-3}$. Summing the inequalities $\langle \chi_{C_i}, \mathbf{u} \rangle \ge 1$ for $0\le i \le k-3$ yields
\begin{equation}\label{eq:ab-even}
(k-2)({u}_a+{u}_b)
+\left\lceil {\frac{k-3}{2}}\right \rceil \sum_{i=1}^{k-2}{u}_{2i-1}+\left\lfloor {\frac{k-3}{2}}\right \rfloor \sum_{i=1}^{k-2}{u}_{2i}
\ge k-2.
\end{equation}
Now, adding $\lceil{\frac{k-3}{2}}\rceil$ times 
$\langle \chi_{C_{k-2}}, \mathbf{u} \rangle \ge 1$ and $\lceil{\frac{k-3}{2}}\rceil+2$ times 
$\langle \chi_{C_{k-1}}, \mathbf{u} \rangle \ge 1$ to \eqref{eq:ab-even}, we get 
\begin{equation}\label{eq:abc_even}
(k-2)({u}_a+{u}_b) +k u_c
+(k-2) \sum_{i=1}^{k-2}{u}_{2i-1}
+(k-1) \sum_{i=1}^{k-2}{u}_{2i}
\ge 2k-2.
\end{equation} Adding the inequalities $\langle \chi_{C}, \mathbf{u} \rangle \ge 1$, 
with $C=\{2i-1 : i \in [k-2]\} \cup \{k-2,\,2k-4\},$
$\langle \chi_{C_0}, \mathbf{u} \rangle \ge 1$, and 
$\langle \chi_{C_{\frac{k-2}{2}}}, \mathbf{u} \rangle \ge 1$
to \eqref{eq:abc_even} yields
\[
k({u}_a+{u}_b+{u}_c)
+k\sum_{i=1}^{2k-4}{u}_i
\ge 2k+1.
\]
Consequently,
\[
{u}_a+{u}_b+{u}_c
+\sum_{i=1}^{2k-4}{u}_i
\ge \frac{2k+1}{k}.
\]

Thus, in both cases, we have $\sum\limits_{i\in E_k} u_i \ge \frac{2k+1}{k}$ for every ${\bf u} \in \SP(I(M))$, and hence, by \Cref{thm:SPWaldschmidt}, $\widehat{\alpha}(I(M))=\frac{2k+1}{k}$. 
\end{proof}

\begin{Proposition}\label{prop:sparse-paving-family}
Let $k \ge 5$ be an integer, and let $M_{k,2k-1}$ be the sparse paving matroid of rank $k$ on the ground set $E_k$ constructed in \Cref{ex:sparse-paving-family}. Then, $\widehat{\rho}(I(M_{k,2k-1}))= \frac{(k-1)^2}{2k-3}> \frac{k}{\widehat{\alpha}(I(M))}$. 
\end{Proposition}
\begin{proof}
We refer to the construction of $N_{k-1,2k-1}$ and $M_{k,2k-1}$ in \Cref{ex:sparse-paving-family}.  By \Cref{prop:asym-single-contraction-matroid},  $$\widehat{\rho}(I(M))= \max \left\{\frac{\rk(M)}{\widehat{\alpha}(I(M))}, \widehat{\rho}_{c,1}(I(M))\right\}= \max \left\{\frac{k^2}{2k+1}, \widehat{\rho}_{c,1}(I(M))\right\},$$  where the final equality holds by \Cref{lem:walds-sparse-paving-family}. We claim that  $\widehat{\rho}_{c,1}(I(M_{k,2k-1}))= \frac{(k-1)^2}{2k-3}$. Observe that for every $e \in E_{k}=[2k-4]\cup \{a,b,c\}$, there is a circuit hyperplane $C$ of $N_{k-1,2k-1}$ such that $e \notin C$. Therefore,  $N_{k-1,2k-1}\backslash e\neq \mathrm{U}_{k-1,2k-2}$ for any $e\in E_k$.  This implies $M_{k,2k-1}/e=(N_{k-1,2k-1}\backslash e)^*\neq \mathrm{U}_{k-1,2k-2}$ for any $e\in E_k$.  Now, it follows from \Cref{thm:asym-res-uniform-contraction} (b) that $\widehat{\rho}_{c,1}(I(M_{k,2k-1}))\le \frac{(k-1)^2}{2k-3}$. 

Next, consider the matroid $M=M_{k,2k-1}/a$ on the ground set $E_k \setminus \{a\}.$ Observe that $M$ is a sparse paving matroid of rank $k-1$ on the ground set of size $2k-2$. Therefore,  $\mathcal{C}_{k-1}(M^*)=\mathcal{CH}(M_{k,2k-1}^*\setminus a)=\mathcal{CH}(N_{k-1,2k-1}\setminus a)=\{C_{k-2},C_{k-1}\}$.  So,  $A(M^*)=\bigcup\limits_{C \in \mathcal{C}_{k-1}(M^*)}C=C_{k-2} \cup C_{k-1}\neq E_k\setminus \{a\}$. Hence, by \Cref{cor:dual-paving-not-covering}~(c), $\widehat{\rho}(I(M_{k,2k-1}/a))=\frac{(k-1)^2}{2k-3}$, and thus,  $\widehat{\rho}_{c,1}(I(M_{k,2k-1}))=\frac{(k-1)^2}{2k-3}$.

Observe that $\frac{(k-1)^2}{2k-3}>\frac{k^2}{2k+1}= \frac{k}{\widehat{\alpha}(I(M))}$.  Hence, by \Cref{prop:asym-single-contraction-matroid}, $\widehat{\rho}(I(M_{k,2k-1}))=\frac{(k-1)^2}{2k-3}$.
\end{proof}

\section{Asymptotic resurgence of simple matroids on small ground sets}\label{app:small-rank}

In this section, we first record in \Cref{tab:master-full} coarse data for all simple matroids on ground sets of size 5,6,7, or 8, mainly with an eye towards \Cref{ques:rhohat=alpha/alphahat}.  We used the database of all matroids on ground sets up to size seven in the Matroids package~\cite{MatroidsSource,MatroidsArticle} in Macaulay2~\cite{M2}, and computed the asymptotic resurgence of the facet ideals of these matroids using the algorithm outlined in \Cref{app:algo}.

In \Cref{tbl:Rank3groundset5},\Cref{tbl:Rank3groundset6}, and \Cref{tbl:Rank3groundset7}, we give exact asymptotic resurgence values of the facet and Stanley-Reisner ideals of all simple matroids of rank three on ground sets up to size seven and compare these values to \GHV{}.  The facet ideal of a rank two matroid achieves equality in \GHV{} by \Cref{lem:rank-two}, so we only consider matroids of rank at least three.

We briefly discuss the matroids of rank three on three or four elements.  The only simple matroid of rank three on three elements is the uniform matroid $\mathrm{U}_{3,3}$.  On four elements, there are two non-isomorphic simple matroids of rank three~\cite[Table~4]{MMIB2012}.  One of these is $\mathrm{U}_{3,4}$, and the other is the almost uniform matroid $\mathrm{U}_{3,4}^{-}$ discussed in \Cref{ex:rhohatnotalphaoveralphahat}.  Both of these have a facet ideal which achieves equality in \GHV{}.

There are $4$ (respectively $9$, $23$) non-isomorphic simple matroids of rank three on $5$ (respectively $6$, $7$) elements (see for instance~\cite[Table~4]{MMIB2012}).  The paper~\cite{MMIB2012} provides an associated database of matroids, available as a Python package, which we accessed via Sage~\cite{sagemath}.

In \Cref{tbl:Rank3groundset5}, \Cref{tbl:Rank3groundset6}, and \Cref{tbl:Rank3groundset7}, the second column is a \textit{geometric representation} of the corresponding matroid (see~\cite[Section~1.5]{Oxley2011}).  The geometric representation consists of a number of points (the ground set) and a number of `lines'  (possibly curved) containing at least three points of the ground set.  Since the matroid is rank three and simple, every point is a rank one flat, and every pair of points determines a rank two flat.  The lines represent rank-two flats with more than two elements.  Any subset of three points not on a line is a basis of the matroid.  Every rank three simple matroid has a geometric representation in terms of points and lines in this way~\cite[Proposition~1.5.6]{Oxley2011}.  We produced the geometric representations using the matroid functions available in Sage.

The third column in \Cref{tbl:Rank3groundset5}, \Cref{tbl:Rank3groundset6}, and \Cref{tbl:Rank3groundset7} records the name or notation for the corresponding matroid given by Oxley in~\cite{Oxley2011} (if we were able to find one) along with paving properties of the matroid.

\begin{Remark}[Paving matroids and initial degree]\label{rem:appinitial}
If $M$ is a rank $3$ matroid, then $\alpha(I(M))=3$ by definition.  If $M$ is a rank $3$ simple matroid, we claim that $\alpha(I_{\Delta(M)})=3$ unless $M$ is a rank $3$ uniform matroid, in which case $\alpha(I_{\Delta(M)})=4$.  We see this as follows.  \textit{Simple} matroids are, by definition, those matroids that have no loops or parallel elements.  Equivalently, a simple matroid is a matroid with no circuits of size one or two. Thus, a rank three simple matroid has circuits of size $3$ and $4$ only (i.e., a rank three simple matroid is \textit{paving} -- see \Cref{rem:sparsepaving} for the definition of a paving matroid).  If $M$ only has circuits of size four, it is a uniform matroid and $\alpha(I_{\Delta(M)})=4$.  Otherwise, $M$ has at least one circuit of size $3$ and so $\alpha(I_{\Delta(M)})=3$.
\end{Remark}

\begin{Remark}[Matroids with a paving dual]\label{rem:appsparsepaving} 
If a rank three matroid on at least five elements has a dual that is paving, then its facet ideal achieves equality in \GHV{} by \Cref{thm:asym-res-dual-paving}.  This explains the coincidence of columns four and five in \Cref{tbl:Rank3groundset5},\Cref{tbl:Rank3groundset6}, and \Cref{tbl:Rank3groundset7} whenever the matroid has a dual that is paving.  Observe that a simple rank three matroid has a dual that is paving if and only if it is sparse paving (this property is recorded in column 3 of Tables~\ref{tbl:Rank3groundset5} --\ref{tbl:Rank3groundset7}).
\end{Remark}

\begin{Remark}[Waldschmidt constant for Stanley-Reisner ideals of sparse paving matroids]\label{rem:appwaldsparsepaving}
If $M$ is a rank $k$ sparse paving matroid  on a ground set of size $n$, then $\widehat{\alpha}(I_{\Delta(M)})=\frac{n}{n-k}$ by ~\cite[Corollary~6.4]{DFKT24} or~\cite[Corollary 7.11]{ManNgu}.  Together with \Cref{rem:appinitial} and \Cref{rem:appsparsepaving}, this explains the values of $\frac{\alpha(I_{\Delta(M)})}{\widehat{\alpha}(I_{\Delta(M)})}$ in rows 1-3 of \Cref{tbl:Rank3groundset5}, rows 1-6 of \Cref{tbl:Rank3groundset6}, and rows 1-14 of \Cref{tbl:Rank3groundset7}.
\end{Remark}

\begin{Remark}[Waldschmidt constant for Stanley-Reisner ideals and the Tutte polynomial]\label{rem:appSRTutte}
In~\cite[Theorem~3.4]{DFKT24}, a formula for the Waldschmidt constant of the Stanley-Reisner ideal of a matroid $M$ is given in terms of the so-called \textit{generalized Hamming weights} of $M$.  It follows from~\cite[Section~4]{JRV2016} that the generalized Hamming weights of $M$ can be detected from the Tutte polynomial of $M$ (\cite{JRV2016} calls generalized Hamming weights simply \textit{higher weights}).  Hence, putting these results together, it follows that the Waldschmidt constant of the Stanley-Reisner ideal of a matroid can be determined from its Tutte polynomial.
\end{Remark}

\begin{Remark}[Waldschmidt constant of the facet ideal is not a Tutte polynomial invariant]\label{rem:appFTutte}
The matroids $Q_6$ and $R_6$ in \Cref{tbl:Rank3groundset6} are known to have the same Tutte polynomial, namely $$T_{Q_6}(x,y)=x^3+y^3+3x^2+2xy+3y^2+4x+4y=T_{R_6}(x,y).$$  However, we can see from \Cref{tbl:Rank3groundset6} that $\widehat{\alpha}(I(Q_6))=5/3$ and $\widehat{\alpha}(I(R_6))=2$.  Thus, the Waldschmidt constant of the facet ideal of a matroid cannot be detected from the Tutte polynomial of the matroid, unlike the Waldschmidt constant of the Stanley-Reisner ideal (see \Cref{rem:appSRTutte}).  Observe from rows 3 and 4 of \Cref{tbl:Rank3groundset6} that $\widehat{\alpha}(I_{\Delta(Q_6)})=\widehat{\alpha}(I_{\Delta(R_6)})=2$.
\end{Remark}

\begin{Remark}[Asymptotic resurgence is not a Tutte polynomial invariant]\label{rem:appAsymTutte}
Following up on \Cref{rem:appFTutte}, observe that the asymptotic resurgences $\widehat{\rho}(I(Q_6))$ and $\widehat{\rho}(I(R_6))$ are also different.  It follows that the asymptotic resurgence of the facet ideal of a matroid cannot be detected from its Tutte polynomial.  By \Cref{prop:Alex duality facet}, it follows that the asymptotic resurgence of the Stanley-Reisner ideal of a matroid also cannot be detected from its Tutte polynomial.  
\end{Remark}

\begin{Remark}[First two rows and last row]
The entries in the first row of \Cref{tbl:Rank3groundset5}, \Cref{tbl:Rank3groundset6}, and \Cref{tbl:Rank3groundset7} follow from \Cref{prop:uniform-matroid}.  The matroids in the second row are almost-uniform matroids on a ground set of size $n$ (where $n=5,6,7$), and their duals are almost-uniform matroids. Therefore, the entries of row two in \Cref{tbl:Rank3groundset5}, \Cref{tbl:Rank3groundset6}, and \Cref{tbl:Rank3groundset7} are obtained by \Cref{rem:appwaldsparsepaving},~\Cref{lem:relaxation-to-uniform}, and~\Cref{thm:relaxation-to-uniform}.

The last row of \Cref{tbl:Rank3groundset5}, \Cref{tbl:Rank3groundset6}, and \Cref{tbl:Rank3groundset7} corresponds to the direct sum $\mathrm{U}_{2,n-1}\oplus \mathrm{U}_{1,1}$ for $n=5,6,7$.  Thus the entries of this row for the facet ideal can be computed using \Cref{prop:uniform-matroid} and \Cref{ex:direct-sum}.

In fact, $\mathrm{U}_{2,n-1}\oplus \mathrm{U}_{1,1}$ is the only rank $3$ simple matroid on $n$ elements that can be written as a direct sum of matroids, unless $n=3$ in which case we also have $\mathrm{U}_{1,1}\oplus \mathrm{U}_{1,1} \oplus \mathrm{U}_{1,1}$.  Every other direct sum of a rank one matroid with a rank two matroid will have parallel elements.
\end{Remark}

\begin{Remark}[Weak Order]\label{rem:chainsinweakorder}
Suppose $M$ and $N$ are rank $3$ matroids with geometric representations $G$ and $H$, respectively.  Then $M\preceq N$ in weak order if and only if every basis in $M$ is also a basis in $N$.  That is, if three points do not lie on a line in the representation $G$, then those three points do not lie on a line in the representation $H$.  Equivalently, if three points are on a line in $H$, then they must be on a line in $G$.  That is, the lines of $H$ are contained in the lines of $G$.  Using this interpretation of weak order in terms of geometric representations, one can determine how the matroids in \Cref{tbl:Rank3groundset5}, \Cref{tbl:Rank3groundset6}, and \Cref{tbl:Rank3groundset7} are related to each other in the weak order.  It is interesting to note that there are chains in the weak order that all attain the upper bound of \Cref{cor:asym-upper-bound}.  For instance, if we let $M_i$ denote the matroid with geometric representation depicted in row $i$ of \Cref{tbl:Rank3groundset7}, observe that $M_6\prec M_4\prec M_2$ and $\widehat{\rho}(I(M_i))=2$ for $i=2,4,6$.  Here $2$ is the upper bound (from \Cref{cor:asym-upper-bound}) on both facet and Stanley-Reisner ideals of non-uniform rank $3$ matroids on a ground set of size $7$.  Even longer chains appear for the Stanley-Reisner ideal.  Observe $M_{11}\prec M_7 \prec M_4 \prec M_2$ and $\widehat{\rho}(I_{\Delta(M_i)})=2$ for $i=2,4,7,11$.  The same occurs for $M_{23}\prec M_{21}\prec M_{15}\prec M_2$.
\end{Remark}

\newpage

\begin{table}[h!]
\centering
\small
\renewcommand{\arraystretch}{2}
\setlength{\tabcolsep}{5pt}

\caption{Counts of non-isomorphic simple matroids, sparse paving matroids, and matroids with a paving dual, on ground sets up to size eight, stratified by rank and organized according to
\Cref{ques:rhohat=alpha/alphahat}.}
\label{tab:master-full}

\begin{tabular}{|c|c|c|c|c|c|c|c|}

\hline
\multicolumn{8}{|c|}{Number of simple non-isomorphic matroids $M$ satisfying:} \\
\hline

\multicolumn{2}{|c|}{} &
\multicolumn{3}{c|}{$\widehat{\rho}(I(M))=\dfrac{\alpha(I(M))}{\widehat{\alpha}(I(M))}$} &
\multicolumn{3}{c|}{$\widehat{\rho}(I(M))>\dfrac{\alpha(I(M))}{\widehat{\alpha}(I(M))}$} \\
\cline{3-8}

\hline

$|E|$ & rank &Total& Sparse Paving & Paving dual &Total & Sparse Paving & Paving Dual \\ 
\hline 

5 & 2 & 1 & 1 & 1 & 0 &0 &0 \\ \hline
  & 3 & 3 & 3 & 3 & 1 & 0 & 0\\ \hline
  & 4 & 1 & 1 & 1  & 2& 1 & 2\\ \hline
  & Total & 5 & 5 & 5 & 3 & 1 & 2\\ \hline\hline

6 & 2 & 1 & 1 & 1 & 0 & 0 & 0\\ \hline
  & 3 & 7 & 6 & 6 & 2 & 0 & 0 \\ \hline
  & 4 & 5 & 2 & 5 & 6 & 2 & 2 \\ \hline
  & 5 & 1 & 1 & 1 &  3 & 1 & 3 \\ \hline
  & Total & 14 & 10 & 13 & 11 & 3 & 5\\ \hline\hline

7 & 2 & 1 & 1 &1  & 0 & 0 &0\\ \hline
  & 3 & 19 & 14 & 14 & 4 & 0 & 0\\ \hline
  & 4 & 28 & 14 & 19 & 21 & 0 & 1\\ \hline
  & 5 & 6 & 2 & 6 &  16 & 2 & 5\\ \hline
  & 6 & 1 & 1 & 1 &  4 & 1 & 4 \\ \hline
  & Total & 55 & 32 & 37 & 45 & 3 & 12\\ \hline\hline

8 & 2 & 1 & 1 & 1 & 0 & 0 & 0\\ \hline
  & 3 & 62 & 32 & 32 & 6 & 0 & 0 \\ \hline
  & 4 & 521 & 270 & 317 & 96 & 0 & 0 \\ \hline
  & 5 & 84 & 25 & 52 & 133 & 7 & 13\\ \hline
  & 6 & 9 & 1 & 9 & 31 & 4 & 9\\ \hline
  & 7 & 1 & 1 & 1 & 5 &1 & 5\\ \hline
  & Total & 678 & 330 & 412 & 271 & 12 & 27\\ \hline

\end{tabular}
\vspace{2em}

\end{table}

\newpage
\begin{longtable}[c]{|c| c | p{1.6 cm} | c | c | c | c |}
\caption{
Asymptotic resurgence of facet and Stanley-Reisner ideals of the four simple matroids (up to isomorphism) of rank three on a ground set of size five.\label{tbl:Rank3groundset5}
}\\

\hline
Row \# &
Geometric representation of $M$ & Name/ Properties  
& $\frac{\alpha(I(M))}{\widehat{\alpha}(I(M))}$ & $\widehat{\rho}(I(M))$ & $\frac{\alpha(I_{\Delta(M)})}{\widehat{\alpha}(I_{\Delta(M)})}$ & $\widehat{\rho}(I_{\Delta(M)})$ \\
 \hline

\endfirsthead

\hline
 \multicolumn{7}{|c|}{Continuation of Table \ref{tbl:Rank3groundset5}}\\
\hline
Row \# & Geometric representation of $M$ & Name in~\cite{Oxley2011} & $\frac{\alpha(I(M))}{\widehat{\alpha}(I(M))}$ & $\widehat{\rho}(I(M))$ & $\frac{\alpha(I_{\Delta(M)})}{\widehat{\alpha}(I_{\Delta(M)})}$ & $\widehat{\rho}(I_{\Delta(M)})$ \\
 \hline

\endhead

\hline

\endfoot

\hline

\endlastfoot

    \raisebox{45 pt}{1} & \includegraphics[scale=.3]{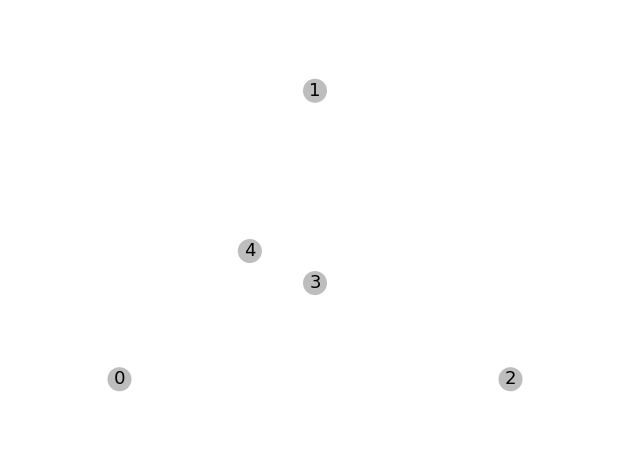}  & \raisebox{45 pt}{\parbox{2 cm}{$\mathrm{U}_{3,5}$\\[10 pt] sparse\\paving}} & \raisebox{45 pt}{$\dfrac{9}{5}$} & \raisebox{45 pt}{$\dfrac{9}{5}$} & \raisebox{45 pt}{$\dfrac{8}{5}$} & \raisebox{45 pt}{$\dfrac{8}{5}$} \\
    \hline
    \raisebox{45 pt}{2} & \includegraphics[scale=.3]{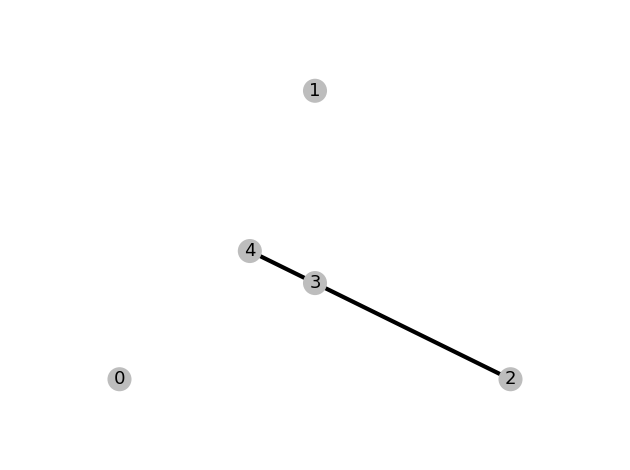} & \raisebox{45 pt}{\parbox{2 cm}{ sparse\\paving}} & \raisebox{45 pt}{$\dfrac{3}{2}$} & \raisebox{45 pt}{$\dfrac{3}{2}$} & \raisebox{45 pt}{$\dfrac{6}{5}$} & \raisebox{45 pt}{$\dfrac{3}{2}$}  \\
    \hline
    \raisebox{45 pt}{3} & \includegraphics[scale=.3]{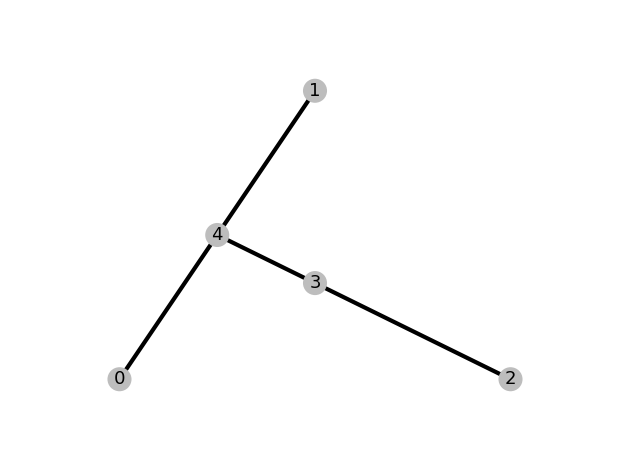} & \raisebox{45 pt}{\parbox{2 cm}{ sparse\\paving}}  & \raisebox{45 pt}{$\dfrac{3}{2}$} & \raisebox{45 pt}{$\dfrac{3}{2}$} & \raisebox{45 pt}{$\dfrac{6}{5}$} & \raisebox{45 pt}{$\dfrac{4}{3}$}   \\
    \hline
    \raisebox{45 pt}{4} & \includegraphics[scale=.3]{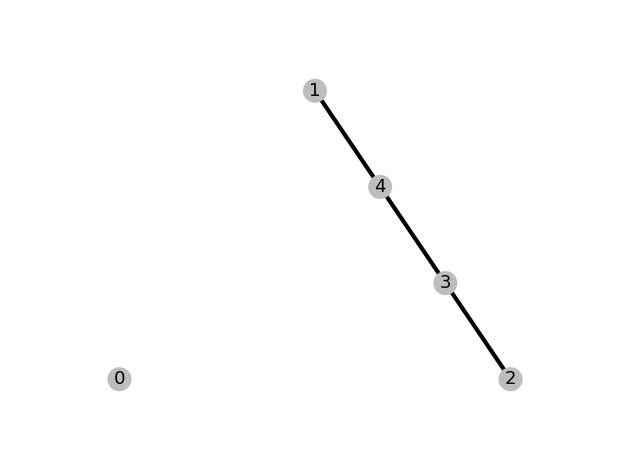} & \raisebox{45 pt}{\parbox{2 cm}{ dual not\\paving}} & \raisebox{45 pt}{$\dfrac{9}{7}$} & \raisebox{45 pt}{$\dfrac{3}{2}$} & \raisebox{45 pt}{$\dfrac{3}{2}$} & \raisebox{45 pt}{$\dfrac{3}{2}$}\\
\end{longtable}

\begin{longtable}[c]{|c| c | p{1.6 cm} | c | c | c | c |}
\caption{
Asymptotic resurgence of facet and Stanley-Reisner ideals of the nine simple matroids (up to isomorphism) of rank three on a ground set of size six.\label{tbl:Rank3groundset6}
}\\

\hline
Row \# &
Geometric representation of $M$ & Name/ Properties  
& $\frac{\alpha(I(M))}{\widehat{\alpha}(I(M))}$ & $\widehat{\rho}(I(M))$ & $\frac{\alpha(I_{\Delta(M)})}{\widehat{\alpha}(I_{\Delta(M)})}$ & $\widehat{\rho}(I_{\Delta(M)})$ \\
 \hline

\endfirsthead

\hline
 \multicolumn{7}{|c|}{Table \ref{tbl:Rank3groundset6}, continued}\\
\hline
Row \# & Geometric representation of $M$ & Name/ Properties  
& $\frac{\alpha(I(M))}{\widehat{\alpha}(I(M))}$ & $\widehat{\rho}(I(M))$ & $\frac{\alpha(I_{\Delta(M)})}{\widehat{\alpha}(I_{\Delta(M)})}$ & $\widehat{\rho}(I_{\Delta(M)})$ \\
 \hline

\endhead

\hline

\endfoot

\hline

\endlastfoot

\raisebox{45 pt}{1} & \includegraphics[scale=.3]{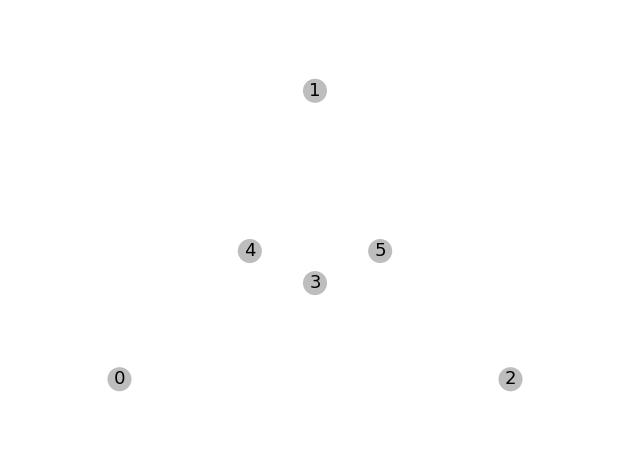} & \raisebox{45 pt}{\parbox{2 cm}{$\mathrm{U}_{3,6}$\\[10 pt] sparse\\paving}} & \raisebox{45 pt}{$2$} & \raisebox{45 pt}{$2$} & \raisebox{45 pt}{$2$} & \raisebox{45 pt}{$2$} \\
\hline
\raisebox{45 pt}{2} &\includegraphics[scale=.3]{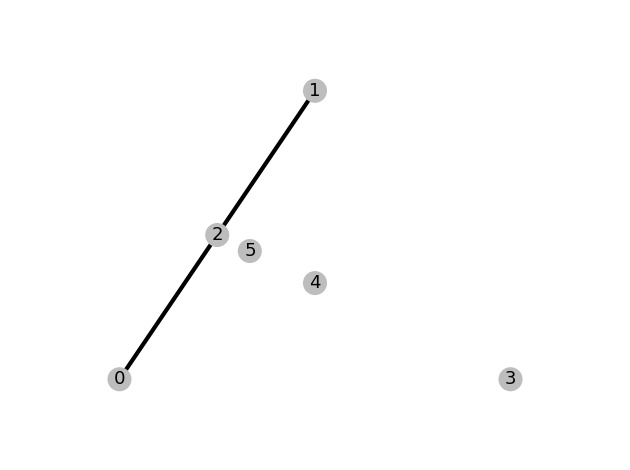} & \raisebox{45 pt}{\parbox{2 cm}{$P_6$\\[10 pt] sparse\\paving}}& \raisebox{45 pt}{$\dfrac{9}{5}$}  & \raisebox{45 pt}{$\dfrac{9}{5}$} & \raisebox{45 pt}{$\dfrac{3}{2}$} & \raisebox{45 pt}{$\dfrac{9}{5}$} \\
\hline
\raisebox{45 pt}{3} & \includegraphics[scale=.3]{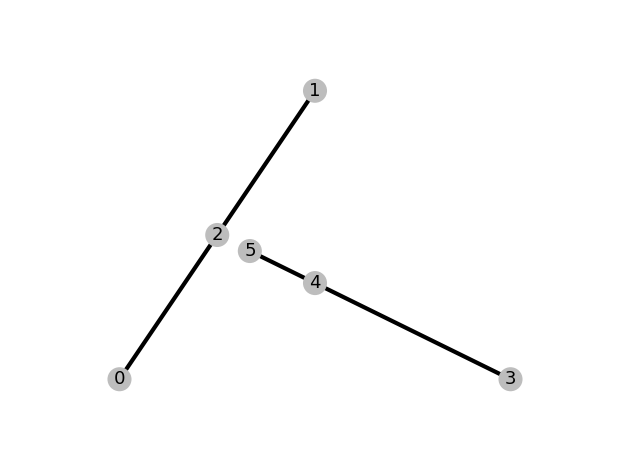} & \raisebox{45 pt}{\parbox{2 cm}{$R_6$\\[10 pt] sparse\\paving}} & \raisebox{45 pt}{$\dfrac{3}{2}$} & \raisebox{45 pt}{$\dfrac{3}{2}$} & \raisebox{45 pt}{$\dfrac{3}{2}$} & \raisebox{45 pt}{$\dfrac{3}{2}$} \\
\hline
\raisebox{45 pt}{4} & \includegraphics[scale=.3]{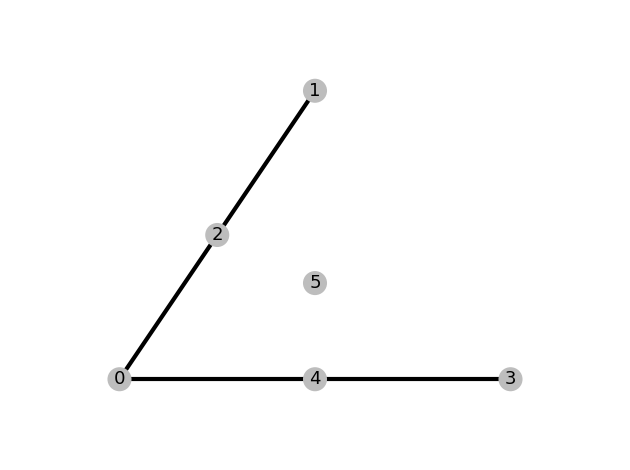} & \raisebox{45 pt}{\parbox{2 cm}{$Q_6$\\[10 pt] sparse\\ paving}} & \raisebox{45 pt}{$\dfrac{9}{5}$} & \raisebox{45 pt}{$\dfrac{9}{5}$} & \raisebox{45 pt}{$\dfrac{3}{2}$} & \raisebox{45 pt}{$\dfrac{9}{5}$} \\
\hline
\raisebox{45 pt}{5} & \includegraphics[scale=.3]{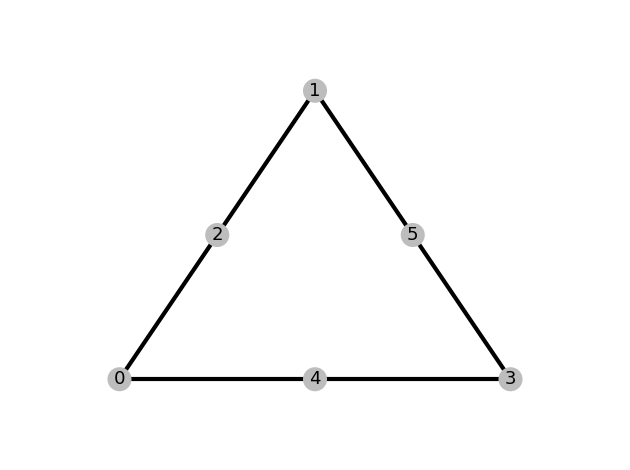} & \raisebox{45 pt}{\parbox{2 cm}{Rank $3$\\ whirl\\[10 pt] sparse\\ paving}} & \raisebox{45 pt}{$\dfrac{5}{3}$} & \raisebox{45 pt}{$\dfrac{5}{3}$} & \raisebox{45 pt}{$\dfrac{3}{2}$} & \raisebox{45 pt}{$\dfrac{5}{3}$} \\
\hline
\raisebox{45 pt}{6} & \includegraphics[scale=.3]{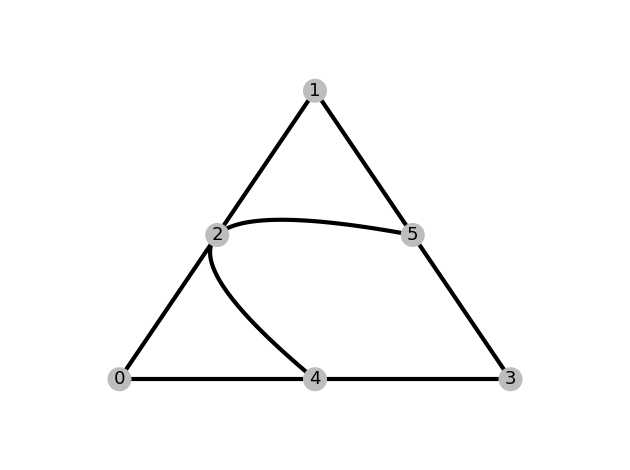} & \raisebox{45 pt}{\parbox{2 cm}{$\Theta_3\\ \mbox{or}\\ M(K_4)$\\[10 pt] sparse\\paving}} & \raisebox{45 pt}{$\dfrac{3}{2}$} & \raisebox{45 pt}{$\dfrac{3}{2}$} & \raisebox{45 pt}{$\dfrac{3}{2}$} & \raisebox{45 pt}{$\dfrac{3}{2}$} \\
\hline
\raisebox{45 pt}{7} & \includegraphics[scale=.3]{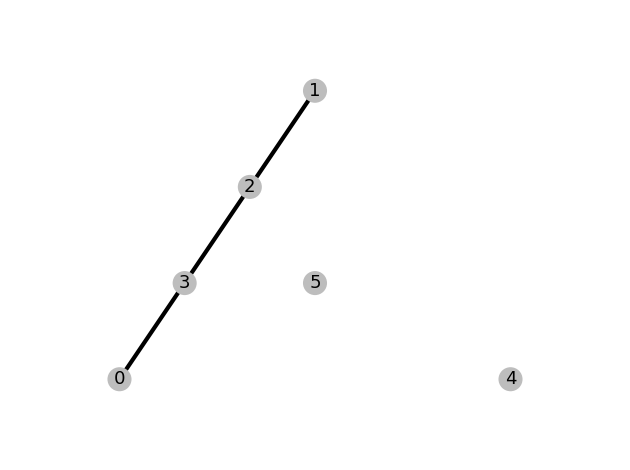} & \raisebox{45 pt}{\parbox{2 cm}{dual not \\paving}} & \raisebox{45 pt}{$\dfrac{3}{2}$} & \raisebox{45 pt}{$\dfrac{8}{5}$} & \raisebox{45 pt}{$\dfrac{3}{2}$} & \raisebox{45 pt}{$\dfrac{9}{5}$} \\
\hline
\raisebox{45 pt}{8} & \includegraphics[scale=.3]{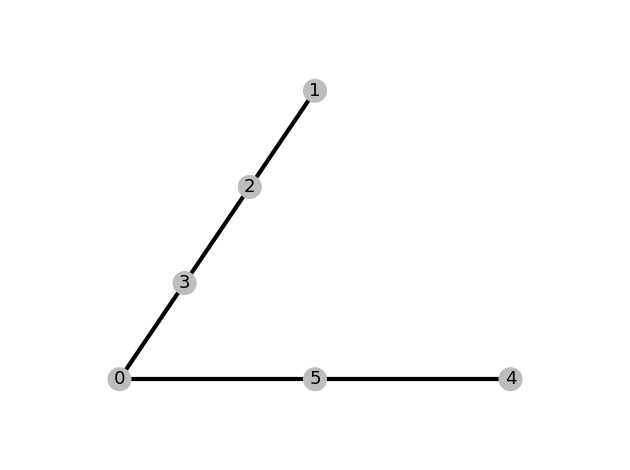} & \raisebox{45 pt}{\parbox{2 cm}{dual not\\paving}} & \raisebox{45 pt}{$\dfrac{3}{2}$} & \raisebox{45 pt}{$\dfrac{3}{2}$} & \raisebox{45 pt}{$\dfrac{3}{2}$} & \raisebox{45 pt}{$\dfrac{3}{2}$} \\
\hline
\raisebox{45 pt}{9} & \includegraphics[scale=.3]{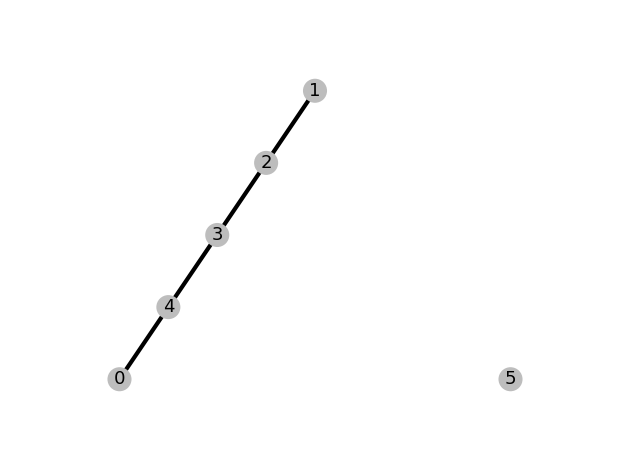} & \raisebox{45 pt}{\parbox{2 cm}{dual not\\paving}} & \raisebox{45 pt}{$\dfrac{4}{3}$} & \raisebox{45 pt}{$\dfrac{8}{5}$} & \raisebox{45 pt}{$\dfrac{9}{5}$} & \raisebox{45 pt}{$\dfrac{9}{5}$} \\
\end{longtable}

\begin{longtable}[c]{|c| p{4.7 cm} | p{1.6 cm} | c | c | c | c|}
\caption{
Asymptotic resurgence of the facet and Stanley-Reisner ideals of the twenty-three simple matroids (up to isomorphism) of rank three on a ground set of size seven.\label{tbl:Rank3groundset7}
}\\

\hline
Row \# & Geometric representation of $M$ & Name/ Properties  
& $\frac{\alpha(I(M))}{\widehat{\alpha}(I(M))}$ & $\widehat{\rho}(I(M))$ & $\frac{\alpha(I_{\Delta(M)})}{\widehat{\alpha}(I_{\Delta(M)})}$ & $\widehat{\rho}(I_{\Delta(M)})$ \\
 \hline

\endfirsthead

\hline
 \multicolumn{7}{|c|}{Table \ref{tbl:Rank3groundset7}, continued}\\
\hline
Row \# & Geometric representation of $M$ & Name/ Properties 
& $\frac{\alpha(I(M))}{\widehat{\alpha}(I(M))}$ & $\widehat{\rho}(I(M))$ & $\frac{\alpha(I_{\Delta(M)})}{\widehat{\alpha}(I_{\Delta(M)})}$ & $\widehat{\rho}(I_{\Delta(M)})$ \\
 \hline

\endhead

\hline

\endfoot

\hline

\endlastfoot

\raisebox{45 pt}{1} & \includegraphics[scale=.3]{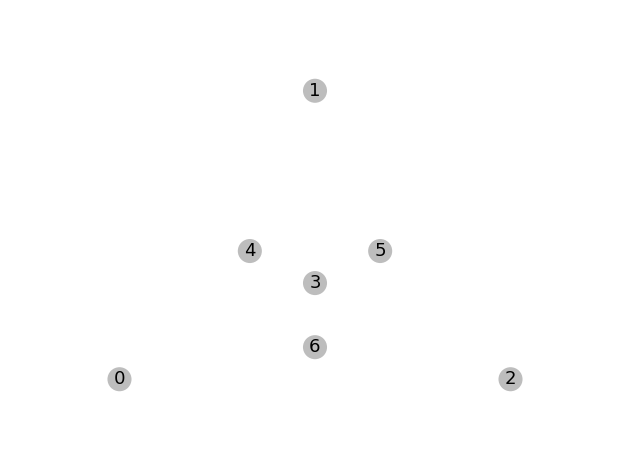} & \raisebox{45 pt}{\parbox{2cm}{$\mathrm{U}_{3,7}$\\[10 pt] sparse\\paving }} & \raisebox{45 pt}{$\dfrac{15}{7}$}  & \raisebox{45 pt}{$\dfrac{15}{7}$} & \raisebox{45 pt}{$\dfrac{16}{7}$} & \raisebox{45 pt}{$\dfrac{16}{7}$} \\
\hline
\raisebox{45 pt}{2} & \includegraphics[scale=.3]{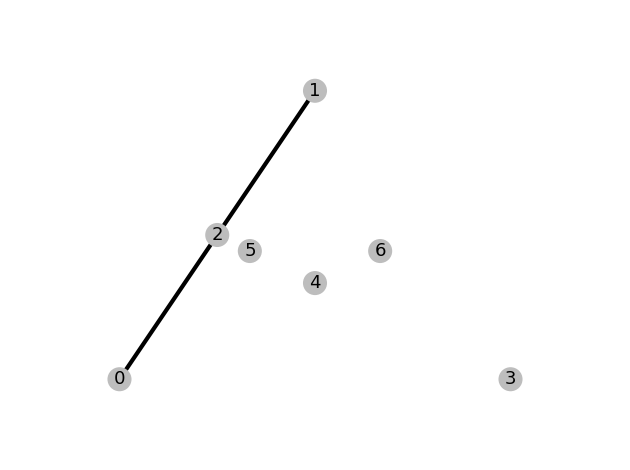} & \raisebox{45 pt}{\parbox{2 cm}{sparse\\paving}} & \raisebox{45 pt}{2} & \raisebox{45 pt}{2} & \raisebox{45 pt}{$\dfrac{12}{7}$} & \raisebox{45 pt}{2} \\
\hline
\raisebox{45 pt}{3} &
\includegraphics[scale=.3]{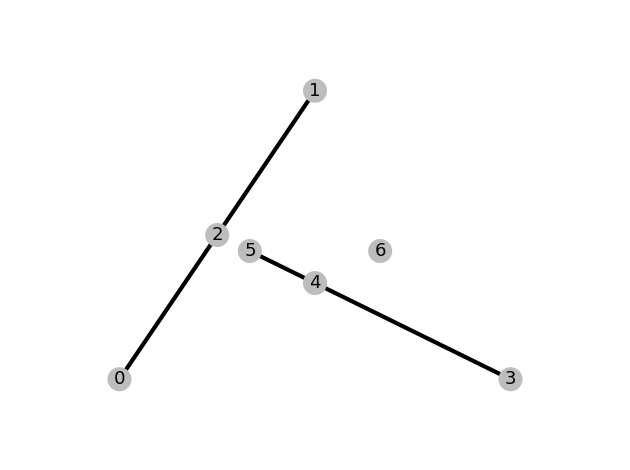} & \raisebox{45 pt}{\parbox{2 cm}{sparse\\paving}}
& \raisebox{45 pt}{$\dfrac{15}{8}$} & \raisebox{45 pt}{$\dfrac{15}{8}$} & \raisebox{45 pt}{$\dfrac{12}{7}$} & \raisebox{45 pt}{2} \\
\hline
\raisebox{45 pt}{4} &
\includegraphics[scale=.3]{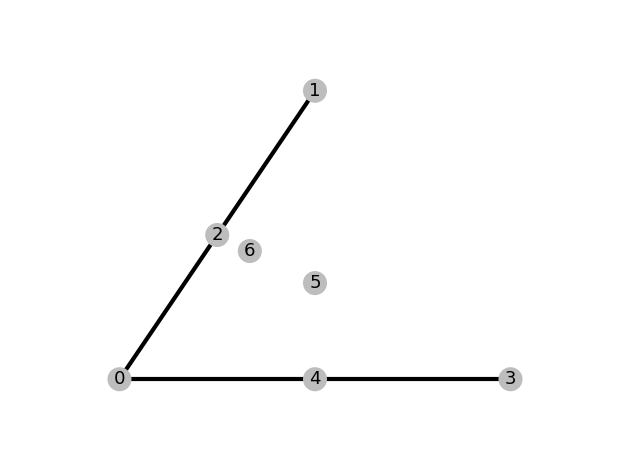} & \raisebox{45 pt}{\parbox{2 cm}{sparse\\paving}} & \raisebox{45 pt}{$2$} & \raisebox{45 pt}{$2$} & \raisebox{45 pt}{$\dfrac{12}{7}$} & \raisebox{45 pt}{$2$} \\
\hline
\raisebox{45 pt}{5} &
\includegraphics[scale=.3]{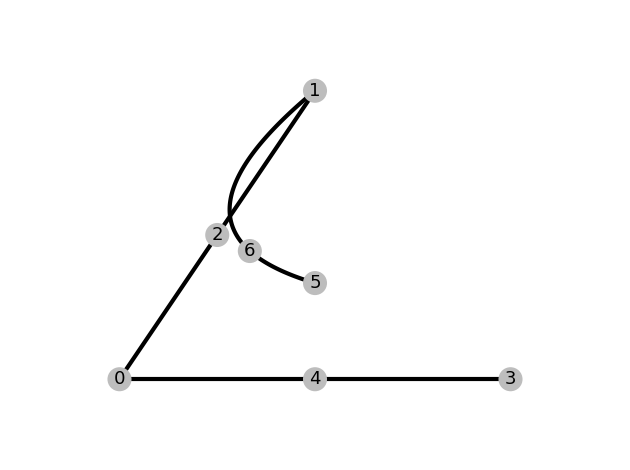} & \raisebox{45 pt}{\parbox{2 cm}{sparse\\paving}} & \raisebox{45 pt}{$\dfrac{24}{13}$} & \raisebox{45 pt}{$\dfrac{24}{13}$} & \raisebox{45 pt}{$\dfrac{12}{7}$} & \raisebox{45 pt}{$\dfrac{28}{15}$} \\
\hline
\raisebox{45 pt}{6} &
\includegraphics[scale=.3]{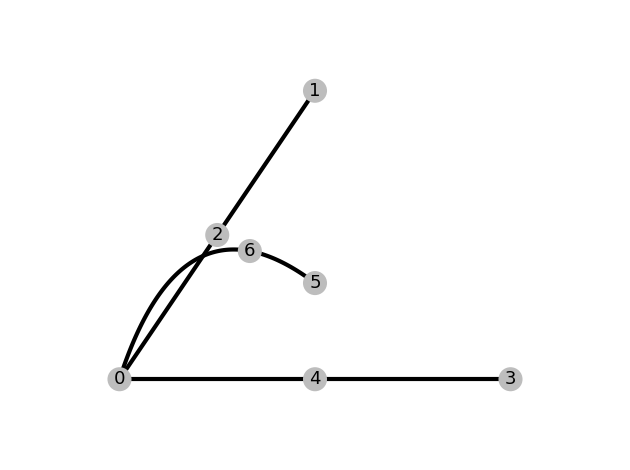} & \raisebox{45 pt}{\parbox{2 cm}{Rank-$3$ \\ free spike\\[10 pt] sparse\\ paving}} & \raisebox{45 pt}{$2$} & \raisebox{45 pt}{$2$} & \raisebox{45 pt}{$\dfrac{12}{7}$} & \raisebox{45 pt}{$2$} \\
\hline
\raisebox{45 pt}{7} &
\includegraphics[scale=.3]{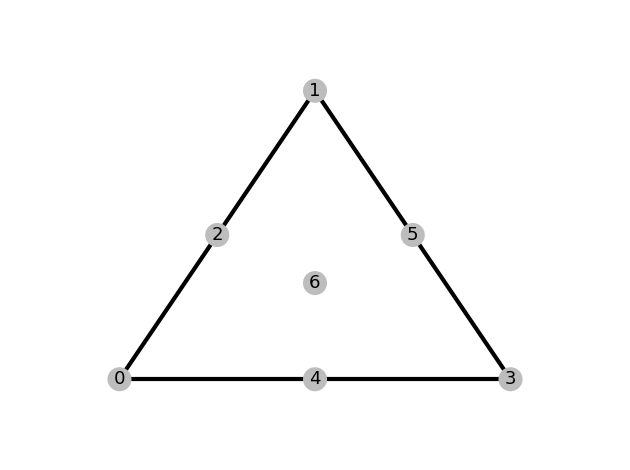} &\raisebox{45 pt}{\parbox{2 cm}{sparse\\paving}} & \raisebox{45 pt}{$\dfrac{28}{11}$} & \raisebox{45 pt}{$\dfrac{28}{11}$} &  \raisebox{45 pt}{$\dfrac{12}{7}$} & \raisebox{45 pt}{$2$} \\
\hline
\raisebox{45 pt}{8} &
\includegraphics[scale=.3]{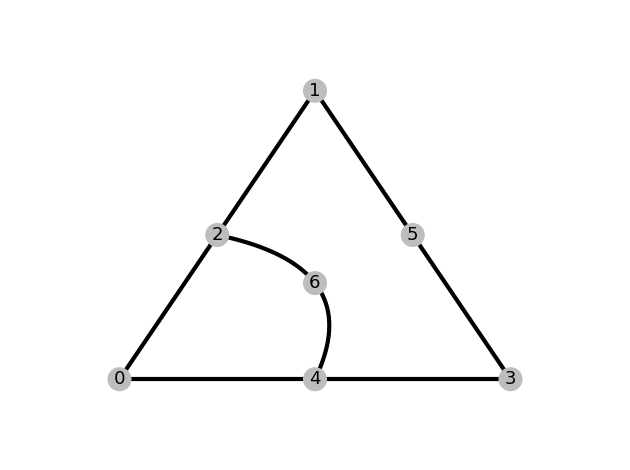} & \raisebox{45 pt}{\parbox{2 cm}{sparse\\paving}} & \raisebox{45 pt}{$\dfrac{9}{5}$} & \raisebox{45 pt}{$\dfrac{9}{5}$} & \raisebox{45 pt}{$\dfrac{12}{7}$} & \raisebox{45 pt}{$\dfrac{20}{11}$} \\
\hline
\raisebox{45 pt}{9} &
\includegraphics[scale=.3]{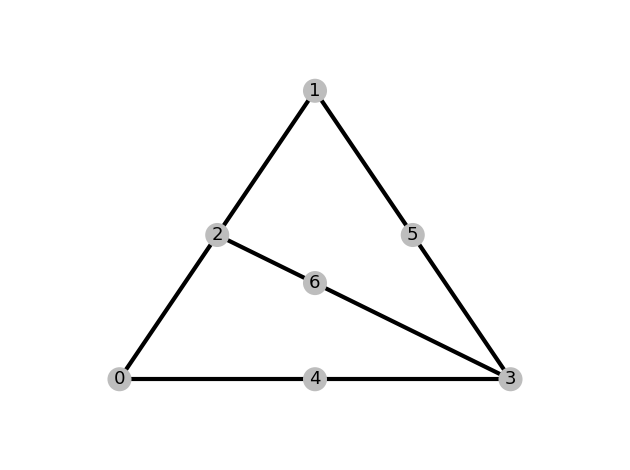} & \raisebox{45 pt}{\parbox{2 cm}{sparse\\paving}} & \raisebox{45 pt}{$\dfrac{15}{8}$} & \raisebox{45 pt}{$\dfrac{15}{8}$} & \raisebox{45 pt}{$\dfrac{12}{7}$} & \raisebox{45 pt}{$\dfrac{36}{19}$} \\
\hline
\raisebox{45 pt}{10} &
\includegraphics[scale=.3]{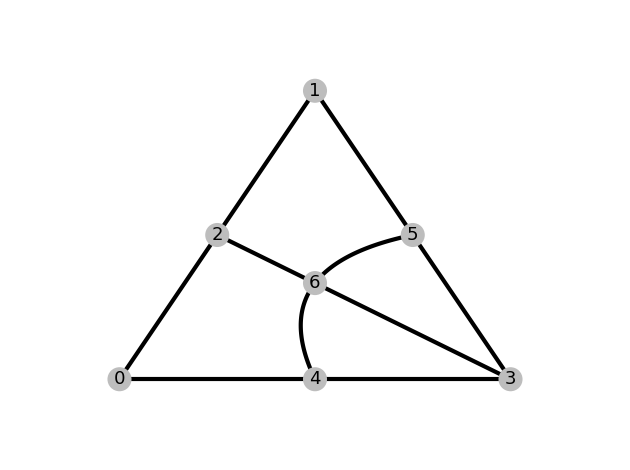} & \raisebox{45 pt}{\parbox{2cm}{$P_7$\\[10 pt] sparse\\paving}} & \raisebox{45 pt}{$\dfrac{12}{7}$} & \raisebox{45 pt}{$\dfrac{12}{7}$} & \raisebox{45 pt}{$\dfrac{12}{7}$} & \raisebox{45 pt}{$\dfrac{12}{7}$}\\
\hline
\raisebox{45 pt}{11} &
\includegraphics[scale=.3]{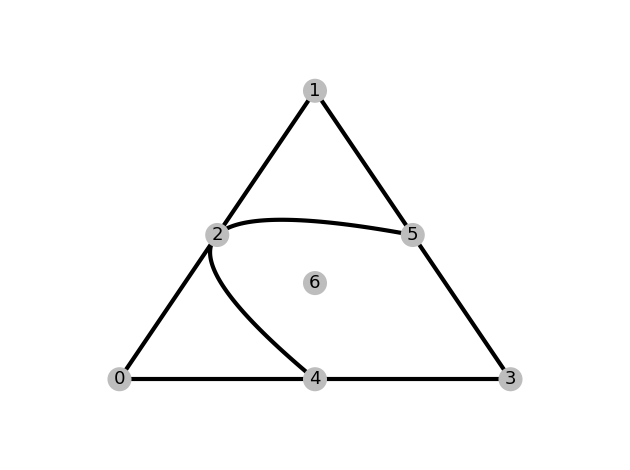} & \raisebox{45 pt}{\parbox{2 cm}{sparse\\paving}} & \raisebox{45 pt}{$\dfrac{15}{8}$} & \raisebox{45 pt}{$\dfrac{15}{8}$} & \raisebox{45 pt}{$\dfrac{12}{7}$} & \raisebox{45 pt}{$2$} \\
\hline
\raisebox{45 pt}{12} &
\includegraphics[scale=.3]{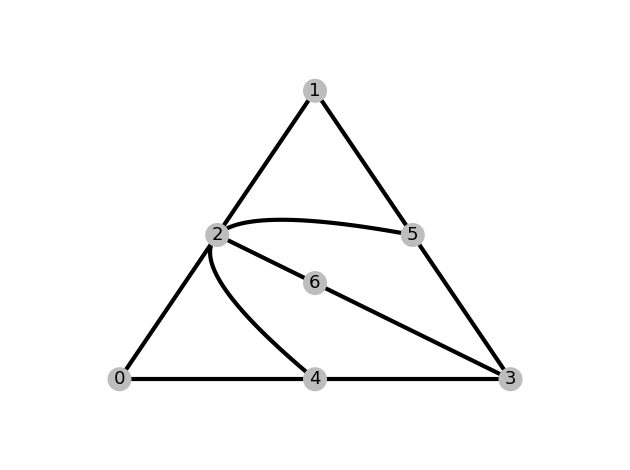} & \raisebox{45 pt}{\parbox{2 cm}{sparse\\paving}} & \raisebox{45 pt}{$\dfrac{24}{13}$} & \raisebox{45 pt}{$\dfrac{24}{13}$} & \raisebox{45 pt}{$\dfrac{12}{7}$} & \raisebox{45 pt}{$\dfrac{28}{15}$}  \\
\hline
\raisebox{45 pt}{13} &
\includegraphics[scale=.3]{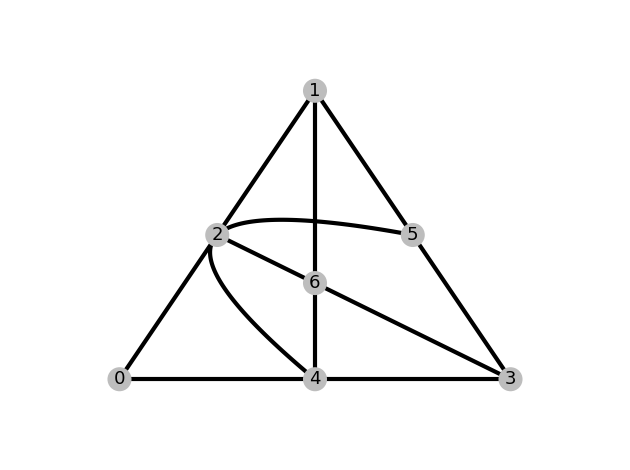} & \raisebox{45 pt}{\parbox{2 cm}{$F_7^-$ \\non-Fano matroid\\[10 pt] sparse\\ paving}} & \raisebox{45 pt}{$\dfrac{9}{5}$} & \raisebox{45 pt}{$\dfrac{9}{5}$} & \raisebox{45 pt}{$\dfrac{12}{7}$} & \raisebox{45 pt}{$\dfrac{20}{11}$} \\
\hline
\raisebox{45 pt}{14} &
\includegraphics[scale=.3]{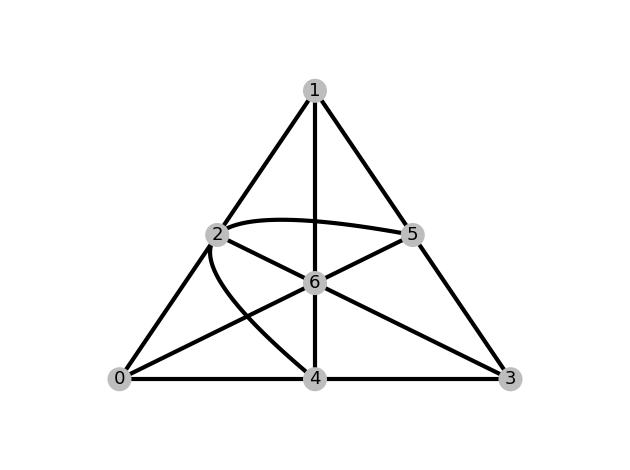} & \raisebox{45 pt}{\parbox{2 cm}{$F_7$ \\ Fano \\ matroid\\[10 pt] sparse\\paving}} & \raisebox{45 pt}{$\dfrac{12}{7}$} & \raisebox{45 pt}{$\dfrac{12}{7}$} & \raisebox{45 pt}{$\dfrac{12}{7}$} & \raisebox{45 pt}{$\dfrac{12}{7}$} \\
\hline
\raisebox{45 pt}{15} &
\includegraphics[scale=.3]{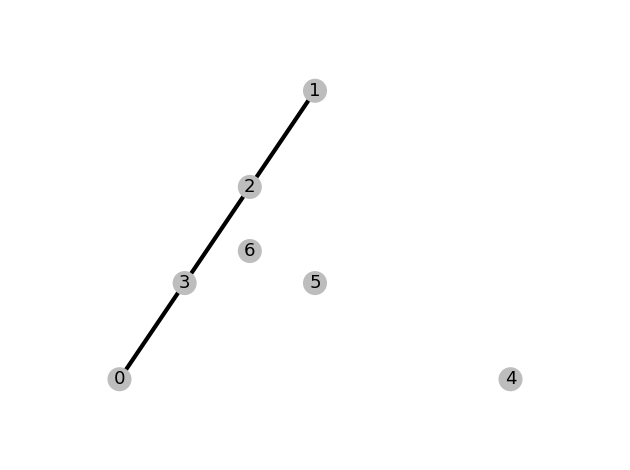} & \raisebox{45 pt}{\parbox{2 cm}{dual not \\ paving}} & \raisebox{45 pt}{$\dfrac{9}{5}$} & \raisebox{45 pt}{$\dfrac{9}{5}$} & \raisebox{45 pt}{$\dfrac{12}{7}$} & \raisebox{45 pt}{$2$}\\
\hline
\raisebox{45 pt}{16}  &
\includegraphics[scale=.3]{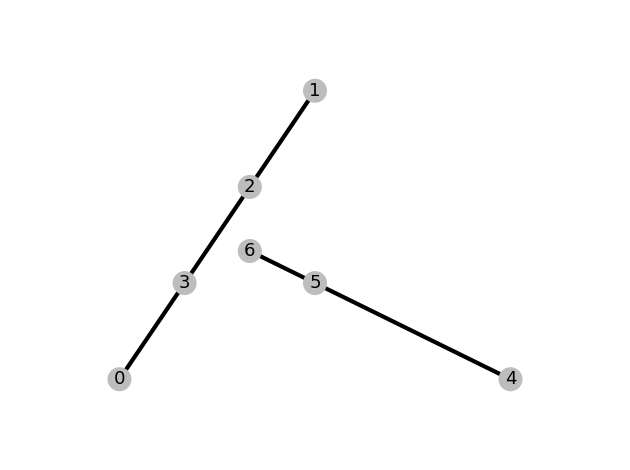} & \raisebox{45 pt}{\parbox{2 cm}{dual not \\ paving}} & \raisebox{45 pt}{$\dfrac{3}{2}$} & \raisebox{45 pt}{$\dfrac{8}{5}$} & \raisebox{45 pt}{$\dfrac{12}{7}$} & \raisebox{45 pt}{$\dfrac{9}{5}$} \\
\hline
\raisebox{45 pt}{17} &
\includegraphics[scale=.3]{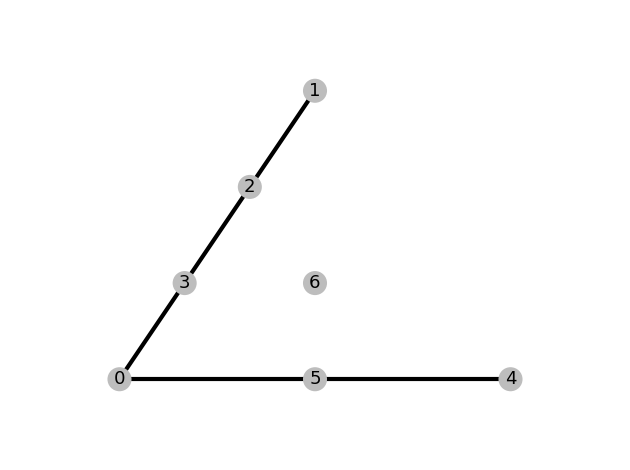} & \raisebox{45 pt}{\parbox{2 cm}{dual not \\ paving}} & \raisebox{45 pt}{$\dfrac{9}{5}$} & \raisebox{45 pt}{$\dfrac{9}{5}$} & \raisebox{45 pt}{$\dfrac{12}{7}$} & \raisebox{45 pt}{$2$} \\
\hline
\raisebox{45 pt}{18} &
\includegraphics[scale=.3]{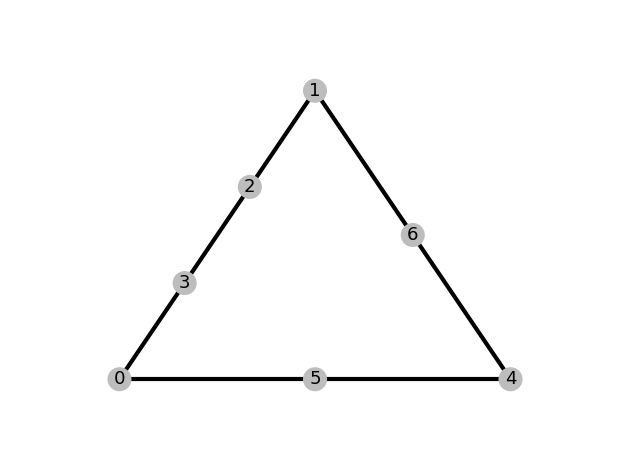} & 
\raisebox{45 pt}{\parbox{2 cm}{dual not \\ paving}} & \raisebox{45 pt}{$\dfrac{9}{5}$} & \raisebox{45 pt}{$\dfrac{9}{5}$} & \raisebox{45 pt}{$\dfrac{12}{7}$} & \raisebox{45 pt}{$\dfrac{24}{13}$} \\
\hline
\raisebox{45 pt}{19} &
\includegraphics[scale=.3]{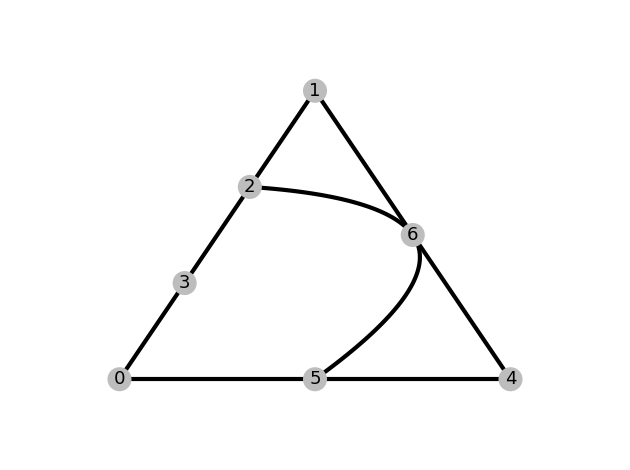} & 
\raisebox{45 pt}{\parbox{2 cm}{$O_7$ \\ dual not \\ paving}}
& \raisebox{45 pt}{$\dfrac{27}{16}$} & \raisebox{45 pt}{$\dfrac{27}{16}$} & \raisebox{45 pt}{$\dfrac{12}{7}$} & \raisebox{45 pt}{$\dfrac{7}{4}$} \\
\hline
\raisebox{45 pt}{20} &
\includegraphics[scale=.3]{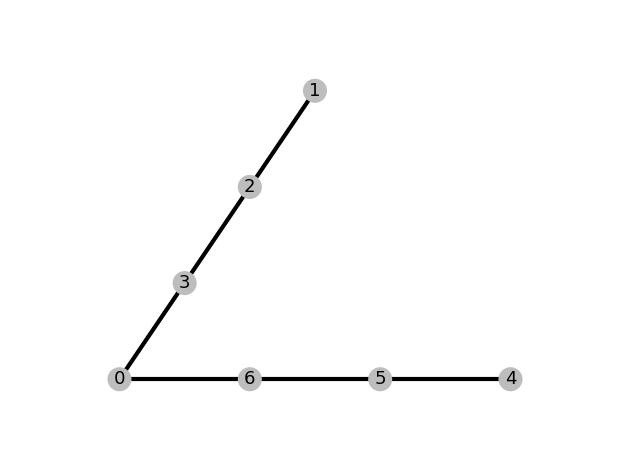} & \raisebox{45 pt}{\parbox{2 cm}{dual not \\ paving}} & \raisebox{45 pt}{$\dfrac{3}{2}$} & \raisebox{45 pt}{$\dfrac{3}{2}$} & \raisebox{45 pt}{$\dfrac{12}{7}$} & \raisebox{45 pt}{$\dfrac{12}{7}$}\\
\hline
\raisebox{45 pt}{21} &
\includegraphics[scale=.3]{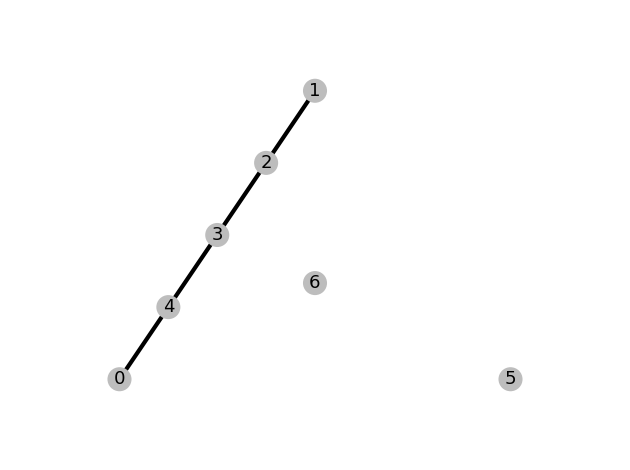} & \raisebox{45 pt}{\parbox{2 cm}{dual not \\ paving}} & \raisebox{45 pt}{$\dfrac{3}{2}$} & \raisebox{45 pt}{$\dfrac{5}{3}$} & \raisebox{45 pt}{$\dfrac{9}{5}$} & \raisebox{45 pt}{$2$} \\
\hline
\raisebox{45 pt}{22} &
\includegraphics[scale=.3]{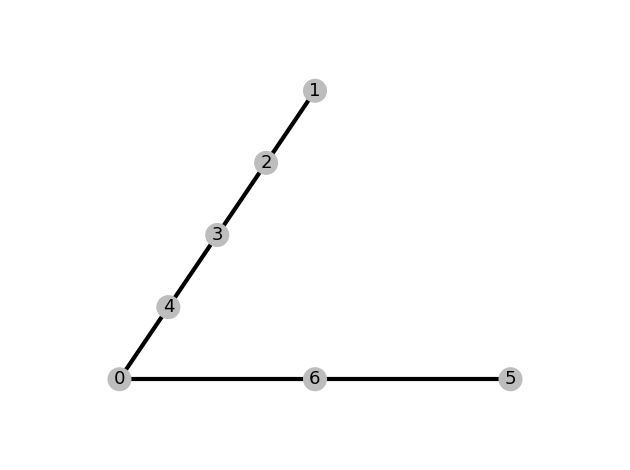} & \raisebox{45 pt}{\parbox{2 cm}{dual not \\ paving}} & \raisebox{45 pt}{$\dfrac{3}{2}$} & \raisebox{45 pt}{$\dfrac{8}{5}$} & \raisebox{45 pt}{$\dfrac{9}{5}$} & \raisebox{45 pt}{$\dfrac{9}{5}$} \\
\hline
\raisebox{45 pt}{23} &
\includegraphics[scale=.3]{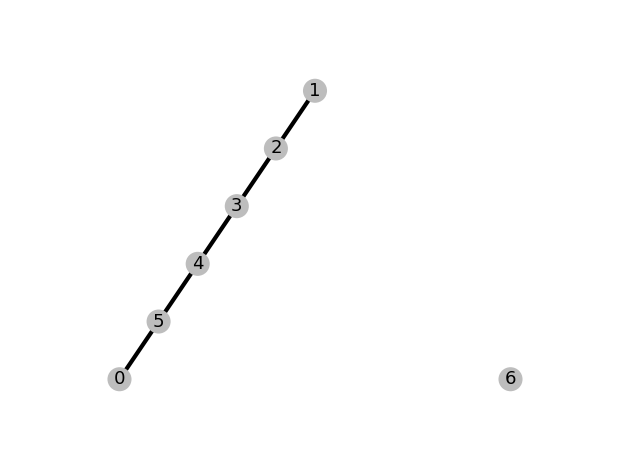} & \raisebox{45 pt}{\parbox{2 cm}{dual not \\ paving}} & \raisebox{45 pt}{$\dfrac{15}{11}$} & \raisebox{45 pt}{$\dfrac{5}{3}$} & \raisebox{45 pt}{$2$} & \raisebox{45 pt}{$2$}\\
\end{longtable}

\end{document}